\documentclass[12pt]{article}
\usepackage{style}
\usepackage{booktabs}
\usepackage{array}

\title{\MakeUppercase{Fixed point Floer cohomology and closed-string mirror symmetry for nodal curves}}
\date{\today}

\author{Maxim Jeffs, Yuan Yao, and Ziwen Zhao}

\begin{document}

\maketitle

\begin{abstract}
We show that for singular hypersurfaces, a version of their genus-zero Gromov-Witten theory may be described in terms of a direct limit of fixed point Floer cohomology groups, a construction which is more amenable to computation and easier to define than the technical foundations of the enumerative geometry of more general singular symplectic spaces. As an illustration, we give a direct proof of closed-string mirror symmetry for nodal curves of genus greater than or equal to $2$, using calculations of (co)product structures on fixed point Floer homology of Dehn twists due to Yao-Zhao \cite{ziwenyao}. 
\end{abstract}

 \tableofcontents

\section{Introduction}\label{sec:intro}

Singular varieties arise frequently in mirror symmetry, even in the simplest cases, such as  mirrors of smooth algebraic curves. While studying the enumerative geometry of such singular varieties intrinsically is certainly possible, outside of the orbifold case it is often technically demanding and not straightforward to compute (see for instance \cite{parker1,parker2,ACGS,chen,jun_li,TZ,TMZ,log_GW} though this list is certainly not exhaustive). An alternative proposal due to Auroux and developed in \cite{jeffs}, uses a more algebraic and categorical approach to defining symplectic invariants of singular hypersurfaces and complete intersections. Passing to the closed-string setting by taking Hochschild cohomology suggests defining new enumerative invariants in terms of direct limits of fixed point Floer cohomology groups of nearby fibers. In the case of smooth algebraic curves of genus greater than or equal to $2$, calculations of the (co)product structures on fixed point Floer homology for Dehn twists have been carried out by \cite{ziwenyao}. Their work allows us to carry out this direct-limit construction explicitly and produce precisely the algebra structures predicted by mirror symmetry. 

Homological mirror symmetry for (smooth) curves has been studied extensively: beginning with \cite{seidel_genus_two} and \cite{efimov_curves}, as well as \cite{heather_thesis}, \cite{pascaleff1,pascaleff2}, and \cite{lekili1,lekili2}. However, it seems that enumerative mirror symmetry for curves has yet to be studied outside of the genus-$1$ case (see \cite{sili,dijkgraaf}).

In \S \ref{sec:hms} we explain the background in mirror symmetry and the closed-string predictions for nodal curves; 
in \S \ref{sec:bmodel} we carry out the calculation of the $B$-model invariants on the mirror side. After setting up our conventions for Lefschetz fibrations in \S \ref{sec:lefschetz}, in \S \ref{sec:seidel}, we define the Seidel class and a twisted closed-open map, and prove the twisted closed-open map takes the Seidel class to the Seidel natural transformation. Finally, in \S \ref{sec:computation}, we calculate the Seidel class and in \S \ref{sec:A_model_calculations} use the results of \cite{ziwenyao} to compute the direct limit along multiplication by this class.

In future work we will study the algebraic structure of fixed point Floer cohomology for symplectomorphisms of algebraic surfaces using similar methods. This opens up avenues for studying fixed points of symplectomorphisms using mirror symmetry, which we will explore further.

\subsection{Statement of results}\label{sec:hms}

If $\Sigma_g$ is a smooth Riemann surface of genus $g \geq 2$, it is known that a mirror can be described as a Landau-Ginzburg model $(X_g,W_g)$ where $X_g$ is a 3-dimensional algebraic variety and $W_g:X_g \to \CC$ is a holomorphic function whose critical locus is a trivalent configuration $Z_g$ of $\PP^1$s and $\AA^1$s \cite{efimov_curves,AAK}. Homological mirror symmetry then predicts that a Fukaya category of $\Sigma_g$ is equivalent to the matrix factorization category $\mathrm{MF}(X_g,W_g)$. After passing to Hochschild cohomology, this should yield the closed-string mirror symmetry equivalence \cite{GKR} with symplectic cohomlogy (with $\CC$ coefficients):
\begin{equation*}
    \mathrm{SH}^k(\Sigma_g) \cong \bigoplus_{i \equiv k \;\mathrm{mod}\;2}\mathrm{HH}^i(\scr{F}(\Sigma_g)) \cong \bigoplus_{i \equiv k \;\mathrm{mod}\;2} \mathrm{HH}^i(\mathrm{MF}(X_g,W_g))
\end{equation*}
that induces an isomorphism of $\ZZ/2$-graded $\CC$-algebras between $\mathrm{SH}^{\ast}(\Sigma_g)$ and $\mathrm{HH}^{\ast}(\mathrm{MF}(X_g,W_g))$.

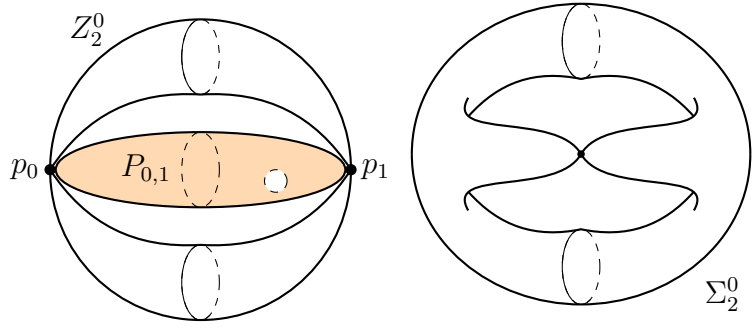
\begin{figure}
\begin{center}
    \begin{tikzpicture}[scale=0.5]
    \begin{scope}
        \draw[thick] (4,2) to[bend left] (0,0) to[bend left] (-4,2) to[out=-90,in=-180] (0,-2) to[out=0, in=-90] (4,2);
        \draw[dashed] (0,-1) ellipse (0.5 and 1);
        \draw (0,0) arc[start angle=90,end angle=300,x radius=0.5,y radius =1];
        \draw (0,3) arc[start angle=90,end angle=300,x radius=0.5,y radius =1];
        \draw (0,6) arc[start angle=90,end angle=300,x radius=0.5,y radius =1];
    \end{scope}
    \begin{scope}[rotate=180,shift={(0,-4)}]
        \draw[thick] (4,2) to[bend left] (0,0) to[bend left] (-4,2) to[out=-90,in=-180] (0,-2) to[out=0, in=-90] (4,2);
        \draw[dashed] (0,-1) ellipse (0.5 and 1);
    \end{scope}
        \filldraw[thick,fill=orange!30] (0,2) ellipse (3.85 and 1);
        \draw[dashed] (0,2) ellipse (0.5 and 1);
        \fill[black] (-4,2) circle (0.15);
        \draw (-4,2) node[anchor=east] {$p_0$};
        \fill[black] (4,2) circle (0.15);
        \draw (4,2) node[anchor=west] {$p_1$};
        \filldraw[dashed,fill=white] (2,1.7) circle (0.3);
        \draw (-0.5,2) node[anchor=east] {$P_{0,1}$};
        \draw (-3,5) node[anchor=south] {$Z_2^0$};
    \end{tikzpicture}
        \begin{tikzpicture}[scale=0.5]
    \begin{scope}
        \draw[thick] (3,1) to[bend left] (0,0) to[bend left] (-3,1);
        \draw[thick] (-4.5,2) to[out=-90,in=-180] (0,-2) to[out=0, in=-90] (4.5,2);
        \draw[dashed] (0,-1) ellipse (0.5 and 1);
        \draw[thick] (3,0.5) to[out=60,in=-60] (0,2);
        \draw[thick] (-3,0.5) to[out=120,in=-120] (0,2);
        \fill[black] (0,2) circle (0.1);
        \draw (0,0) arc[start angle=90,end angle=300,x radius=0.5,y radius =1];
        \draw (0,6) arc[start angle=90,end angle=300,x radius=0.5,y radius =1];
    \end{scope}
    \begin{scope}[rotate=180,shift={(0,-4)}]
        \draw[thick] (3,1) to[bend left] (0,0) to[bend left] (-3,1);
        \draw[thick] (-4.5,2) to[out=-90,in=-180] (0,-2) to[out=0, in=-90] (4.5,2);
        \draw[thick] (3,0.5) to[out=60,in=-60] (0,2);
        \draw[thick] (-3,0.5) to[out=120,in=-120] (0,2);
        \draw[dashed] (0,-1) ellipse (0.5 and 1);
        \draw (-3,5) node[anchor=north west] {$\Sigma_{2}^0$};
    \end{scope}
    \end{tikzpicture}
\end{center}
\caption{The mirror $Z_2$ of a nodal genus-$2$ curve $\Sigma_{2}^0$.}\label{fig:once-punctured-mirror}
\end{figure}

One expects that a nodal degeneration of the curve $\Sigma_g$, which we denote by $\Sigma_{g}^0$, is mirror to removing one smooth point from the critical locus of $W_g$; in the sense that we should have an equivalence of categories between $\scr{F}(\Sigma_{g}^0)$, the Fukaya category of the nodal curve $\Sigma_{g}^0$ defined as in \cite{jeffs}, with the matrix factorization category $\mathrm{MF}(X^{0}_g, W^{0}_g)$, where the critical locus $Z_g^0$ of $W^0_g$ should be the complement of a single smooth point in the critical locus $Z_g$ of $W_g$. The $A$-model invariant of $\Sigma_g^0$ we consider is:

\begin{definition}
    The \textbf{symplectic cohomology} of the nodal curve $\Sigma_{g}^0$ is given by the direct limit
    \begin{equation*}
        \mathrm{SH}^{\ast}(\Sigma_{g}^0) = \dlim_d \mathrm{HF}^{\ast}(\Sigma_g,\phi^d)
    \end{equation*}
    where $\mathrm{HF}^{\ast}(\Sigma_g, \phi^d)$ is the fixed point Floer cohomology (in the sense of \cite{ziwenyao}) of the composition of the Dehn twists around the vanishing cycles, and the direct limit is taken along multiplication by the Seidel class $S \in \mathrm{HF}^0(\Sigma_g,\phi)$ (defined in \S \ref{sec:seidel} below).
\end{definition}

The ring structure on symplectic cohomology $\mathrm{SH}^{\ast}(\Sigma_{g}^0)$ is induced by the product structure
    \begin{equation*}
        \mathrm{HF}^{\ast}(\Sigma_g,\phi^i) \otimes \mathrm{HF}^{\ast}(\Sigma_g,\phi^j) \to \mathrm{HF}^{\ast}(\Sigma_g,\phi^{i+j})
    \end{equation*}
on fixed point Floer (co)homology as in \cite{ziwenyao}.

The justification for this definition will be given in \S \ref{sec:seidel}, where we prove Theorem \ref{thm:HF_equals_QH}, showing that our symplectic cohomology algebra does indeed compute the Hochschild cohomology of the Fukaya category of the nodal curve:

\begin{theorem*} Suppose $M$ is a non-degenerate Liouville manifold in the sense of \cite[Definition 1.1]{ganatra_thesis},  then the twisted closed-open map $\scr{CO}_{\phi}$ is an isomorphism and there is an equivalence of graded algebras:
    \begin{equation*}
    \mathrm{HH}^{\ast}(\scr{F}(M^0)) \cong \dlim_d \mathrm{HF}^{\ast}(\phi^d)
\end{equation*}
where the connecting maps in the direct limit are given by multiplication by the Seidel class $S$ in $\mathrm{HF}^0(\phi)$.
\end{theorem*}

\begin{remark}
    As the computations in Theorem \ref{thm:ring_structure} and \ref{thm:HF^1 for nodal curves} show, we can think of the symplectic cohomology $\mathrm{SH}^*(\Sigma_g^0)$ as the quantum deformation of the singular cohomology of $\Sigma_g^0$ with the node removed. As this result falls out of a direct computation, it is not clear in what level of generality this result is expected to hold. Our definition of symplectic cohomology applies to singular hypersurfaces of any dimension, though it is not clear whether we can still view our definition of symplectic cohomology as a deformation of singular cohomology (of the complement of the singular locus). For this reason, we have described our $A$-model invariant as `symplectic cohomology' rather than the quantum cohomology of $\Sigma_g^0$, even though the curve $\Sigma_g$ may be closed.
\end{remark}

We are certainly not claiming that this is the same as other notions of Gromov-Witten theory that may be defined for nodal curves; and this is certainly not the only way the symplectic cohomology of such a curve could be defined. However, attempting to compute the quantum cohomology of a nodal curve by na\"ive counts of actual holomorphic spheres inside the nodal curve will only yield trivial results.

The invariants on the $B$-side that we shall consider are given by taking the cohomology of

\begin{definition}
    The sheaf $\tilde{T}_{Z}$ of \textbf{balanced vector fields} on a trivalent configuration $Z$ of $\AA^1$s and $\PP^1$s is the sheaf whose sections are vector fields on $Z$ (vanishing at the nodes) whose rotation numbers around every node sum to zero.
\end{definition}

The justification for this definition will be given in \S \ref{sec:bmodel} where we prove Theorem \ref{thm:b_model}, showing that the cohomology of the sheaves of balanced vector fields computes the Hochschild cohomology of the matrix factorization category: 

\begin{theorem*}
Let ${Z_{g}^0}$ be the critical locus of the LG model $(X^0_g, W^0_g)$ mirror to the nodal curve $\Sigma_{g}^0$; then
\begin{align*}
\mathrm{HH}^{\mathrm{even}}(\mathrm{MF}(X^0_g,W^0_g)) & \cong  H^0(\scr{O}_{Z_{g}^0}) \osum H^1(\tilde{T}_{Z_{g}^0}) \\
    \mathrm{HH}^{\mathrm{odd}}(\mathrm{MF}(X^0_g,W^0_g)) & \cong H^1(\scr{O}_{Z_{g}^0}) \osum H^0(\tilde{T}_{Z_{g}^0})
\end{align*}
as $\CC$-algebras and modules respectively.
\end{theorem*}

\begin{remark}
    One could say that $\tilde{T}_{Z_g^0}$ represents the vector fields corrected by the sheaf of vanishing cycles on $\mathrm{crit}(W_g)$ as in \cite{GKR}.
\end{remark}

With these definitions our main theorem is:

\begin{theorem}(\textbf{Closed-string mirror symmetry for nodal curves})  \label{thm:main_theorem}
There is an equivalence of graded algebras
\begin{equation*}
     \mathrm{SH}^{i}(\Sigma_{g}^0) \cong \bigoplus_{j+k \cong i \; \mathrm{mod} \; 2} H^j(Z_{g}^0, \wedge^k \tilde{T}_{Z_{g}^0})
\end{equation*}
Moreover, the same result holds if $\Sigma_{g}^0$ has several nodes whose vanishing cycles are disjoint homologically linearly independent closed curves.
\end{theorem}

\begin{remark}[Multiple Dehn twists on $\Sigma_g$] \label{remark: topological assumption on Dehn twists}
Let $\Sigma_g$ be the closed Riemann surface with genus $g\ge 2$ (respectively $\Sigma_{g,k}$ the $k$-th punctured Riemann surface). Whenever we talk about performing multiple Dehn twists on $\Sigma_g$ (resp. $\Sigma_{g,k}$) along different circles, we always assume the following. Let $C_1,\cdots,C_\ell$ denote the embedded closed curves along which we are performing the Dehn twists: we assume that they are disjoint and homologically linearly independent. Equivalently, this means that $C_1\cup C_2\cup\cdots \cup C_\ell$ is non-separating.
\end{remark}

The proof of this theorem involves directly computing both the symplectic cohomology via a direct limit of fixed point Floer cohomology groups (Theorem \ref{thm:ring_structure}), and the cohomology of sheaves of balanced vector fields on the mirror (Theorem \ref{thm:b_model}) to be isomorphic to the same graded algebra.

A very similar calculation can be performed for nodal curves with punctures. Let $\Sigma_{g,k}$ denote $\Sigma_g$ with finitely many punctures $\{p_1, p_2, \cdots,p_k\}$ and disjoint homologically linearly independent vanishing cycles $C_1, \cdots, C_{\ell}$, and let $\Sigma_{g,k}^0$ denote the corresponding punctured nodal curve. The punctured curve $\Sigma_{g,k}^0$ is then mirror to a Landau-Ginzburg model $(X_{g}\dash, W_{g}\dash)$ whose critical locus $Z_{g}\dash$ is a degeneration of ${Z_{g}^0}$ that has one additional trivalent node for each puncture (see Figure \ref{fig:singular}). Then we have

\begin{figure}
\begin{center}
    \begin{tikzpicture}[scale=0.5]
    \begin{scope}
        \draw[thick] (4,2) to[out=220,in=30] (0,-1) to[out=160,in=-40] (-4,2) to[out=-120,in=-120] (0,-1) to[out=-60, in=-70] (4,2);
        \draw[dashed] (-2,-0.5) ellipse (0.3 and 0.65);
        \draw[dashed] (2,-0.5) ellipse (0.3 and 0.73);
        \draw (0,3) arc[start angle=90,end angle=300,x radius=0.5,y radius =1];
        \draw (0,6) arc[start angle=90,end angle=300,x radius=0.5,y radius =1];
        \fill[black] (0,-1) circle (0.15);
        \draw (0,-3) node {$P_{\infty}$};
    \end{scope}
    \begin{scope}[rotate=180,shift={(0,-4)}]
        \draw[thick] (4,2) to[bend left] (0,0) to[bend left] (-4,2) to[out=-90,in=-180] (0,-2) to[out=0, in=-90] (4,2);
        \draw[dashed] (0,-1) ellipse (0.5 and 1);
    \end{scope}
        \filldraw[thick,fill=orange!30] (0,2) ellipse (3.85 and 1);
        \draw[dashed] (0,2) ellipse (0.5 and 1);
        \fill[black] (-4,2) circle (0.15);
        \draw (-4,2) node[anchor=east] {$p_0$};
        \fill[black] (4,2) circle (0.15);
        \draw (4,2) node[anchor=west] {$p_1$};
        \filldraw[dashed,fill=white] (2,1.7) circle (0.3);
        \draw (-0.5,2) node[anchor=east] {$P_{0,1}$};
        \draw (-3,5) node[anchor=south] {$Z_2^{\prime}$};
        \draw[thick] (1.2,-4) -- (0,-1) -- (-1.2,-4);
        \draw[dashed] (0,-4) ellipse (1.2 and 0.5);
    \end{tikzpicture}
        \begin{tikzpicture}[scale=0.6]
    \begin{scope}
        \draw[thick] (3,1) to[bend left] (0,0) to[bend left] (-3,1);
        \draw[thick] (-4.5,2) to[out=-90,in=-180] (0,-2) to[out=0, in=-90] (4.5,2);
        \draw[dashed] (0,-1) ellipse (0.5 and 1);
        \draw[thick] (3,0.5) to[out=60,in=-60] (0,2);
        \draw[thick] (-3,0.5) to[out=120,in=-120] (0,2);
        \fill[black] (0,2) circle (0.1);
        \draw (0,0) arc[start angle=90,end angle=300,x radius=0.5,y radius =1];
        \draw (0,6) arc[start angle=90,end angle=300,x radius=0.5,y radius =1];
        \filldraw[dashed,fill=white] (1.7,-1.1) circle (0.3);
    \end{scope}
    \begin{scope}[rotate=180,shift={(0,-4)}]
        \draw[thick] (3,1) to[bend left] (0,0) to[bend left] (-3,1);
        \draw[thick] (-4.5,2) to[out=-90,in=-180] (0,-2) to[out=0, in=-90] (4.5,2);
        \draw[thick] (3,0.5) to[out=60,in=-60] (0,2);
        \draw[thick] (-3,0.5) to[out=120,in=-120] (0,2);
        \draw[dashed] (0,-1) ellipse (0.5 and 1);
        \draw (-3,5) node[anchor=north west] {$\Sigma_{2}^0$};
    \end{scope}
    \end{tikzpicture}
\end{center}
\caption{The mirror $Z_2^{\prime}$ of a punctured nodal genus-$2$ curve $\Sigma_{2,1}^0$.}\label{fig:singular}
\end{figure}
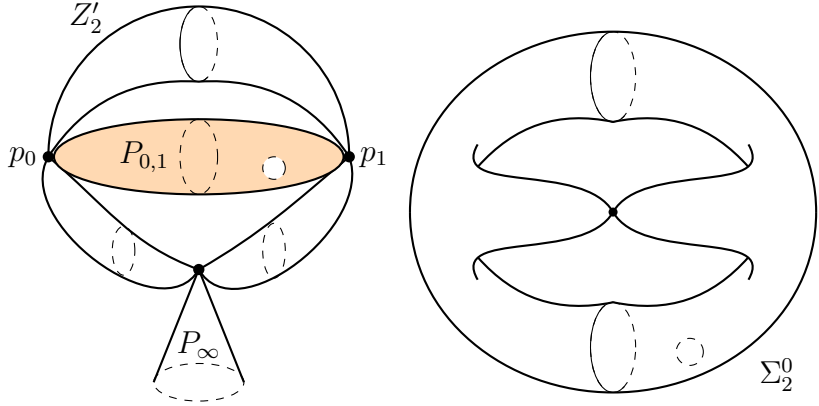

\begin{theorem} (\textbf{Closed string mirror symmetry for punctured nodal curves}) \label{thm:mirror symmetry statement for punctured nodal curves}
There is an isomorphism of $\ZZ/2$-graded algebras: 
\[
\mathrm{SH}^i(\Sigma_{g,k}^0) \cong \bigoplus_{j+\ell \cong i \; \mathrm{mod}\; 2} H^j({Z_{g}^{\prime}}, \wedge^\ell \tilde{T}_{Z_{g}^{\prime}})
\]
\end{theorem}

The reader may find the following table helpful for keeping track of notation (under the assumption that $\Sigma_{g,k}$ has disjoint homologically linearly independent vanishing cycles).
\begin{figure}[H]
\begin{center}
\begin{tabular}{ c | c | c }
\toprule
 $A$-\textbf{side} & $B$-\textbf{side} & \textbf{critical locus} \\ \midrule
 smooth compact curve $\Sigma_g$ & LG model $(X_g, W_g)$ & $Z_g$ trivalent configuration of $(3g-3)$-many $\PP^1$s \\ 
 Dehn twist $\phi$ on $\Sigma_g$ along $C$ & -- & line bundle $\scr{L}$ on $Z_g$, degree-$1$ on one $\PP^1$ \\
 $\ell$-nodal compact curve $\Sigma_{g}^0$ & LG model $(X_g^0, W_g^0)$  & $\ell$-punctured trivalent configuration $Z_{g}^0$ (Fig. \ref{fig:once-punctured-mirror}) \\
nodal $k$-punctured curve $\Sigma_{g,k}^0$ & LG model $(X_g\dash, W_g\dash)$ & punctured curve $Z_g\dash$ with $k$-many $\AA^1$s (Fig. \ref{fig:singular})\\
\bottomrule
\end{tabular}
\end{center}
\end{figure}

\begin{remark}
All of our results above apply only in the case $g \geq 2$. In the case where $g=1$, Theorem \ref{thm:main_theorem} above follows more or less directly from Theorem \ref{thm:HF_equals_QH} combined with Theorem 2 from \cite{jeffs}.
\end{remark}

\subsection{Homogeneous coordinate rings}

If $\phi$ is the monodromy around a large complex structure limit point, one expects that it corresponds under an appropriate homological mirror symmetry equivalence to tensoring by an ample line bundle $\scr{L}$ on the mirror \cite{jeffs}. In our case, a Dehn twist on $\Sigma_g$ is mirror to tensoring by a degree-$1$ line bundle $\scr{L}$ on one $\PP^1$-component of ${{Z_{g}}}=\mathrm{crit}(W_g)$. This motivates the following theorem:

\begin{theorem}(\textbf{Mirror symmetry for homogeneous coordinate rings})\label{thm:homog}
There is an equivalence
\begin{equation*}
    \bigoplus_{d=0}^{\infty} \mathrm{HF}^k(\Sigma_g, \phi^d)
\cong 
\bigoplus_{d=0}^{\infty}\bigosum_{i+j \cong k \; \text{mod} \; 2} H^{i}(\wedge^{j} T_{{{Z_{g}}}} \otimes \scr{L}^{\otimes d})
\end{equation*}
of graded algebras (for $k=0$) and graded modules (when $k=1$) where the grading is given by the order $d$. Similarly, by starting with $\phi^2$,
\begin{equation*}
    \bigoplus_{d=0}^{\infty} \mathrm{HF}^k(\Sigma_g, \phi^{2d})
\cong 
\bigoplus_{d=0}^{\infty}\bigosum_{i+j \cong k \; \text{mod} \; 2} H^{i}(\wedge^{j} T_{{{Z_{g}}}} \otimes \scr{L}^{\otimes {2d}})
\end{equation*}
as graded algebras/modules.
\end{theorem}

There are analogous results also for $\phi^k$ for $k \geq 3$. The rings that arise in this case are also intrinsically interesting, as they give rise to different embeddings of the nodal elliptic curve into weighted projective spaces. 

In \S \ref{sec:bmodel} we justify these $B$-side invariants as the result of an expected twisted HKR theorem for matrix factorization categories. We then calculate these homogeneous coordinate rings on the $B$-side; in \S \ref{sec:product} and \S \ref{sec:phi_squared_homogenous} we carry out the calculation on the $A$-side and verify that we obtain the same algebra.

\subsubsection*{Total spaces of line bundles}

There is an alternative way to formulate these results: let $f: E \to \CC^{\ast}$ be the $\phi$-twisted $\Sigma_{g,k}$-bundle over $\CC^{\ast}$. When $\Sigma_{g,k}$ is an open surface (i.e. $k \geq 1$), we can make $E$ into a Liouville domain and define $\mathrm{SH}^{\ast}(E,f)$, the symplectic cohomology (in the sense of \cite{GPS1}) of the Liouville sector obtained from $(E,f)$ by placing a stop at $f\inv(-\infty)$. If one assumes that a twisted version of the K\"unneth theorem of \cite{GPS2} to hold, we have
\begin{equation*}
    \mathrm{SH}^0(E,f) \cong \bigoplus_{d=0}^{\infty} \mathrm{HF}^0(\Sigma_{g,k},\phi^{d})
\end{equation*}
Therefore, from Theorem \ref{thm:homog} we may deduce:
\begin{theorem}\label{thm:total_spaces}
If a twisted K\"unneth theorem for symplectic cohomology is assumed to hold, then there is an equivalence of algebras
\begin{equation*}
    \mathrm{SH}^0(E,f)  \cong H^0(L, \scr{O}_L)
\end{equation*}
    where $L$ is the total space of the line bundle $\scr{L}$ over ${{Z_{g}}}$.
\end{theorem}
This we may regard also as a form of mirror symmetry. If $\Sigma_{g,k}$ can be realized as a very affine hypersurface $H \subset (\CC^{\ast})^2$ and $\Sigma^{0}_{g,k}$ is its large complex structure limit induced by a choice of subdivision of the Newton polytope (as in \cite{AAK}), then it is proved in \cite[Proposition 3.2.2]{gammage_jeffs} that the twisted $H$-bundle $E$ is homologically mirror to the total space of the canonical line bundle $K_{X}$. The proof of Theorem \ref{thm:total_spaces} is analogous to the proof given there, assuming Theorem \ref{thm:homog}.

\subsection*{Notation and Conventions}

The \textit{Fukaya category} $\scr{F}(M)$ refers to the split-closure of the $A_{\infty}$-category of $\ZZ/2$-graded twisted complexes over the Fukaya category, (partially) wrapped if appropriate (as in \cite{GPS1}). Grading and sign conventions are as in \cite{ganatra_thesis}. The \textit{matrix factorization category} $\mathrm{MF}$ is the $\ZZ/2$-graded dg-derived category of coherent matrix factorizations. We write $\mathrm{hom}$ for morphism complexes in a dg/$A_{\infty}$-category, and $\mathrm{Hom}$ for their cohomology. We write $\scr{C}-\mathrm{bimod}$ for the category of $\scr{C}-\scr{C}$ bimodules, and $\scr{Y}^{\ell}, \scr{Y}^{r}$ for left and right Yoneda modules respectively. All coefficient rings are $\CC$ unless otherwise stated; all algebras are $\ZZ/2$-graded.

\subsection*{Acknowledgements}

 We would like to thank Kai Xu for asking us questions about fixed point Floer homology that led to many fruitful ideas; we would also like to thank Shaoyun Bai for telling us about his forthcoming work with Paul Seidel \cite{shaoyun}. MJ would also like to thank Sheel Ganatra and Xujia Chen for helpful conversations, as well as his advisor Denis Auroux for his invaluable support and guidance. MJ was partially supported by the Rutherford Foundation of the Royal Society of New Zealand, NSF grants DMS-1937869 and DMS-2202984, and by Simons Foundation grant \#385573.

\clearpage

\section{$B$-model calculation}\label{sec:bmodel}

The critical locus $\mathrm{crit}(W_g)$ is a trivalent configuration of $\PP^1$s, with $(2g-2)$ nodes, call them $p_i$, and $(3g-3)$ irreducible $\PP^1$-components. When $\Sigma_{g}^0$ has a single node, the mirror $(X^0_g, W^0_g)$ to $\Sigma_{g}^0$ has critical locus ${Z_{g}^0} = \mathrm{crit}(W^0_g)$ which is the critical locus of $W_g$ punctured at a single point: say this puncture occurs on component $P_{0,1}$ (containing $p_0,p_1$) without loss of generality. 

\begin{figure}[H]
\begin{center}
\begin{tikzpicture}[scale=0.4]
    \begin{scope}
        \draw[thick] (4,2) to[bend left] (0,0) to[bend left] (-4,2) to[out=-90,in=-180] (0,-2) to[out=0, in=-90] (4,2);
        \draw[dashed] (0,-1) ellipse (0.5 and 1);
        \draw (0,0) arc[start angle=90,end angle=300,x radius=0.5,y radius =1];
        \draw (0,6) arc[start angle=90,end angle=300,x radius=0.5,y radius =1];
    \end{scope}
    \begin{scope}[rotate=180,shift={(0,-4)}]
        \draw[thick] (4,2) to[bend left] (0,0) to[bend left] (-4,2) to[out=-90,in=-180] (0,-2) to[out=0, in=-90] (4,2);
        \draw[dashed] (0,-1) ellipse (0.5 and 1);
    \end{scope}
    \begin{scope}[scale=1.5,rotate=90,shift={(-2.7,-4.7)}]
        \draw[thick] (4,2) to[bend left] (0,0) to[bend left] (-4,2) to[out=-90,in=-180] (0,-2) to[out=0, in=-90] (4,2);
        \draw[dashed] (0,-1) ellipse (0.5 and 1);
    \end{scope}
    \begin{scope}[rotate=180]
            \begin{scope}[scale=1.5,rotate=90,shift={(2.7,-4.7)}]
        \draw[thick] (4,2) to[bend left] (0,0) to[bend left] (-4,2) to[out=-90,in=-180] (0,-2) to[out=0, in=-90] (4,2);
        \draw[dashed] (0,-1) ellipse (0.5 and 1);
    \end{scope}
    \end{scope}
    \begin{scope}[shift={(0,-12)}]
    \begin{scope}[rotate=180,shift={(0,-4)}]
        \filldraw[thick,fill=orange!30] (4,2) to[bend left] (0,0) to[bend left] (-4,2) to[out=-90,in=-180] (0,-2) to[out=0, in=-90] (4,2);
        \draw[dashed] (0,-1) ellipse (0.5 and 1);
    \end{scope}
        \begin{scope}
        \draw[thick] (4,2) to[bend left] (0,0) to[bend left] (-4,2) to[out=-90,in=-180] (0,-2) to[out=0, in=-90] (4,2);
        \draw[dashed] (0,-1) ellipse (0.5 and 1);
        \draw (0,0) arc[start angle=90,end angle=300,x radius=0.5,y radius =1];
        \draw (0,6) arc[start angle=90,end angle=300,x radius=0.5,y radius =1];
        \filldraw[dashed,fill=white] (2,4.5) circle (0.4);
         \draw (-2.5,4.5) node[anchor=west] {$P_{0,1}$};
    \end{scope}
            \fill[black] (-4,2) circle (0.15);
                \draw (-4,2) node[anchor=west] {$p_0$};
        \fill[black] (4,2) circle (0.15);
        \draw (4,2) node[anchor=east] {$p_1$};
    \end{scope}
        \fill[black] (-4,2) circle (0.15);
        \draw (-4,2) node[anchor=west] {$p_2$};
        \fill[black] (4,2) circle (0.15);
        \draw (4,2) node[anchor=east] {$p_3$};
        \draw (-4,-4) node[anchor=east] {$Z_3^0$};
    \end{tikzpicture}
    \hspace{20pt}
            \begin{tikzpicture}[scale=0.5]
    \begin{scope}
        \draw[thick] (3,1) to[bend left] (0,0) to[bend left] (-3,1);
        \draw[thick] (-4.5,2) to[out=-90,in=-180] (0,-2) to[out=0, in=-90] (4.5,2);
        \draw[dashed] (0,-1) ellipse (0.5 and 1);
        \draw[thick] (3,0.5) to[out=60,in=-60] (0,2);
        \draw[thick] (-3,0.5) to[out=120,in=-120] (0,2);
        \fill[black] (0,2) circle (0.1);
        \draw (0,0) arc[start angle=90,end angle=300,x radius=0.5,y radius =1];
        \draw (0,6) arc[start angle=90,end angle=300,x radius=0.5,y radius =1];
    \end{scope}
    \begin{scope}[rotate=180,shift={(0,-4)}]
        \draw[thick] (3,0.5) to[out=60,in=-60] (0,2);
        \draw[thick] (-3,0.5) to[out=120,in=-120] (0,2);
    \end{scope}
    \begin{scope}[rotate=180,shift={(0,-8)}]
        \draw[thick] (2.7,1.6) to[out=-90,in=0] (0,0) to[out=180,in=-90] (-2.7,1.6);
        \draw[thick] (-4.5,2) to[out=-90,in=-180] (0,-2) to[out=0, in=-90] (4.5,2);
        \draw[thick] (2.6,1) to[out=120,in=0] (0,2);
        \draw[thick] (-2.6,1) to[out=60,in=-180] (0,2);
        \draw[dashed] (0,-1) ellipse (0.5 and 1);
    \end{scope}
        \draw[thick] (4.5,2) -- (4.5,6);
        \draw[thick] (-4.5,2) -- (-4.5,6);
        \draw[thick] (-3,3) to[out=60,in=0] (0,4);
        \draw[thick] (3,3) to[out=120,in=180] (0,4);
        \draw[dashed] (0,5) ellipse (0.5 and 1);
        \draw (0,10) arc[start angle=90,end angle=300,x radius=0.5,y radius =1];
        \draw (-3,-3) node[anchor=east] {$\Sigma_{3}^0$};
    \end{tikzpicture}
    \end{center}
    \caption{The mirror $Z_3$ of a nodal genus-$3$ curve $\Sigma_{3}^0$}
\end{figure}
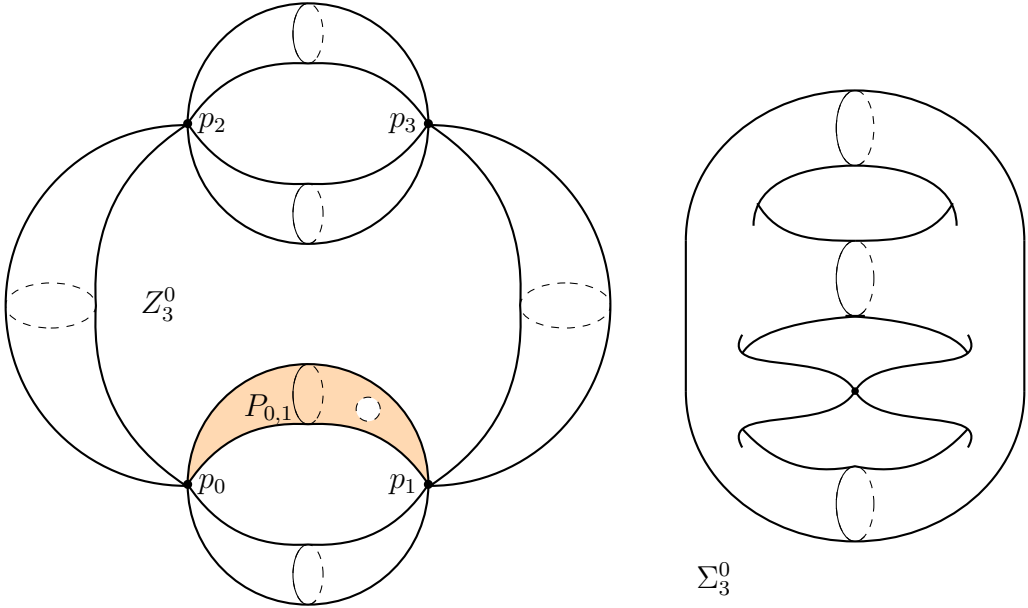

\subsection{Hochschild-Kostant-Rosenberg theorem and balanced vector fields}

In this subsection, we justify our use of balanced vector fields as our $B$-model invariants with the following theorem:

\begin{theorem}
    Let ${Z_{g}^0}$ be the critical locus of the LG model $(X^0_g, W^0_g)$ mirror to the nodal curve $\Sigma_{g}^0$; then
\begin{align*}
\mathrm{HH}^{\mathrm{even}}(\mathrm{MF}(X^0_g,W^0_g)) & \cong  H^0(\scr{O}_{Z_{g}^0}) \osum H^1(\tilde{T}_{Z_{g}^0})  \\
    \mathrm{HH}^{\mathrm{odd}}(\mathrm{MF}(X^0_g,W^0_g)) & \cong H^1(\scr{O}_{Z_{g}^0}) \osum H^0(\tilde{T}_{Z_{g}^0})
\end{align*}
as $\CC$-algebras and as modules respectively.
\end{theorem}
\begin{proof}
By the Hochschild-Kostant-Rosenberg theorem for global matrix factorization categories \cite[Theorem 3.1]{lin_pomerleano}, we can compute Hochschild cohomology of the matrix factorization category $\mathrm{MF}(X_g,W_g)$ as the hypercohomology of the complex $(\bigwedge^{\ast} T_{X_g}, \iota_{\dd W})$ of sheaves of polyvector fields with differential twisted by $\dd W$. This we can compute by taking a Zariski open cover of a neighbourhood of the critical locus given by open sets $U_p$ whose intersection with $\mathrm{crit}(W_g)$ is the complement of all components of ${Z_{g}}$ not adjacent to node $p$. In each of these open sets $U_p$, the LG model takes the form $(\CC^3, x y z)$ \cite{heather_thesis}. A simple calculation shows that the cohomology of the complex $(\bigwedge^{\ast} T_{\CC^3}, \iota_{\dd (xyz)})$ is given by:
\begin{equation*}
    H^i\left(\wedge^{\ast} T_{\CC^3}, \iota_{\dd (xyz)}\right) \cong \left\{\begin{array}{lll}\CC[x,y,z]/(xy,yz,xz) & \mathrm{for} & i=0 \\ \CC[x,y,z]/(xy,yz,xz)\langle x \partial_x -  y \partial_y, y \partial_y - z \partial_z \rangle & \mathrm{for} & i=1 \\
    0 & \mathrm{for} & i\geq 2
    \end{array}\right.
\end{equation*}
Thus the local regular functions are unchanged, while the vector fields are required to be `balanced' in the sense that the rotation numbers around the nodes sum to zero. Therefore the hypercohomology of this sheaf $(\bigwedge^{\ast} T_{X_g}, \iota_{\dd W})$ is the same as the cohomology of the sheaf of balanced vector fields $\tilde{T}_{Z_g}$ on $Z_g$. An entirely analogous argument applies when $Z_{g}^0$ has punctures or additional $\AA^1$-components. 
\end{proof}

\subsection{Sheaf cohomology of balanced vector fields}

In this section we shall calculate the algebraic structure of the $B$-model invariants directly:

\begin{theorem}\label{thm:b_model}
Let ${Z_{g}^0}$ be the critical locus of the LG model $(X^0_g, W^0_g)$ mirror to the nodal curve $\Sigma_{g}^0$; then
\begin{align*}
H^0(\scr{O}_{Z_{g}^0}) \osum H^1(\tilde{T}_{Z_{g}^0})  & \cong A \\
 H^1(\scr{O}_{Z_{g}^0}) \osum H^0(\tilde{T}_{Z_{g}^0}) & \cong A \osum \CC^{2g-2}
\end{align*}
as $\CC$-algebras and as $A$-modules respectively, where $A$ is the $\CC$-algebra $A = \CC[Y,Z]/(YZ = Y^3+Z^2)$.
\end{theorem}
\begin{proof}
To compute this sheaf cohomology, we can take a Zariski open cover of ${Z_{g}}$ given by $U_p$, the complement of all components of ${Z_{g}}$ not adjacent to node $p$; for $i \neq 0,1$ this is the affine scheme $\set{xy=yz=xz=0} \subset \CC^3$. Except for $i,j=0,1$, every two-fold intersection $U_{p_i} \cap U_{p_j}$ is a disjoint union of three $\CC^{\ast}$s. For $i,j=0,1$ we have
\begin{align*}
  &\scr{O}_{Z_{g}}(U_{p_0}) \cong \CC[x,y,z,(1+x)\inv]/(xy,yz,xz) \\ & \scr{O}_{Z_{g}}(U_{p_1}) \cong \CC[x\dash,y\dash,z\dash,(1+x\dash)\inv]/(x\dash y\dash,y\dash z\dash ,x\dash z\dash) \\
  &\scr{O}_{Z_{g}}(U_{p_0} \cap U_{p_1}) \cong \CC[x^\pm,(1+x)\inv]\osum \CC[y^\pm]\osum \CC[z^\pm]
\end{align*}
and the restriction map takes $x,y,z \mapsto x,y,z$ and $x\dash,y\dash,z\dash \mapsto 1/x,1/y,1/z$ respectively. Here we identify $P_{0,1} \cong \PP^1\setminus\set{-1}$ with local coordinates $x$ near $p_0 = 0$ and $x\dash = 1/x$ near $p_1 = \infty$. All three-fold intersections $U_i \cap U_j \cap U_k$ are empty for $i<j<k$ so there are no $H^2$ contributions.

Regardless of the genus, $H^0(\scr{O}_{Z_{g}^0})$ is always equivalent to the algebra $A = \CC[Y,Z]/(YZ - Z^2- Y^3)$ of functions on a nodal affine cubic curve. This is simply because regular functions on ${Z_{g}^0}$ must be constant over the compact $\PP^1$ components, so the functions on the punctured component $P_{0,1}$ must hence take on the same value at $0$ and $\infty$. This is therefore equivalent to the algebra of regular functions on a once-punctured $\PP^1$ with two points identified, which is an affine nodal elliptic curve. Since the once-punctured nodal elliptic curve has no moduli, this ring of functions is isomorphic to $A$. Explicitly, regular functions on $Z_{g}^0$ correspond to rational functions $f \in \CC[x,(1+x)^{-1}]$ with $f(0) = f(\infty)$. These are generated as an algebra by the functions 
\begin{align*}
    1 & = \frac{1+x}{1+x}, \\
    Y &= \frac{x}{(1+x)^2},\\
    Z & = \frac{x^2}{(1+x)^3}
\end{align*}
which indeed satisfy the relation $YZ = Y^3+Z^2$.

To calculate the cohomology of $\scr{O}_{Z_{g}^0}$ we use the \v{C}ech complex:
 \begin{equation*}
     0 \to \scr{O}_{Z_{g}^0}({Z_{g}^0}) \to \bigoplus_{i} \scr{O}_{{Z_{g}^0}}(U_{p_i}) \overset{\dd}{\longrightarrow} \bigoplus_{i < j} \scr{O}_{{Z_{g}^0}}(U_{p_i} \cap U_{p_j})  \to  0
 \end{equation*}
The map $\dd$ will be surjective on Laurent polynomials with no constant terms, but $\dd$ will take a constant $c \in \scr{O}_{Z_{g}^0}(U_{p_i})$ to $(c,c,c) \in \bigosum_{j}\scr{O}_{{Z_{g}^0}}(U_{p_i} \cap U_{p_j})$ where $p_j$ are the three vertices adjacent to $p_i$. Over all $i$, this means the image of $\dd$ inside the $(3g-3)$-dimensional space of constant functions in $\bigoplus_{i > j} \scr{O}_{{Z_{g}^0}}(U_{p_i} \cap U_{p_j})$ is the span of $(2g-2)$ vectors, though only $(2g-3)$ are linearly independent. The only other constant functions in the image are those in the image of the restriction map $\scr{O}_{Z_{g}^0}(U_{p_0}) \osum \scr{O}_{Z_{g}^0}(U_{p_1}) \to \scr{O}_{{Z_{g}^0}}(U_{p_0} \cap U_{p_1})$, which takes
\begin{equation*}
    \br{\frac{1}{1+x}, \frac{1}{1+x\dash}} \mapsto \frac{1}{1+x} + \frac{x}{1+x} = 1
\end{equation*}
so that $(1,0,0) \in \bigosum_{j}\scr{O}_{{Z_{g}^0}}(U_{p_0} \cap U_{p_j})$ is also in the image of $\dd$. Therefore this gives us a cokernel of dimension $(3g-3) - (2g-3)  - 1 = g-1$ and so this \v{C}ech cohomology calculation tells us that $H^1(\scr{O}_{Z_{g}^0}) \cong \CC^{g-1}$, with all classes represented by constant functions.

The global sections of the sheaf $\tilde{T}_{Z_{g}^0}$ of balanced vector fields is given by $A \langle x \partial_x \rangle \oplus \CC^{g-1}$ (as an $A$-module), consisting of vector fields that satisfy the balancing condition at every vertex. On each compact $\PP^1$ component, every vector field is a constant multiple of $x\partial_x$; around every vertex, these multiples must sum to zero. On the punctured component $P_{0,1}$, vector fields are of the form  $f(x) x \partial_x $ for $f \in \scr{O}(P_{0,1})$. The balancing conditions at every node force the rotation numbers of the vector field at $0$ and $\infty$ of $P_{0,1}$ to agree, so that $f(0) = f(\infty)$. Again, this is the algebra $A$ of functions on a nodal affine cubic curve, and so balanced vector fields are of the form $f(x) x \partial_x$ for $f \in A$ over $P_{0,1}$. The $2g-2$ balancing conditions, along with the $3g-3$ components, mean that there are $g-1$ constant rotations that need to be specified also: $A$ acts on this summand by evaluation at $0$.

To calculate the cohomology of $\tilde{T}_{Z_{g}^0}$, we consider the \v{C}ech complex
 \begin{equation*}
    0 \to  \tilde{T}_{{Z_{g}^0}}({Z_{g}^0}) \to \bigoplus_{i} \tilde{T}_{{Z_{g}^0}}(U_{p_i}) \overset{\dd}{\longrightarrow} \bigoplus_{i < j} \tilde{T}_{{Z_{g}^0}}(U_{p_i} \cap U_{p_j})  \to 0
 \end{equation*}
 Again, for $i \neq 0,1$, the intersections $U_{p_i} \cap U_{p_j}$ are a disjoint union of three $\CC^{\ast}$s and so
 \begin{align*}
    \tilde{T}_{Z_{g}^0}(U_{p_i}) & \cong \CC[x,y,z]/(xy,yz,xz)\langle x \partial_x -  y \partial_y, y \partial_y - z \partial_z \rangle\\
      \bigoplus_{j} \tilde{T}_{Z_{g}^0}(U_{p_i} \cap U_{p_j}) & \cong \CC[x^{\pm}]\langle \partial_x \rangle \osum \CC[y^\pm]\langle \partial_y \rangle \osum \CC[z^{\pm}]\langle \partial_z \rangle
 \end{align*}
Because of the balancing condition, the image of the restriction maps $\tilde{T}_{Z_{g}^0}(U_{p_i}) \to \bigoplus_{j} \tilde{T}_{Z_{g}^0}(U_{p_i} \cap U_{p_j})$ is the subspace of vector fields of the form $f(x)\partial_x + g(y)\partial_y + h(z) \partial_z$ with the condition that $f(0)+g(0)+h(0) = 0$. 

However, for the punctured component $P_{0,1}$ the restriction map $\tilde{T}_{Z_{g}^0}(U_{p_0}) \osum \tilde{T}_{Z_{g}^0}(U_{p_1}) \to \tilde{T}_{Z_{g}^0}(U_{p_0} \cap U_{p_1})$ takes
\begin{equation*}
    \br{-\frac{(x\dash)^2 \partial_{x\dash}}{1+x\dash}, \frac{x^2 \partial_x}{1+x}} \mapsto \frac{x \partial_{x}}{1+x} + \frac{x^2 \partial_x}{1+x}  = x \partial_{x}
\end{equation*}
where we use the same coordinates on $P_{0,1}$ as above. Here, since
\begin{equation*}
    \frac{x^2 \partial_x}{1+x} = \frac{x}{1+x}(x\partial_x - y \partial_y)
\end{equation*}
both are balanced vector fields themselves (i.e. sections of $\tilde{T}_{Z_{g}^0}(U_{p_i})$ for $i=0,1$) and so the unbalanced vector field $x\partial_{x}$ is in the image of $\dd$. Since all balanced vector fields are also in the image of $\dd$, this means that $\dd$ is surjective and therefore the \v{C}ech cohomology calculation yields $H^1(\tilde{T}_{Z_{g}^0}) = 0$.
\end{proof}

In the case of Dehn twists around $k$ disjoint homologically linearly independent closed curves on $\Sigma_g$ (see Remark \ref{remark: topological assumption on Dehn twists}), the mirror will be $(X\dash_g,W\dash_g)$, where the critical locus $\mathrm{crit}(W_g\dash) = {Z_{g}^{\prime}}$ has $k$ punctures on different components. Then a straightforward generalization of the above calculation implies:

\begin{theorem}\label{b side: multiple twists}
Let $Z_{g}^{\prime}$ be the critical locus of the LG model $(X\dash_g, W\dash_g)$ mirror to the $k$-nodal curve $\Sigma_{g}^0$; then
\begin{align*}
\mathrm{HH}^{\mathrm{even}}(\mathrm{MF}(X\dash_g,W\dash_g)) & \cong  H^0(\scr{O}_{{Z_{g}^{\prime}}}) \osum H^1(\tilde{T}_{{Z_{g}^{\prime}}})  \cong A \times _{\mathbb{C}} A \times_{\mathbb{C}} A \times... \times_{\mathbb{C}} A \\
    \mathrm{HH}^{\mathrm{odd}}(\mathrm{MF}(X\dash_g,W\dash_g)) & \cong H^1(\scr{O}_{{Z_{g}^{\prime}}}) \osum H^0(\tilde{T}_{{Z_{g}^{\prime}}}) \cong A \times _{\mathbb{C}} A \times_{\mathbb{C}} A \times... \times_{\mathbb{C}} A \osum \CC^{2g-2}
\end{align*}
as $\CC$-algebras and as $A$-modules respectively, where $A$ is the $\CC$-algebra $A = \CC[Y,Z]/(YZ = Y^3+Z^2)$ and $\times_{\mathbb{C}}$ denotes the fiber product of rings over their common map $A\rightarrow \mathbb{C}$ (iterated $k$ times in each line). 
\end{theorem}

\subsection{Mirrors to open surfaces}

Now consider $Z_{g}\dash$, a trivalent configuration of one $\AA^1$ component and $(3g+4)$-many $\PP^1$ components (one of them punctured), as in Figure \ref{fig:singular}. Again denote the punctured component $P_{0,1}$ and we will denote the $\AA^1$ component by $P_{\infty}$, with local coordinate $t=0$ corresponding to the nodal point. We explain how to modify the above calculations for this case. 

\begin{theorem}\label{thm:b_model_singular}
Let ${Z_{g}\dash}$ be the critical locus of the LG model $(X\dash_g, W\dash_g)$ mirror to the open nodal curve $\Sigma_{g,1}^0$; then
\begin{align*}
H^0(\scr{O}_{Z_{g}\dash}) \osum H^1(\tilde{T}_{Z_{g}\dash})  & \cong A  \times_{\CC} \CC[T]\\
H^0(\tilde{T}_{Z_{g}\dash}) \osum H^1(\scr{O}_{Z_g\dash}) & \cong  \; (\CC[W] \times_{\CC} \CC[T]) \osum \CC^{2g-2}
\end{align*}
as $\CC$-algebras and as $A\times_{\CC} \CC[T]$-modules respectively, where $A$ denotes the $\CC$-algebra $A = \CC[Y,Z]/(YZ - Z^2- Y^3)$. The fiber product in the first is taken along the evaluation at zero maps; while in the second it is taken along the evaluation maps $F(W) \mapsto F(0)-F(1)$ and $g(T) \mapsto g(0)$. The algebra $A$ acts on $\CC[W]$ by $Y \mapsto W-W^2, Z \mapsto W^2-W^3$.
\end{theorem}
\begin{proof}
In the case of $H^0(\scr{O}_{Z\dash_g})$, the global functions on $Z_{g}\dash$ simply consist of a pair $(f,g)$ where $f \in A$ is a global function on $P_{0,1}$ (with $f(0) = f(\infty)$) and $g \in \CC[t]$ is a function on $\AA^1$ such that $g(0) = f(0)$. This means exactly that $H^0(\scr{O}_{Z\dash_g})$ is the fiber product $A \times_{\CC} \CC[T]$ of algebras over their evaluation at zero maps to $\CC$. For the same reasons as in the proof of Theorem \ref{thm:b_model}, $H^1(\tilde{T}_{Z\dash_g}) = 0$.

To calculate the cohomology group $H^1(\scr{O}_{Z\dash_g})$ we may use the same \v{C}ech complex as above. Now, the intersections $U_i \cap U_j$ will consist of $(3g-3)$ copies of $\CC^{\ast}$s and one pair of pants. Again, $\dd$ will be surjective onto Laurent polynomials with no constant terms, and there are $2g-2+1$ constraints on the possible constant terms (corresponding to the number of nodes), minus $1$ for the additional way of creating constant terms on the component $P_{0,1}$ as described in the proof of Theorem \ref{thm:b_model}. Therefore, the dimension of the cokernel of $\dd$ is $(3g-3) - (2g-2+1-1) = g-1$, and each of these classes can be represented by constant functions. Hence $H^1(\scr{O}_{Z\dash_g}) \cong \CC^{g-1}$ and $H^0(\scr{O}_{Z\dash_g})$ acts by constants. 

Finally, we calculate $H^0(\tilde{T}_{Z\dash_g})$, the global balanced vector fields. If we write $f(x) x \partial_x$ for a section of $\tilde{T}_{Z\dash_g}$ over $P_{0,1}$ (with $f \in \CC[x,(1+x)\inv]$ and $f(\infty) < \infty$), and $g(t)t\partial_t$ for a section over $P_{\infty}$ (with $g \in \CC[t]$), then because of the non-compact $\AA^1$ component $P_{\infty}$, the balancing conditions imply that we must have that $f(0) - f(\infty) = g(0)$. A basis for $H^0(\scr{O}_{Z\dash_g})$ consists of two kinds of vector fields:
\begin{enumerate}
    \item Those given by non-zero pairs $(f(x) x \partial_x, g(t) t \partial_t)$ with $f(0) - f(\infty) = g(0)$ (extended appropriately to the compact components);
    \item Those $g-1$ linearly independent vector fields given by constant rotations of the compact components satisfying the balancing conditions, vanishing on $P_{0,1}$ and $P_{\infty}$.
\end{enumerate}
The algebra $H^0(\scr{O}_{Z\dash_g})$ acts by evaluation at zero on vector fields of type (2): this gives a submodule isomorphic to $\CC^{g-1}$. Vector fields of type (1) correspond to pairs
\begin{equation*}
    \left(F\left(\frac{x}{1+x}\right) x \partial_x, g(t) t \partial_t \right)
\end{equation*} 
where $F \in \CC[W]$ is a polynomial in one variable $W$ and $g \in \CC[T]$ is a polynomial in $T$, satisfying a balancing condition: since the difference in rotation numbers $f(0) - f(\infty)$ is given by $F(0) - F(1)$, the balancing condition becomes $F(0) - F(1) =g(0)$. In this notation, the generator $Y$ of $A$ acts by $W-W^2$ and the generator $Z$ by $W^2-W^3$, while $T$ simply acts by multiplication by $T$. This gives a complete description of the $H^0(\scr{O}_{Z\dash_g})$-module structure on the fiber product $\CC[W] \times_{\CC} \CC[T]$ of $\CC[W]$ and $\CC[T]$ over the evaluation maps $F(W) \mapsto F(0)-F(1)$ and $g(T) \mapsto g(0)$.
\end{proof}

When $\Sigma_{g,k}$ has disjoint homologically linearly independent vanishing cycles and $k \geq 1$, the mirror $Z_g\dash$ to the nodal curve $\Sigma_{g,k}^0$ with more than one puncture will be another trivalent configuration of of $k$-many $\AA^1$ components and $(3g+3+k)$-many $\PP^1$ components (one of them punctured). We have a straightforward generalization to this case:
 
\begin{theorem}
Let ${Z_{g}\dash}$ be the critical locus of the LG model $(X\dash_g, W\dash_g)$ mirror to the open nodal curve $\Sigma_{g,k}^0$; then
\begin{align*}
H^0(\scr{O}_{Z_{g}\dash}) \osum H^1(\tilde{T}_{Z_{g}\dash})  \cong A \times_{\CC}\CC[T_1] \times_{\CC} \CC[T_2] \times_{\CC} \cdots \times_{\CC}\CC[T_k]
\end{align*}
as $\CC$-algebras, where the fiber product is taken over the evaluation maps at zero; and
\begin{equation*}
 H^0(\tilde{T}_{Z_{g}\dash}) \osum H^1(\scr{O}_{Z_g\dash}) \cong  \; (\CC[W]\times_{\CC}(\CC[T_1] \times \CC[T_2] \times \cdots \CC[T_k])) \oplus \CC^{2g-2}
\end{equation*}
as $A  \times_{\CC}\CC[T_1] \times_{\CC} \CC[T_2] \times_{\CC} \cdots \times_{\CC}\CC[T_k]$-modules, where the fiber product is taken over the evaluation maps $F(0) - F(1) = g_1(0) + \cdots + g_k(0)$, and $A$ acts on $\CC[W]$ as above and on $\CC^{2g-2}$ by constants.
\end{theorem}
\begin{proof}
    Observe that in the above, we will introduce one additional variable $t_i$ for each $\AA^1$ component, and regular functions on $Z_{g}\dash$ will consist of tuples of polynomials $(f, g_1, \dots, g_k)$ with $f \in \scr{O}_{P_{0,1}}$ and $g_i \in \CC[t_i]$ satisfying $f(0)=f(\infty)=g_1(0) =\cdots = g_k(0)$. For the balanced vector fields, the difference between the rotation numbers of $f(x) x \partial_x$ at $0$ and $\infty$ will again be given by the total rotation number around infinity of the vector fields $g(t_1)t_1 \partial_{t_1}, \dots, g(t_k) t_k \partial_{t_k}$ on the $\AA^1$-components, equal to $g_1(0) + \cdots + g_k(0)$.
\end{proof}






\subsection{Homogeneous coordinate rings}

In this section, we will describe the $B$-side calculations used to prove Theorem \ref{thm:homog}, after first justifying why this is the appropriate quantity to calculate. 

One expects a version of the Hochschild-Kostant-Rosenberg theorem with coefficients (see \cite[Theorem 3.11]{kraehmer} for the affine case) to apply for matrix factorization categories, so that the Hochschild cohomology with coefficients in a power of a line bundle $\scr{L}$ on $Z_g$ (extended to $X_g$) is isomorphic to
\begin{equation*}
    \mathrm{HH}^{k}(\mathrm{MF}(X_g,W_g), \scr{L}^{\otimes d}) \cong \bigosum_{i+j \cong k \; \text{mod} \; 2} \mathbb{H}^{i}(\wedge^{j} T_{{Z_{g}}} \otimes \scr{L}^{\otimes d})
\end{equation*}
the hypercohomology of the sheaf of (balanced) polyvector fields on the critical locus $Z_g$ with coefficients in a power of $\scr{L}$. 

The right hand side of the above can be calculated explicitly:
\begin{theorem}
If $\scr{L}$ is a line bundle over $Z_g$ that has degree $1$ over one single $\PP^1$ component and is trivial over all the others, then there is an equivalence of algebras:
    \begin{equation*}
        \bigosum_{k=0}^{\infty} H^0(\scr{L}^{\otimes k}) \osum H^1(\tilde{T}_{Z_g} \otimes \scr{L}^{\otimes k}) \cong R \osum \CC
    \end{equation*}
    where $R$ is the graded $\CC$-algebra $\CC[X,Y,Z]/(XYZ = Y^3 +Z^2)$ with $|X|=1, |Y|=2,|Z|=3$ and the $\CC$ summand is a square-zero extension by the module $R/(X,Y,Z)$. Moreover, there is an equivalence of $R$-modules,
    \begin{equation*}
                \bigosum_{k=0}^{\infty} H^1(\scr{L}^{\otimes k}) \osum H^0(\tilde{T}_{Z_g} \otimes \scr{L}^{\otimes k}) \cong R \osum (\CC[X])^{2g-2} \osum \CC
    \end{equation*}
    where $\CC$ is the $R$-module $R/(X,Y,Z)$.
\end{theorem}

\begin{proof}
The line bundle $\scr{L}$ is trivial over all but one component of $Z_g$, where it has degree $1$. Let $P_{0,1}$ denote this distinguished $\PP^1$-component of $Z_g$ over which $\scr{L}$ is non-trivial, so that $\scr{L}^{\otimes k}|_{P_{0,1}} \cong \scr{O}_{\PP^1}(k)$. Written in local coordinates, sections of $\scr{L}^{\otimes k}$ over $P_{0,1}$ correspond to pairs $(f,f\dash)$, where $f \in \CC[z], f\dash \in \CC[z\dash]$ are a pair of polynomials satisfying $f(z) = z^k f\dash(\frac{1}{z})$. Sections of $\scr{L}^{\otimes k}$ must be constant over all other components of $Z_g$, so that we must also have $f(0) = f\dash(0)$. Therefore, sections of $\scr{L}^{\otimes k}$ correspond to sections of $\scr{O}_{\PP^1}(k)$ that have the same value at $0$ and $\infty$. Hence we see that the graded algebra structure on $\bigosum_{k \geq 0} H^0(\scr{L}^{\otimes k})$ is just isomorphic to the homogeneous coordinate ring of a degree-$1$ line bundle on a nodal elliptic curve. One may show using an elementary argument that there is an equivalence of graded algebras
    \begin{equation*}
        \bigoplus_{k=0}^{\infty} H^0(\scr{L}^{\otimes k}) \cong \CC[X,Y,Z]/(XYZ = Y^3 +Z^2)
\end{equation*}
where $|X|=1, |Y|=2,|Z|=3$ (see \cite[Proposition IV.4.6]{hartshorne} or \cite[\S 6.1]{hosgood}).

Explicit representatives for these generators in local coordinates on $P_{0,1}$ are given by:
\begin{itemize}
    \item $X$ is represented by $(1+z, 1+z\dash)$, and is $1$ on all other components;
    \item $Y$ is represented by $(z, z\dash)$, and is $0$ on all other components;
    \item $Z$ is represented by $(z^2, z\dash)$, and is $0$ on all other components.
\end{itemize}
which satisfy the relation $XYZ = Y^3 +Z^2$.

Next, we calculate the module structure on $\bigosum_{k \geq 0} H^0(\tilde{T}_{Z_g} \otimes \scr{L}^{\otimes k})$, the balanced vector fields with coefficients in $\scr{L}^{\otimes k}$. Over the component $P_{0,1}$, a section of $\tilde{T}_{Z_g} \otimes \scr{L}^{\otimes k}$ corresponds to a pair $(f, f\dash)$ with $f \in \CC[z], f\dash \in \CC[z\dash]$ satisfying $f(z) = z^{k} f\dash(\frac{1}{z})$. Over all of $Z_g$ a section $s$ of $\tilde{T}_{Z_g} \otimes \scr{L}^{\otimes k}$ corresponds to a tuple $(f, f\dash, a_2, \cdots, a_{3g-3})$ with $a_i \in \CC$, subject to $2g-4$ balancing conditions among the $a_i$s around each node except $p_0,p_1$, plus two further balancing conditions at $p_0$ and $p_1$, given by $f(0)+a_2+a_3=0, f\dash(0)+a_4+a_5=0$. Here the corresponding section $s$ is given by the vector field $a_i z\partial_z$ on component $i \neq 1$ of $Z_g$. We can find a basis of sections of $\tilde{T}_{Z_g} \otimes \scr{L}^{\otimes k}$ consisting of:
\begin{enumerate}
    \item sections with $f(0) = f\dash(0)$,
    \item sections that are zero on $P_{0,1}$.
\end{enumerate}
There are $k$ basis vectors of the first kind, corresponding to a basis of $H^0(\scr{L}^{\otimes k})$; there are $g-1$ basis vectors of the second kind, there being $3g-4$ constants to choose, subject to $2g-3$ balancing conditions. Hence, 
\begin{equation*}
    H^0(\tilde{T}_{Z_g} \otimes \scr{L}^{\otimes k}) \cong H^0(\scr{L}^{\otimes k}) \osum \CC^{g-1}
\end{equation*}
and the action of the algebra $\bigosum_{k \geq 0} H^0(\scr{L}^{\otimes k})$ preserves this direct sum decomposition. As a module over $\bigosum_{k \geq 0} H^0(\scr{L}^{\otimes k})$, the first summand corresponds to $\bigosum_{k \geq 0} H^0(\scr{L}^{\otimes k})$ as a module over itself. Of the generators of the ring $\CC[X,Y,Z]/(XYZ = Y^3 + Z^2)$, the sections $Y,Z$ act on the $\CC^{g-1}$ summand by $0$, while $X$ is $1$ outside of $P_{0,1}$, and so restricts to an isomorphism of $\CC^{g-1} \subset H^0(\tilde{T}_{Z_g} \otimes \scr{L}^{\otimes k})$ to $\CC^{g-1} \subset H^0(\tilde{T}_{Z_g} \otimes \scr{L}^{\otimes (k+1)})$. Therefore, as a module,
\begin{equation*}
    \bigosum_{k=0}^{\infty}H^0(\tilde{T}_{Z_g} \otimes \scr{L}^{\otimes k}) \cong \CC[X,Y,Z]/(XYZ = Y^3 + Z^2) \osum (\CC[X])^{g-1}
\end{equation*}

Next, we calculate the module structure on the cohomology groups $\bigosum_{k \geq 0} H^1(\scr{L}^{\otimes k})$. To calculate the cohomology of $\scr{L}^{\otimes k}$ we use the \v{C}ech complex:
 \begin{equation*}
     0 \to \scr{L}^{\otimes k}({Z_{g}}) \to \bigoplus_{i} \scr{L}^{\otimes k}(U_{p_i}) \overset{\dd}{\longrightarrow} \bigoplus_{i < j} \scr{L}^{\otimes k}(U_{p_i} \cap U_{p_j})  \to  0
 \end{equation*}
 where $U_i$ is the same open cover used above. We have
\begin{align*}
  &\scr{L}^{\otimes k}(U_{p_i}) \cong \CC[x,y,z]/(xy,yz,xz) \\ & \scr{L}^{\otimes k}(U_{p_j}) \cong \CC[x\dash,y\dash,z\dash]/(x\dash y\dash,y\dash z\dash ,x\dash z\dash) \\
  &\scr{L}^{\otimes k}(U_{p_i} \cap U_{p_j}) \cong \CC[x^\pm]\osum \CC[y^\pm]\osum \CC[z^\pm]
\end{align*}
and the restriction map takes $x,y,z \mapsto x,y,z$ and $x\dash,y\dash,z\dash \mapsto 1/x,1/y,1/z$ respectively, except on the component $P_{0,1}$, where $f\dash(z\dash) \mapsto z^k f\dash(\frac{1}{z})$. For $k \geq 0$, this means that the image of $\dd$ contains all polynomials with no constant terms, and so the cokernel of $\dd$ is spanned by $g-1$ classes represented by constant functions (which can be chosen to be zero on $P_{0,1}$) just as in the proof of Theorem \ref{thm:b_model} above. Therefore, $H^1(\scr{L}^{\otimes k}) \cong \CC^{g-1}$ and $Y,Z \in H^0(\scr{L}^{\otimes k})$ act by $0$, while again since $X$ is $1$ on all components other than $P_{0,1}$, multiplying by $X$ gives an isomorphism $H^1(\scr{L}^{\otimes k}) \to H^1(\scr{L}^{\otimes (k+1)})$. Hence, as a $\bigosum_{k \geq 0} H^0(\scr{L}^{\otimes k})$-module, 
\begin{equation*}
    \bigosum_{k=0}^{\infty} H^1(\scr{L}^{\otimes k})  \cong (\CC[X])^{g-1} \osum \CC
\end{equation*}
where the extra $\CC$ summand comes from the case $k=0$ and so all positive-degree generators of $R$ act by zero on this summand.

Finally, we need to calculate $H^1(\tilde{T}_{Z_g} \otimes \scr{L}^{\otimes k})$, again using the \v{C}ech complex:
 \begin{equation*}
     0 \to \tilde{T}_{Z_g} \otimes\scr{L}^{\otimes k}({Z_{g}}) \to \bigoplus_{i} \tilde{T}_{Z_g} \otimes\scr{L}^{\otimes k}(U_{p_i}) \overset{\dd}{\longrightarrow} \bigoplus_{i < j} \tilde{T}_{Z_g} \otimes\scr{L}^{\otimes k}(U_{p_i} \cap U_{p_j})  \to  0
 \end{equation*}
 where
  \begin{align*}
    \tilde{T}_{Z_{g}}\otimes\scr{L}^{\otimes k}(U_{p_i}) & \cong \CC[x,y,z]/(xy,yz,xz)\langle x \partial_x -  y \partial_y, y \partial_y - z \partial_z \rangle\\
        \tilde{T}_{Z_{g}}\otimes\scr{L}^{\otimes k}(U_{p_i}) & \cong \CC[x\dash,y\dash,z\dash]/(x\dash y\dash,y\dash z\dash,x\dash z\dash)\langle x\dash \partial_{x\dash} -  y\dash \partial_{y\dash}, y\dash \partial_{y\dash} - z\dash \partial_{z\dash} \rangle\\
      \bigoplus_{j} \tilde{T}_{Z_{g}}\otimes\scr{L}^{\otimes k}(U_{p_i} \cap U_{p_j}) & \cong \CC[x^{\pm}]\langle \partial_x \rangle \osum \CC[y^\pm]\langle \partial_y \rangle \osum \CC[z^{\pm}]\langle \partial_z \rangle
 \end{align*}
 and the restriction map takes $x,y,z \mapsto x,y,z$ and $x\dash,y\dash,z\dash \mapsto 1/x,1/y,1/z$ respectively, except on the component $P_{0,1}$, where $f\dash(z\dash) z\dash \partial_{z\dash} \mapsto - z^k f\dash(\frac{1}{z}) z \partial_z$. For $k \geq 1$, this restriction map on $U_0 \cap U_1$ means the rotation numbers of sections of $\tilde{T}_{Z_g} \otimes \scr{L}^{\otimes k}$ over $P_{0,1}$ need not be the same at $p_0$ and $p_1$, so, as in the proof of Theorem \ref{thm:b_model}, all unbalanced vector fields are in the image of $\dd$, and $\dd$ is surjective, meaning $H^1(\tilde{T}_{Z_{g}}\otimes\scr{L}^{\otimes k}) = 0$. For $k=0$, the image of $\dd$ contains only balanced vector fields and so $H^1(\tilde{T}_{Z_g}) \cong \CC$, and the product of this class with itself is therefore zero.
 \end{proof}

The degree-$0$ part of the above algebra is a kind of homogeneous coordinate ring and can be calculated using an elementary Riemann-Roch argument: see \cite[Proposition IV.4.6]{hartshorne}. We can use this to see that if we start with $\scr{L}^{\otimes 2}$ instead, we will get:
\begin{proposition} \label{prop:B side calculation for phi^2}
There is an equivalence of graded algebras
\begin{equation*}
        \bigoplus_{k=0}^{\infty} H^0(\scr{L}^{\otimes 2k}) \cong \CC[X,Y,Z]/(XYZ = Y^2 +Z^4)
\end{equation*}
where $|X|=1,|Y|=2,|Z|=1$. Denote this graded algebra by $R_2$.
\end{proposition}
This comes from the map ${{Z_{g}}} \to \PP(1,1,2)$ whose image is a quartic hypersurface in the weighted projective plane $\PP(1,1,2)$: see \cite[\S 6.1]{hosgood}.

Starting from this, it is not difficult to prove that
\begin{theorem}
If $\scr{L}$ is a line bundle over $Z_g$ that has degree $1$ over one single $\PP^1$ component and is trivial over all the others, then there is an equivalence of graded algebras:
    \begin{equation*}
        \bigosum_{k=0}^{\infty} H^0(\scr{L}^{\otimes 2k}) \osum H^1(\tilde{T}_{Z_g} \otimes \scr{L}^{\otimes 2k}) \cong R_2 \oplus \CC
    \end{equation*}
    where $\CC$ summand is a square-zero extension by $R_2/(X,Y,Z)$. Moreover, there is an equivalence of $R_2$-modules,
    \begin{equation*}
                \bigosum_{k=0}^{\infty} H^1(\scr{L}^{\otimes 2k}) \osum H^0(\tilde{T}_{Z_g} \otimes \scr{L}^{\otimes 2k}) \cong R_2 \oplus (\CC[X])^{2g-2} \oplus \CC
    \end{equation*}
    where $\CC$ is a square-zero extension by $R_2/(X,Y,Z)$.
\end{theorem}

There are analogous results for the graded rings obtained by starting with $\scr{L}^{\otimes k}$ instead for $k \geq 3$, corresponding to further embeddings of nodal elliptic curves into weighted projective spaces: see \cite[\S 6.1]{hosgood} for details.

\begin{remark}
    The results from \cite{hosgood,kraehmer} are stated with the assumption of smoothness, but they continue to apply in our case with the singular curve ${{Z_{g}}}$.
\end{remark}

\section{Standard Lefschetz fibrations}\label{sec:lefschetz}

In this section we review some of the important geometric setups that will be relevant in the A-side calculations, especially those related to the Seidel class. 

Given a compactly-supported symplectomorphism $\phi: M \to M$ of a Liouville domain $(M,\lambda_M)$, define a \textbf{standard Lefschetz fibration} $\pi: E \to \CC$ for $\phi$ to be an exact symplectic Lefschetz fibration over $\CC$ (in the sense of \cite{seidel}) with total space a Liouville domain $(E,\lambda_E)$ with the following properties:
\begin{itemize}
    \item Each smooth fiber $F_t = \pi\inv(t)$ is diffeomorphic to $M$ and the restricted Liouville form $(F_t, \lambda_{E}|_{F_t})$ makes the fiber $F_t$  exact symplectomorphic to $(M, \lambda_M)$ via symplectic parallel transport.
    \end{itemize}
Using clockwise parallel transport around the boundary of the unit disk we get symplectomorphisms $\phi_t: M \to F_{\e{it}}$ and identify $\phi_{2 \pi}: M \to M$ with the symplectomorphism $\phi$. Then we require the additional condition:
    \begin{itemize}
        \item Outside of the disk $\set{|z| \leq {1-\epsilon}}$, there is an isomorphism of $\pi: E \to \CC$ as an exact symplectic fiber bundle, with the symplectization of a mapping torus $\mathbb{R}\times M_\phi$, where $M_\phi$ is the mapping torus of $\phi$ with the fiberwise Liouville $1$-form induced from $\lambda_M$. Here we view the symplectization of the mapping torus $\RR\times M_{\phi}$ as an exact symplectic fiber bundle with base $\RR\times S^1$, and fiber $(M,\lambda_M)$.
    \end{itemize}

One has a more general notion of a \textbf{standard Lefschetz fibration over a Riemann surface $S$ with ends}: strip-like ends, cylindrical ends, and \textit{twisting data} (a choice of integer $b_i$ for every boundary component and interior puncture), compare \cite[p.239]{seidel}. Here the counterclockwise parallel transport of $\pi:E \to S$ along each boundary component or (sufficiently close to each) interior puncture is identified with the corresponding $\phi^{b_i}$, and $\pi:E \to S$ comes with isomorphisms (as exact symplectic fiber bundles) with the (positive or negative) symplectization of an open mapping torus of $\phi^{b_i}$ over each cylindrical end, and with the trivial product bundle over each strip-like end. The above is a special case where $S=\CC$ with a positive cylindrical end outside $\set{|z|\leq 1 -\epsilon}$ and twisting data given by $1$.

If $V \subset M$ is a framed exact Lagrangian sphere in $M$, Seidel in \cite{seidel_LES} builds a canonical standard Lefschetz fibration over $\CC$ for the Dehn twist $\phi_V$, with a single critical point of $\pi$ over $0\in \mathbb{C}$. We may then choose to identify the fiber over $1$ with $M$, and identify $V$ with the vanishing cycle obtained by parallel transport along the ray $\RR_{\geq 0}$. By a lemma of \cite{seidel_LF2} which allows us to identify a Lefschetz fibration at infinity with the mapping cylinder of its global monodromy, the second condition will then follow. By \cite{giroux_pardon}, there is an essentially unique way to associate a standard Lefschetz fibration to a Dehn twist on a Weinstein manifold, up to Weinstein homotopy. In general, of course, there is no canonical standard Lefschetz fibration associated to a symplectomorphism (despite the name). 

For later use, we give a recall of Seidel's construction of a standard Lefschetz fibration associated to a Dehn twist:

\begin{proposition}\label{prop:standard_LF}
Given a circle $S^1 \subset \Sigma$ inside a punctured Riemann surface $(\Sigma,\omega)$, we can choose a Liouville form $\lambda_\Sigma$ so that we can construct the standard Lefschetz fibration for the Dehn twist $\phi$ along this circle.
\end{proposition}

\begin{proof}
Most of this is already done in e.g. \cite{seidel_LES}. In the following, we will closely follow the setup in \cite{seidel_LES} and explain the necessary modifications that need to be made to suit our purposes. 

Fix positive real numbers $\lambda, r>0$. We choose a local coordinate $(x,y)\in (-\lambda,\lambda)\times S^1$ for the region around the Dehn twist, over which the one-from $\lambda_\Sigma=x dy$. Fix a Morse function $H_0$ on the complement of a small neighborhood of the circle $\{x=0\}\subset \Sigma$, with two local minimum corresponding to the fixed points $e_0^1$ and $e_1^1$. We take the local model of the fibration $E$ over the disk $\mathbb{D}_r$ to be the same as the one described in \cite{seidel_LES} equation (1.17), with one modification coming from the Hamiltonian perturbation $H_0$:
\[
\omega=d(xdy+(g(|x|)-1)\tilde{R}_{r}(|x|)d\theta +H_0 d\theta),
\]
where $g:[0,\lambda)\to \mathbb{R}$ is a smooth, non-decreasing function with $g(s)=0$ for $s$ small and $g(s)=1$ for $s$ close to $\lambda$, and $\tilde{R}_r(s)=\frac{1}{2}(s-\sqrt{s^2+r^2/4})$.

Outside of the twist region, the fibration over $\mathbb{C}$ is defined as the trivial product together with the fiberwise two-form
\[
\omega = \omega_\Sigma+ d(H_0 d\theta).
\]

To extend the fibration to $\mathbb{C}\setminus \mathbb{D}_r$, we pick a non-decreasing smooth function $\psi:[0,\infty)\to[0,\infty)$ with the following properties:
\begin{enumerate}
    \item $\psi(s)=s$ for $s\in[0,r]$, and
    \item There exist $R>r$ and $C>0$ such that $\psi(s)=C$ for $s>R$.
\end{enumerate}

Now we define the local structure of the fibration $E$ over $\mathbb{C}$ to be given by the two-form
\[
\omega=d(xdy+(g(|x|)-1)\tilde{R}_{\psi(r)}(|x|)d\theta +H_0(x,y)d\theta).
\]

It is not difficult to see that over $\mathbb{C}\setminus \mathbb{D}_R$, the fibration is isomorphic (as an exact Lefschetz fibration) to the symplectization of the mapping torus $(R,\infty) \times M_\phi$, where $\phi$ is the time-1 map of $(g(|x|)-1)\tilde{R}_C(|x|)+H_0$.
\end{proof}

\begin{remark} \label{remark:seidel class for Riemann surface with boundary}
Note that our $\pi:E \to \CC$ is an exact symplectic Lefschetz fibration in the sense of \cite{seidel}: the total space $(E,\lambda_E)$ is made into a Liouville domain by taking a horizontal convex slice and then rounding the corners. Therefore, our definition of the Seidel class involves only counting those sections contained inside this domain; if wrapping is performed at infinity (in the open setting), then this is performed in the completion of the domain and does not affect the count for the Seidel class.
\end{remark}
\begin{remark}
    If we choose a different Liouville form $\lambda_\Sigma'$ that is only different from $\lambda_\Sigma$ outside of the neighborhood (described in the above proof) of the circle, the resulting Lefschetz fibration, when viewed as a symplectic manifold, will be the same.
\end{remark}
For the non-exact case, let $\Sigma$ denote a closed Riemann surface with symplectic form $\omega_\Sigma$. Using the same techniques as in the exact case (again, we need to fix a local primitive $\lambda_\Sigma=xdy$ of $\omega_\Sigma$ in a neighborhood of the circle with coordinates $(x,y)\in (-\lambda,\lambda)\times S^1$) we have
\begin{proposition}
We can construct a (non-exact) Lefschetz fibration $E\rightarrow \mathbb{C}$ with the following properties.
\begin{itemize}
\item The generic smooth fiber is $\Sigma$, a Riemann surface with punctures with symplectic form $\omega_\Sigma$.
\item The only critical point lies over $0\in \mathbb{C}$, with the monodromy map a Dehn twist about a circle $S^1 \subset \Sigma$.
\item Away from the origin $0\in \mathbb{C}$, the Lefschetz fibration is isomorphic as a symplectic fiber bundle to the symplectization of a mapping torus $M_\phi \times \mathbb{R}$, where $M_\phi$ is the mapping torus and the symplectic structure on $M_\phi \times \mathbb{R}$ is given by $dr\wedge d\theta + \omega$.
\end{itemize}
Hence we have a (nonexact) standard Lefschetz fibration for the Dehn twist of $\Sigma$ around this circle.
\end{proposition}

For future use we now construct a Lefschetz fibration with the following properties:

\begin{proposition}\label{prop:two_crit_points}
    Let $\Sigma$ be a Riemann surface (closed or with punctures). Let $V_{-},V_{+} \subset \Sigma$ be vanishing circles along which we perform Dehn twists $\phi_{V_i}$. Then we can construct an (exact or non-exact, respectively) Lefschetz fibration $\pi:E\rightarrow \CC$ with the following properties
    \begin{itemize}
        \item The generic smooth fiber is $\Sigma$, with symplectic form $\omega_\Sigma$ (in the exact case with also equip it with $d\lambda_\Sigma =\omega_\Sigma$).
        \item There are two critical points, and they lie over $\pm 1 \in \mathbb{C}$. The monodromy map around $\pm 1$ is a Dehn twist around $V_{\pm} \subset \Sigma$.
        \item Outside of a large disk, the Lefschetz fibration is isomorphic as a symplectic fiber bundle to the symplectization of a mapping torus $M_{\phi} \times \mathbb{R}$ of $\phi = \phi_{V_{-}}\phi_{V_{+}}$, where $M_{\phi}$ is the mapping torus with the symplectic structure on $M_\phi \times \mathbb{R}$ given by $dr\wedge d\theta + \omega$.
    \end{itemize}
    Hence we have a standard Lefschetz fibration for the composition of Dehn twists of $\Sigma$ around this $V_{+}, V_{-}$.
\end{proposition}
\begin{proof}
    Take a pair of pants with punctures $\pm 1 \in \CC$ and $\infty$. Construct, as in \cite{ziwenyao}, a symplectic fiber bundle over this pair of pants with fiber $\Sigma$, monodromy $\phi_{V_{\pm}}$ around $\pm 1$ and monodromy $\phi$ around $\infty$. Next, near $\pm 1$ glue in the standard Lefschetz fibrations with one critical point (and monodromy $\phi_{V_{\pm}}$). This is possible because sufficiently far away from the singular fiber the standard Lefschetz fibrations we constructed in Proposition \ref{prop:standard_LF} are isomorphic to the symplectization of a mapping torus.
\end{proof}

\section{Symplectic cohomology of singular hypersurfaces}\label{sec:seidel}

Suppose $f:X \to \CC$ is a holomorphic function on a Stein manifold $X$, having general fiber $M$ and a single singular fiber $M^0$ over $0$; and suppose $\phi:M \to M$ is the counterclockwise monodromy symplectomorphism of the fiber. The Fukaya category $\scr{F}(M^0)$ of the singular fiber is defined in \cite{jeffs}, and is quasi-equivalent to the localization of the Fukaya category $\scr{F}(M)$ of the smooth fiber at Seidel's natural transformation $s: \mathrm{id} \to \phi$ (see below). The purpose of this section is to prove:

\begin{theorem}\label{thm:HF_equals_QH} Suppose $M$ is a non-degenerate Liouville manifold in the sense of \cite[Definition 1.1]{ganatra_thesis}: then the twisted closed-open map $\scr{CO}_{\phi}$ is an isomorphism and there is an equivalence of graded algebras:
    \begin{equation*}
    \mathrm{HH}^{\ast}(\scr{F}(M^0)) \cong \dlim_d \mathrm{HF}^{\ast}(\phi^d)
\end{equation*}
where the connecting maps in the direct limit are given by multiplication by the Seidel class $S$ in $\mathrm{HF}^0(\phi)$ (see Definition \ref{def:seidel class} below).
\end{theorem}

The proof of the first part of this theorem was known to Ganatra, as a generalization of the arguments from \cite{ganatra_thesis}; we give an outline below for completeness. Compare also \cite{abouzaid_ganatra, kotelskiy}, \cite[Conjecture 7.17]{seidel_LF2}, forthcoming work of Shaoyun Bai and Paul Seidel \cite{shaoyun}, as well as forthcoming work of Shuo Zhang. By Seidel's split-generation theorem \cite{seidel}, this non-degeneracy hypothesis will hold whenever $f: X \to \CC$ is an exact symplectic Lefschetz fibration coming from a Lefschetz pencil with a smooth fiber at infinity.

\begin{remark}
In \cite{jeffs}, the monodromy $\phi$ was taken to be the clockwise monodromy, and $s$ to be a natural transformation to the identity from the monodromy functor. This difference is entirely a matter of convention and the categories resulting from localization will be quasi-equivalent. The difference in convention is chosen to be closer to \cite{seidel}. Moreover, the definition of $s$ used in \cite{jeffs}, as the cone of the $\cap - \cup$ adjunction, can be shown to be equivalent to the definition given in \cite{seidel}: this is a result of \cite{abouzaid_seidel}.
\end{remark}

Our setting is substantially simpler than that considered in \cite{abouzaid_seidel}. Since we consider only the case where $f:X \to \CC$ has one Lefschetz critical value at $0$, the wrapped Fukaya category $\scr{W}(X,f)$ is generated by the thimble $T$; any definition of the $\cap$ functor will take $\cap T$ to be the vanishing cycle $V$, then the exact triangle
\begin{center}
\begin{tikzcd}
\cap \cup \ar[swap]{dr}{+1} & {} & \mathrm{id} \ar[swap]{ll}{\eta}\\
{} & \mu \ar{ur}{s} & {}
\end{tikzcd}
\end{center}
simply becomes, for any Lagrangian $L$ in $M$,
\begin{center}
\begin{tikzcd}
\cap \cup \ar[swap]{dr}{+1} & {} & \mathrm{id} \ar[swap]{ll}{\eta}\\
{} & \mu \ar{ur}{s} & {}
\end{tikzcd}
\end{center}
which is exactly Seidel's exact triangle from \cite{seidel1997floer}, and $s$ is exactly his section counting map from \cite{seidel}.

\begin{remark}
\textbf{Signs:} to show that the result of Theorem \ref{thm:HF_equals_QH} holds with $\CC$ coefficients, as we will use in \S \ref{sec:computation} and \S \ref{sec:A_model_calculations}, we will need to ensure that our Floer-theoretic arguments hold with signs. If one uses the standard setup of orientation lines and canonical orientations, as in \cite{seidel,ganatra_thesis}, by choosing consistent orientations of moduli spaces of domains and equipping our Lagrangians with gradings, spin structures and orientations, our moduli spaces are canonically oriented relative to the orientation lines at the ends (see \cite[Lemma B.1]{ganatra_thesis} or \cite[(12.8)]{seidel}). Then the fact that our arguments work with signs is essentially automatic, with two minor subtleties:
\begin{itemize}
    \item Since the holomorphic curves we consider live in the total space of a Lefschetz fibration $E$ rather than in $\Sigma$ itself, we will need to use a (canonical) `stabilization' identification of the orientation lines for $p \in L_i \cap L_{i+1}$ and $\tilde{p} \in \tilde{L}_i \cap \tilde{L}_{i+1}$, given by lifting the brane structures to $E$, in order to have signs for our section counting maps. We have a similiar identification for fixed points of $\phi$.
    \item In our case of a Riemann surface $\Sigma$, we have essentially at most two choices of spin structures on any connected Lagrangian: in the following, we will implicitly choose the non-trivial (bounding) spin structure on compact Lagrangians, so that Seidel's vanishing result continues to hold when taken with signs (see \cite[Example 17.3]{seidel}).
\end{itemize}
For further details on signs and orientations, the reader can refer to the forthcoming \cite{shaoyun}.
\end{remark}

\begin{remark}
Note that the results in this section are stated in the setting where $M$ is an exact symplectic manifold, and so do not directly apply to those calculations in \S \ref{sec:product} where $\Sigma_g$ is a closed Riemann surface. Nevertheless, one expects the same theorems to hold also in the monotone setting, and with the same proofs, provided the appropriate technology were developed.
\end{remark}

\subsection{Compatibility with wrapping}\label{sec: Compatibility with wrapping}

In the following we shall also want to consider operations between \textit{wrapped} Floer complexes, in the cases where $M$ is not compact. Thus we will want to consider domains that carry (implicit) rescaling data. This of course poses no problem if we cut off our rescaling diffeomorphism to lie away from the region in which the compactly-supported symplectomorphism is taking place. Thus the following is largely a straightforward combination of \cite[\S 17]{seidel} and \cite[\S 4]{ganatra_thesis}.

Suppose $(M,\lambda_M)$ is a Liouville domain, with boundary $\partial M$. Denote its completion to a Liouville manifold by $\hat{M}$, given by attaching $M$ to $\partial M \times [0,\infty)_r$ with Liouville form $\lambda_{\hat{M}} = r \lambda_{\partial M}$. We say a Hamiltonian on $\hat{M}$ is \textbf{admissible} if $\partial_r H(x,r)>0,\partial_x H(x,r)=0$ on $\partial M \times (0,\infty)$, if $H(x,r) =\frac{1}{2} r^2$ outside of a compact set, and $\partial_r H(x,0)=0$. We say an almost-complex structure $J$ is of \textbf{$c$-rescaled contact type} if $\lambda(J(\partial_r)) = -1$ for $r$ sufficiently small, and $c r^{-1} \lambda \circ J = \dd r$ for some constant $c$ outside of some compact set. Our Floer data will always be a choice of admissible Hamiltonian $H_t$ and rescaled contact-type almost-complex structure $J_t$ for each generator of a Floer complex (possibly $t$-dependent, to break the $S^1$ symmetry). We say a Lagrangian $L$ inside $\hat{M}$ is \textbf{strictly cylindrical} if $L = \Lambda \times [0,\infty)$ inside $\partial M \times [0,\infty)$, where $\Lambda \subset \partial M$ is a compact Legendrian submanifold of $(M,\lambda|_{\partial M})$. For our purposes in \S \ref{sec:product}, $M$ will be a Stein manifold, so the Fukaya category is generated by strictly-cylindrical Lagrangians (see \cite{GPS2}), and defining our operations for these Lagrangians is sufficient.

It is important to note that in the following sections, when choosing Floer data and perturbation data, we use the same classes of Hamiltonians and almost-complex structures as we use to define the operations on the fixed point Floer homology groups on $M$ as in \S 2 of \cite{ziwenyao} away from the puncture regions (if any), and use admissible Hamiltonians and almost complex structures of rescaled contact type near the punctures. We do this in order for the closed-open map to respect the product structure (Theorem \ref{thm:product_agrees}). The details of the product operation in fixed point Floer (co)homology are described in Section \ref{sec:product}.

Define a \textbf{rescaling diffeomorphism} $\psi^{\rho}:\hat{M} \to \hat{M}$ as follows: $\psi^{\rho} = \mathrm{id}$ on $M$, and on $\partial M \times [0,\infty)$, it is $\psi^{\rho}(x,r) = (x, f_{\delta,\rho}(r))$, where $f_{\delta,\rho}(r)$ is a small convex smoothing of:
\begin{align*}
    f_{\delta,\rho}(r) = \left\{ \begin{matrix} r, & r \geq 2\delta, \\ \delta^{-1}(\rho-1)r^2-(\rho-2)r, & \delta \leq r  \leq 2\delta, \\ \rho r, & 0 \leq r  \leq \delta.\end{matrix} \right.
\end{align*}
Given an admissible Hamiltonian $H$, let $h_{\delta,\rho}(r)$ denote the unique smooth solution to the differential equation:
\begin{equation*}
    H(f_{\delta,\rho}^2(r)) h\dash(r) + f_{\delta,\rho}\dash(r)h(r)\pd{H}{r}-\pd{H}{r}=0
\end{equation*}
with $h_{\delta,\rho}(0)=1$. The diffeomorphism $\psi^\rho$ has the following important properties:
\begin{lemma}\label{lemma:handy}
Given an admissible Hamiltonian $H$, choose $\delta$ sufficiently small so that all of the $1$-periodic orbits of $H$ take place outside $\partial M \times [0, 2\delta]$. Then:
\begin{enumerate}
    \item For all $\rho>0$, $\psi^\rho$ takes strictly cylindrical Lagrangians to strictly cylindrical Lagrangians;
    \item The pullback $\rho^{-2}(\psi^\rho)^{\ast}H$ is also an admissible Hamiltonian;
    \item If $J$ is an almost-complex structure on $(\hat{M},\lambda)$ that is of rescaled contact type, then the pullback $(\psi^\rho)^\ast J$ is an almost-complex structure for the same $(\hat{M},\lambda)$ and is also of rescaled contact type (for a different value of $c$).
    \item If $\phi$ is any compactly-supported symplectomorphism of $M$ extended to $\hat{M}$, then $\psi^\rho$ commutes with $\phi$.
\end{enumerate}
\end{lemma}
Since $\rho^{-2}(\psi^\rho)^{\ast}H$ is an admissible Hamiltonian, so is $h_{\rho,\delta}(r)(\psi^\rho)^{\ast}H$. Moreover, under $\psi^\rho$, orbits of $H$ are in bijection with periodic orbits of $h_{\rho,\delta}(r)(\psi^\rho)^{\ast}H$ on $\partial M \times [0,\infty)$:
\begin{proposition}
Given any compactly-supported symplectomorphism $\phi: M \to M$, there is a canonical isomorphism of fixed point Floer cohomology complexes:
\begin{equation*}
        \mathrm{CF}^\ast(\phi,H_t,J_t) \cong \mathrm{CF}^\ast(\phi, h_{\delta,\rho}(\psi^\rho)^\ast(H_t), (\psi^\rho)^\ast J_t)
\end{equation*}
Likewise, given any pair of strictly cylindrical exact Lagrangians $L_0,L_1$, we have a canonical isomorphism of wrapped Lagrangian Floer cochain complexes:
\begin{equation*}
    \mathrm{CF}^\ast(L_0,L_1,H_t,J_t) \cong \mathrm{CF}^\ast(\psi^\rho L_0, \psi^\rho L_1, h_{\delta,\rho}(\psi^\rho)^\ast(H_t), (\psi^\rho)^\ast J_t).
\end{equation*}
\end{proposition}

The isomorphism of the two complexes arises from the fact that $\phi$ is the identity outside of a compact set: the left hand side is the fixed point Floer cohomology of $\phi$ perturbed by the Hamiltonian flow of $H_t$ near the punctures, using the almost complex structure $J_t$ near the punctures to define the differential. The right hand side is the fixed point Floer cohomology of $\phi$, perturbed by $h_{\delta,\rho}(\psi^\rho)^*(H_t)$ near the punctures, with the almost complex structure $(\psi^\rho)^*J_t$ used to define the differential. The details of how we perform wrapping on fixed point Floer cohomology are described in Section \ref{sec:fixed_point_floer_punctured}. 

Following \cite{seidel}, the boundary of the total space $E$ as a Liouville manifold can be separated into two parts, the vertical part $\partial^v E \cong \partial D \times M$ and the horizontal part $\partial^h E \cong D \times \partial M$. We may likewise define a diffeomorphism $\tilde{\psi}^{\rho}:\hat{E} \to \hat{E}$ that is the identity in the interior $E$, acts by $\psi^\rho$ fiberwise on $\partial^h E$, and by the standard rescaling function on $\CC$ that is constant in the fibers on $\partial^v E$. By our above observation, $\tilde{\psi}^\rho$ takes fibered Lagrangians in $E$ to fibered Lagrangians for all $\rho>0$. Moreover, if $\tilde{L}_i$ denotes the parallel transport of a cylindrization of a Lagrangian $L_i \subset \Sigma$ along a radial arc, then $\tilde{\psi}^\rho(\tilde{L}_i) = \tilde{\br{\psi^\rho L_i}}$.

Using our rescaling functions $\psi^\rho$ and $\tilde{\psi}^\rho$ we may now use identical definitions (as in \cite[Definitions 4.5,4.7,4.7]{ganatra_thesis}) of when perturbation data are adapted to a choice of Floer data for a Riemann surface $S$ with weighted strip-like ends and cylindrical, carrying a standard Lefschetz fibration $\pi:E \to S$ (itself adapted to the ends of $S$ in the sense of \cite[(17b)]{seidel}). This perturbation data consists of a choice of a $1$-form $K \in \Omega^1(S, C^{\infty}(E))$ and a domain-dependent almost-complex structure $J_S$ on $E$, satisfying a list of compatibility conditions with strip-like ends, weighting data, boundary conditions, Floer data, and fibration structures, which can be found in \cite[\S 4.1]{ganatra_thesis} and \cite[\S 17]{seidel}. As stated above, because our rescaling by $\psi^\rho$ and twisting by $\phi$ take place in disjoint regions of $M$ there is no obstruction to finding Floer data and perturbation data satisfying both sets of conditions. We may then carry through the same analysis as in \cite{ganatra_thesis,seidel} to define operations on (wrapped) Floer complexes: we shall leave this data implicit unless it is significant to the argument at hand.

For instance, when defining the product in fixed point Floer cohomology, we want to consider sections of a certain symplectic fiber bundle $\pi:E \to S$ where $S$ is a three-punctured sphere equipped with two negative and one positive cylindrical ends, of weights $n,m,$ and $m+n$, respectively. Counting isolated points in a moduli space of perturbed pseudoholomorphic sections defines a map
\begin{equation*}
        \mathrm{CF}(\phi^n,H_t,J_t) \otimes \mathrm{CF}(\phi^m,H_t,J_t)  \to \mathrm{CF}(\phi^{n+m}, h_{\rho,\delta} (\psi^\rho)^\ast(H_t), (\psi^\rho)^\ast J_t) \cong \mathrm{CF}(\phi^{m+n},H_t,J_t)
\end{equation*}
which is exactly the product as defined in \cite{ziwenyao}, extended to the wrapped setting.

\subsection{Section-counting maps}

For the purposes of illustration, we will use a dashed curve on a figure representing a Riemann surface to indicate any boundary marked points that are mapped to an intersection between $L_i$ and $\phi^d(L_{i+1})$, or cylindrical ends that are asymptotic to an orbit of $\phi^d$. This is purely illustrative: our surfaces contain no seams or cuts (though the analysis could alternatively be set up this way), and they do not indicate that marked points must lie on the same line. Interior critical values of the Lefschetz fibration are denoted by solid dots; domains biholomorphic to $D$ have solid boundary, while those biholomorphic to $\CC$ have dashed boundary.

\begin{definition}\label{def:seidel class}
Given a standard symplectic Lefschetz fibration $\pi: E \to \CC$ with global monodromy $\phi$, and a choice of compatible perturbation data $(K,J)$, we define a moduli space $\scr{M}(E,J,K,p)$ to be the set of $K$-perturbed $J$-holomorphic sections of the Lefschetz fibration $\pi: E \to \CC$ that are horizontally $C^1$-asymptotic to the orbit of $\phi$ starting at $p$ (in the sense that $u(r \e{i t})$ converges to $\phi_t(p)$ in $C^1(M_{r})$ as $r \to \infty$). 

For generically chosen compatible perturbation data $(K, J)$, this moduli space is a topological manifold (see \cite[p.237]{seidel}) and compact when the dimension is zero. In this case we define the \textbf{Seidel element} to be $S \in \mathrm{CF}(\phi)$ given by the count of dimension zero moduli spaces,
\begin{equation*}
    S = \sum_{p \in \mathrm{Fix}(\phi)} \# \scr{M}(E,J, K,p)\; [p]
\end{equation*}
with canonically determined signs. In the wrapped case we equip $\CC$ with a weight-$1$ positive cylindrical end.
\end{definition}

One may verify using standard methods that this is indeed a cocycle, and we call the result the \textbf{Seidel class}, though this could  be called a special kind of Borman-Sheridan class.   

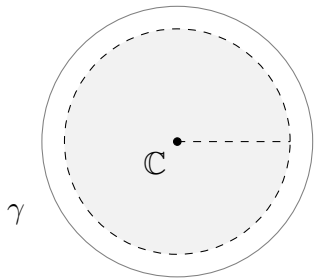
\begin{figure}
\begin{center}
\begin{tikzpicture}[scale=1]
\filldraw[fill=gray!10,dashed] (0,0) circle (1.5);
\draw[dashed] (0,0) to (1.5,0);
\draw[fill=black] (0,0) circle (0.05);
\draw[gray,->] (0,0) circle (1.8);
\node[anchor=north east] at (200:2) {$\gamma$};
\node[anchor=north east] at (0,0) {$\CC$};
\end{tikzpicture}
\caption{Domains used to define Seidel element.}
\label{fig:seidel_element}
\end{center}
\end{figure}

\begin{remark}
    Because of the non-uniqueness of standard Lefschetz fibrations, this class may depend on the choice of standard Lefschetz fibration for $\phi$ when $\phi$ is not a single Dehn twist.
\end{remark}

We have a similar construction in the open sector: given a spin exact cylindrical Lagrangian $L$ and a vanishing cycle $V$ inside a Liouville manifold $M$, Seidel in \cite[(17d)]{seidel} defines a cocycle $s \in \mathrm{CF}(L, \phi_V(L))$ via a count of sections. The arguments in \cite[\S 17]{seidel} essentially show that this extends to a degree-$0$ natural transformation $\mathrm{id} \to \phi_V$ between $A_{\infty}$-functors on the Fukaya category $\scr{F}(M)$, and hence an element of $\mathrm{HH}^{0}(\scr{F}(M),\phi_V)$ which we call \textbf{Seidel's natural transformation}. Seidel's construction applies more generally: a count of suitably perturbed sections of a standard Lefschetz fibration for $\phi$ over a domain (modulo reparametrization) such as in Figure \ref{fig:seidel_transformation} defines a map:
\begin{equation*}
    s_k: \mathrm{CF}(L_{k-1},L_k)\otimes \cdots \otimes \mathrm{CF}(L_0,L_1) \to \mathrm{CF}(L_0,\phi(L_k))
\end{equation*}
which gives a term of a natural transformation $s:\mathrm{id} \to \phi$.

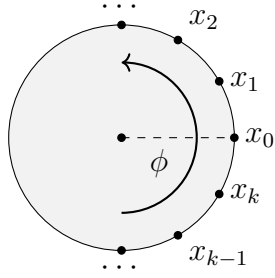
\begin{figure}
\begin{center}
\begin{tikzpicture}[scale=0.5]
\filldraw[fill=gray!10] (0,0) circle (3);
\draw[dashed] (0,0) to (3,0);
\draw[fill=black] (0,0) circle (0.1);
\draw[fill=black] (3,0) circle (0.1);
\node[anchor=west] at (3,0) {$x_0$};
\draw[fill=black] (30:3) circle (0.1);
\node[anchor=west] at (30:3) {$x_1$};
\draw[fill=black] (60:3) circle (0.1);
\node[anchor=south west] at (60:3) {$x_{2}$};
\draw[fill=black] (90:3) circle (0.1);
\node[anchor=south] at (90:3) {$\cdots$};
\draw[fill=black] (270:3) circle (0.1);
\node[anchor=north] at (270:3) {$\cdots$};
\draw[fill=black] (300:3) circle (0.1);
\node[anchor=north west] at (300:3) {$x_{k-1}$};
\draw[fill=black] (330:3) circle (0.1);
\node[anchor=west] at (330:3) {$x_k$};
\draw[thick,->] (-90:2) arc[start angle=-90, end angle=90,radius=2];
\node[anchor=north] at (1,0) {$\phi$};
\end{tikzpicture}
\caption{Domains used to define Seidel's natural transformation.}
\label{fig:seidel_transformation}
\end{center}
\end{figure}

Likewise, counting sections of a standard symplectic fibration over domains as in Figure \ref{fig:bimodule}:
\begin{equation*}
    \mu_{\phi}^{k,\ell}(x_k, x_{k-1},\dots,x_1,y_1,x_{1}\dash, \dots, x_{\ell}\dash) = \sum_{y_0 \in L_{\ell}\dash \cap \phi(L_k)} \#\scr{M}_{k,\ell}(E_{k,\ell}, \phi, J, K; p, x_k, \dots, x_1,y_1, x_{1}\dash,\dots,x_{\ell}\dash,y_0) \; [y_0]
\end{equation*}
with canonically determined signs, defines a term of the bimodule structure on $\Gamma_\phi$:
\begin{equation*}
    \mu_{\phi}^{k,\ell}:\mathrm{CF}(L_{k-1},L_k)\otimes \cdots \otimes \mathrm{CF}(L_0,L_1)\otimes
\mathrm{CF}(L\dash_0,\phi(L_0))\otimes \mathrm{CF}(L\dash_{0},L\dash_1)\otimes \cdots \otimes \mathrm{CF}(L_{\ell}\dash,L_{\ell-1}\dash)
\to \mathrm{CF}(L\dash_{\ell},\phi(L_k))
\end{equation*}

\begin{figure}
\begin{center}
\begin{tikzpicture}[scale=0.5]
\filldraw[fill=gray!10] (0,0) circle (3);
\draw[dashed] (-3,0) to (3,0);
\draw[fill=black] (3,0) circle (0.1);
\draw[fill=black] (-3,0) circle (0.1);
\node[anchor=west] at (3,0) {$y_0$};
\draw[fill=black] (30:3) circle (0.1);
\node[anchor=west] at (30:3) {$x_{\ell}\dash$};
\draw[fill=black] (60:3) circle (0.1);
\node[anchor=south west] at (60:3) {$x_{\ell-1}\dash$};
\draw[fill=black] (-60:3) circle (0.1);
\node[anchor=north west] at (-60:3) {$x_{k-1}$};
\draw[fill=black] (90:3) circle (0.1);
\node[anchor=south] at (90:3) {$\cdots$};
\draw[fill=black] (270:3) circle (0.1);
\node[anchor=north] at (270:3) {$\cdots$};
\draw[fill=black] (300:3) circle (0.1);
\node[anchor=south east] at (160:3) {$x_{1}\dash$};
\node[anchor=east] at (-3,0) {$y_1$};
\node[anchor=north east] at (200:3) {$x_{1}$};
\draw[fill=black] (160:3) circle (0.1);
\draw[fill=black] (200:3) circle (0.1);
\draw[fill=black] (330:3) circle (0.1);
\node[anchor=west] at (330:3) {$x_{k}$};
\draw[thick,->] (-90:2) to (90:2);
\node[anchor=north] at (1,0) {$\phi$};
\end{tikzpicture}
\caption{Domains used to define bimodule structure for $\phi$.}
\label{fig:bimodule}
\end{center}
\end{figure}
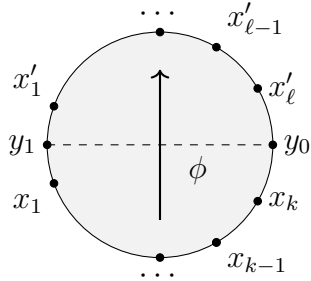

\subsection{Twisted closed-open maps}

To relate the Seidel class to the Seidel natural transformation we want to consider a twisted closed-open map
\begin{equation*}
    \scr{CO}_{\phi}: \mathrm{HF}^{\ast}(\phi) \to \mathrm{HH}^{\ast}(\scr{F}(M), \Gamma_\phi)
\end{equation*}
where $\Gamma_\phi$ is the $A_{\infty}$-bimodule induced by the symplectomorphism $\phi$ of a Liouville manifold $M$ (see \cite{seidel}), and $\mathrm{HH}^{\ast}(\scr{F}(M), \Gamma_\phi)$ denotes the Hochschild cohomology of $\scr{F}(M)$ with coefficients in this bimodule.
 
\begin{definition}\label{def:twisted_co}
Let $\scr{Q}_k$ be the moduli space of closed disks $S_k$ with $k$ negative boundary points $p_1,\dots,p_k$ and one positive boundary marked point $p_0$ fixed at $1$; as well as an interior negative puncture fixed at $0$, equipped with a negative cylindrical end. For each such $S_k$, we equip it with twisting data given by $1$ at the interior puncture, and $0$ over every boundary component except that between $p_k$ and $p_0$. Then we fix a standard symplectic fiber bundle $\pi: E_k \to S_k$ with fiber $M$, compatible with ends and twisting data.

Each of these we equip with choices of Floer data and Lagrangian labels $L_0,\dots,L_k \subset \Sigma$, modified so that there are counterclockwise moving boundary conditions along the boundary segment between $p_k$ and $p_0$ given by the isotopy $\phi$ (cf. \cite[p.244]{seidel}).

Given choices of compatible perturbation data $(K,J)$, let $\scr{M}_k(E_k, \phi, J, K, p,x_k, \dots, x_1,x_0)$ denote the moduli space of $K$-perturbed $J$-holomorphic sections $u: S_k \to E_k$ over some domain $S_k \in \scr{Q}_k$, satisfying Lagrangian boundary conditions along $\tilde{L}_0, \dots \tilde{L}_k$ (the parallel transport of $L_0, \dots, L_k$ along $\partial S_k$), that are horizontally asymptotic to the orbit $p$ of $\phi$ around $0$. 

For generic consistent perturbation data $(K,J)$, these moduli spaces are topological manifolds, compact when the dimension is zero, and we may define the \textbf{twisted closed-open map} $\scr{CO}_{\phi}: \mathrm{CF}^{\ast}(\phi) \to \mathrm{CC}^{\ast}(\scr{F}(M), \Gamma_\phi)$ as follows (cf. \cite[p.66]{ganatra_thesis}).

Given a fixed point $p \in \mathrm{Fix}(\phi)$ and morphisms in $\scr{F}(M)$ given by $x_i \in \mathrm{CF}(L_{i-1}, L_i)$, we define a Hochschild cochain $\scr{CO}_{\phi}(p)$ in 
\begin{equation*}
    \mathrm{CC}^{\ast}(\scr{F}(M), \Gamma_\phi) = \prod_{L_0, \dots, L_k} \mathrm{Hom}(\mathrm{CF}(L_{k-1},L_k) \otimes \cdots \otimes \mathrm{CF}(L_0, L_1), \mathrm{CF}(L_0, \phi(L_k)))
\end{equation*}
via a sum over $k \geq 0$ of dimension-zero moduli spaces:
\begin{equation*}
    \scr{CO}_{\phi}^k(p)(x_k, \dots, x_1) = \sum_{x_0 \in L_k \cap \phi(L_0)} \# \scr{M}_k(E_k, \phi, J, K; p, x_k, \dots, x_1, x_0) \; [x_0]
\end{equation*}
with their canonically determined signs, as illustrated in Figure \ref{fig:open-closed} below.
\end{definition}

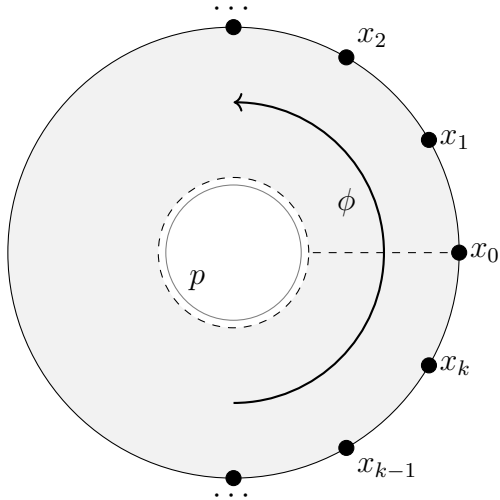
\begin{figure}[H]
\begin{center}
\begin{tikzpicture}[scale=0.5]
\begin{scope}[scale=2]
\filldraw[fill=gray!10] (0,0) circle (3);
\draw[dashed] (0,0) to (3,0);
\draw[fill=white,dashed] (0,0) circle (1);
\draw[fill=black] (3,0) circle (0.1);
\node[anchor=west] at (3,0) {$x_0$};
\draw[fill=black] (30:3) circle (0.1);
\node[anchor=west] at (30:3) {$x_1$};
\draw[fill=black] (60:3) circle (0.1);
\node[anchor=south west] at (60:3) {$x_{2}$};
\draw[fill=black] (90:3) circle (0.1);
\node[anchor=south] at (90:3) {$\cdots$};
\draw[fill=black] (270:3) circle (0.1);
\node[anchor=north] at (270:3) {$\cdots$};
\draw[fill=black] (300:3) circle (0.1);
\node[anchor=north west] at (300:3) {$x_{k-1}$};
\draw[fill=black] (330:3) circle (0.1);
\node[anchor=west] at (330:3) {$x_k$};
\draw[thick,->] (-90:2) arc[start angle=-90, end angle=90,radius=2];
\node[anchor=north] at (1.5,1) {$\phi$}; 
\end{scope}
\draw[gray,->] (0,0) circle (1.8);
\node[anchor=north east] at (200:0.5) {$p$};
\end{tikzpicture}
\caption{Domains used to define twisted closed-open map.}
\label{fig:open-closed}
\end{center}
\end{figure}

Again, one can show using standard methods that this gives a chain map and so descends to a map $\scr{CO}_{\phi}: \mathrm{HF}^{\ast}(\phi) \to \mathrm{HH}^{\ast}(\scr{F}(M), \Gamma_\phi)$.

\begin{theorem}\label{thm:S_equals_S} Given a choice of standard Lefschetz fibration for $\phi$, Seidel's natural transformation $s \in \mathrm{HH}^{\ast}(\scr{F}(M),\phi)$ is the image under the twisted closed-open map $\scr{CO}_{\phi}$ of the Seidel class $S \in \mathrm{HF}(\phi)$.
\end{theorem}
\begin{proof} This argument is a fairly straightforward combination of the compactness arguments of \cite[\S 4]{ganatra_thesis} applied to the section-counting maps of \cite[\S 17]{seidel}. Let $\scr{Q}_k$ be the moduli space of closed disks $D_k$ with $k$ incoming boundary marked points $p_i$, one outgoing marked point $p_0$ fixed at $1$, and a distinguished point at $0$ which will be the critical value of a Lefschetz fibration. We construct a family of standard Lefschetz fibrations parametrized by $r \in (0,1]$, living over each domain $S_k \in \scr{Q}_k$, with Lefschetz fibration $E_r$ pulled back from the chosen standard Lefschetz fibration $E \to \CC$ for $\phi$ by the map $z \to z/r$ for $r \neq 0$. This has the effect of flattening the Lefschetz fibration over the complement of $D_{r(1-\epsilon)}$. Denote the extended moduli space of sections of this Lefschetz fibration (with appropriate perturbation data $(J,K)$) by $\tilde{\scr{M}}_k(E_r, \phi, J, K, x_0, x_1, \dots, x_k)$, consisting of pairs $(u,r)$ where $u: D_{k} \to E_r$ is a $K$-perturbed $J$-holomorphic section of the Lefschetz fibration $E_r$ over $S_k \in \scr{Q}_k$ satisfying the boundary conditions described in Definition \ref{def:twisted_co} above. After choosing suitably generic perturbation data, when the dimension is zero, counting elements of this moduli space yields a map:
\begin{equation*}
    T^k: \mathrm{CF}(L_{k-1},L_k) \otimes \cdots \otimes \mathrm{CF}(L_0, L_1) \to \mathrm{CF}(L_0, \phi(L_k))
\end{equation*}
of degree $-k-1$, giving an element in $CC^{-1}(\scr{F}(M), \phi)$.

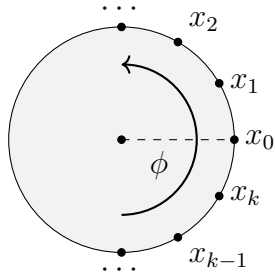
\begin{figure}[H]
\begin{center}
\begin{tikzpicture}[scale=0.5]
\filldraw[fill=gray!10] (0,0) circle (3);
\draw[dashed] (0,0) to (3,0);
\draw[fill=black] (0,0) circle (0.1);
\draw[fill=black] (3,0) circle (0.1);
\node[anchor=west] at (3,0) {$x_0$};
\draw[fill=black] (30:3) circle (0.1);
\node[anchor=west] at (30:3) {$x_1$};
\draw[fill=black] (60:3) circle (0.1);
\node[anchor=south west] at (60:3) {$x_{2}$};
\draw[fill=black] (90:3) circle (0.1);
\node[anchor=south] at (90:3) {$\cdots$};
\draw[fill=black] (270:3) circle (0.1);
\node[anchor=north] at (270:3) {$\cdots$};
\draw[fill=black] (300:3) circle (0.1);
\node[anchor=north west] at (300:3) {$x_{k-1}$};
\draw[fill=black] (330:3) circle (0.1);
\node[anchor=west] at (330:3) {$x_k$};
\draw[thick,->] (-90:2) arc[start angle=-90, end angle=90,radius=2];
\node[anchor=north] at (1,0) {$\phi$};
\end{tikzpicture}
\caption{Domains used to define the Hochschild cochain $T$.}
\end{center}
\end{figure}

Taking the Gromov compactification of $\tilde{\scr{M}}_k(E,\phi,J,K,x_0,x_1,\dots,x_k)$ gives a manifold fibered over $[0,1]$, and the codimension-$1$ boundary is covered by the union of the images of the natural inclusions of the moduli spaces of the form below (see \cite{ganatra_thesis,seidel} for this kind of diagrammatic proof). As well as taking our perturbation data to be universal and consistent under gluing, one must verify that the Lefschetz fibrations constructed are consistent under gluing of components inside the boundary strata of Deligne-Mumford moduli space (see \cite[p.27]{seidel_LES} for the notion of gluing Lefschetz fibrations and its effect on section-counting maps). This poses no difficulty in our case since Lefschetz fibrations with a single critical point are essentially unique in a strong sense. 
\begin{enumerate}
    \item At $r=0$ we have Figure \ref{fig:part_1}, which represents the closed-open map applied to the Seidel class; 
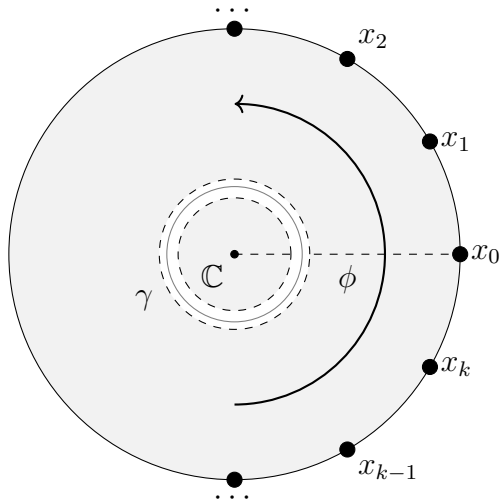
\begin{figure}[H]
\begin{center}
\begin{tikzpicture}[scale=0.5]
\begin{scope}[scale=2]
\filldraw[fill=gray!10] (0,0) circle (3);
\draw[dashed] (0,0) to (3,0);
\draw[fill=white,dashed] (0,0) circle (1);
\draw[fill=black] (3,0) circle (0.1);
\node[anchor=west] at (3,0) {$x_0$};
\draw[fill=black] (30:3) circle (0.1);
\node[anchor=west] at (30:3) {$x_1$};
\draw[fill=black] (60:3) circle (0.1);
\node[anchor=south west] at (60:3) {$x_{2}$};
\draw[fill=black] (90:3) circle (0.1);
\node[anchor=south] at (90:3) {$\cdots$};
\draw[fill=black] (270:3) circle (0.1);
\node[anchor=north] at (270:3) {$\cdots$};
\draw[fill=black] (300:3) circle (0.1);
\node[anchor=north west] at (300:3) {$x_{k-1}$};
\draw[fill=black] (330:3) circle (0.1);
\node[anchor=west] at (330:3) {$x_k$};
\draw[thick,->] (-90:2) arc[start angle=-90, end angle=90,radius=2];
\node[anchor=north] at (1.5,0) {$\phi$}; 
\end{scope}
\filldraw[fill=gray!10,dashed] (0,0) circle (1.5);
\draw[dashed] (0,0) to (1.5,0);
\draw[fill=black] (0,0) circle (0.1);
\draw[gray,<-] (0,0) circle (1.8);
\node[anchor=north east] at (200:2) {$\gamma$};
\node[anchor=north east] at (0,0) {$\CC$};
\end{tikzpicture}
\end{center}
\caption{Domains representing the closed-open map applied to the Seidel element}
\label{fig:part_1}
\end{figure}
    
    \item At $r=1$, we have Figure \ref{fig:part_2}, which represents Seidel's natural transformation;

\begin{figure}[H]
\begin{center}
\begin{tikzpicture}[scale=0.5]
\filldraw[fill=gray!10] (0,0) circle (3);
\draw[dashed] (0,0) to (3,0);
\draw[fill=black] (0,0) circle (0.1);
\draw[fill=black] (3,0) circle (0.1);
\node[anchor=west] at (3,0) {$x_0$};
\draw[fill=black] (30:3) circle (0.1);
\node[anchor=west] at (30:3) {$x_1$};
\draw[fill=black] (60:3) circle (0.1);
\node[anchor=south west] at (60:3) {$x_{2}$};
\draw[fill=black] (90:3) circle (0.1);
\node[anchor=south] at (90:3) {$\cdots$};
\draw[fill=black] (270:3) circle (0.1);
\node[anchor=north] at (270:3) {$\cdots$};
\draw[fill=black] (300:3) circle (0.1);
\node[anchor=north west] at (300:3) {$x_{k-1}$};
\draw[fill=black] (330:3) circle (0.1);
\node[anchor=west] at (330:3) {$x_k$};
\draw[thick,->] (-90:2) arc[start angle=-90, end angle=90,radius=2];
\node[anchor=north] at (1,0) {$\phi$};
\end{tikzpicture}
\end{center}
\caption{Domains used to define the Seidel natural transformation.}
\label{fig:part_2}
\end{figure}
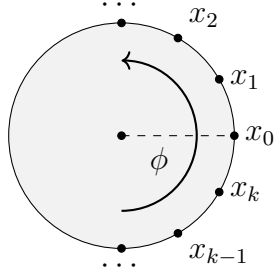

    \item We have the Deligne-Mumford degenerations in Figures \ref{fig:part_3} and \ref{fig:part_4} of $S_k$ which represent (respectively) the first and second terms of the Hochschild coboundary of $T$:

\begin{figure}[H]
    \begin{center}
\begin{tikzpicture}[scale=0.5]
\filldraw[fill=gray!10] (0,0) circle (3);
\begin{scope}[shift={(-5,0)}]
\filldraw[fill=gray!10] (0,0) circle (2);
\draw[fill=black] (0:2) circle (0.1);
\draw[fill=black] (60:2) circle (0.1);
\node[anchor=south west] at (60:2) {$x_i$};
\draw[fill=black] (120:2) circle (0.1);
\node[anchor=south] at (120:2) {$x_{i+1}$};
\draw[fill=black] (-60:2) circle (0.1);
\node[anchor=north] at (-120:2) {$x_{i+\ell-1}$};
\draw[fill=black] (-120:2) circle (0.1);
\node[anchor=north west] at (-60:2) {$x_{i+\ell}$};
\node[anchor=east] at (180:2) {$\vdots$};
\end{scope}
\draw[dashed] (0,0) to (3,0);
\draw[fill=black] (0,0) circle (0.1);
\draw[fill=black] (3,0) circle (0.1);
\node[anchor=west] at (3,0) {$x_0$};
\draw[fill=black] (30:3) circle (0.1);
\node[anchor=west] at (30:3) {$x_1$};
\draw[fill=black] (60:3) circle (0.1);
\node[anchor=south west] at (60:3) {$x_{2}$};
\draw[fill=black] (90:3) circle (0.1);
\node[anchor=south] at (90:3) {$\cdots$};
\draw[fill=black] (270:3) circle (0.1);
\node[anchor=north] at (270:3) {$\cdots$};
\draw[fill=black] (300:3) circle (0.1);
\node[anchor=north west] at (300:3) {$x_{k-1}$};
\draw[fill=black] (330:3) circle (0.1);
\node[anchor=west] at (330:3) {$x_k$};
\draw[thick,->] (-90:2) arc[start angle=-90, end angle=90,radius=2];
\node[anchor=north] at (1,0) {$\phi$};
\end{tikzpicture}
\end{center}
\caption{Domains representing $T$ composed with the $A_{\infty}$ operations.}
\label{fig:part_3}
\end{figure}
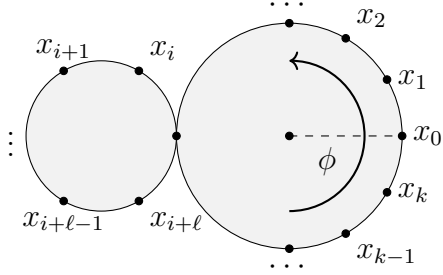

    \begin{figure}[H]
    \begin{center}
\begin{tikzpicture}[scale=0.5]
\filldraw[fill=gray!10] (0,0) circle (3);
\begin{scope}[shift={(5,0)}]
\filldraw[fill=gray!10] (0,0) circle (2);
\draw[fill=black] (0:2) circle (0.1);
\draw[fill=black] (60:2) circle (0.1);
\draw[fill=black] (120:2) circle (0.1);
\draw[fill=black] (-60:2) circle (0.1);
\draw[fill=black] (-120:2) circle (0.1);
\node[anchor=west] at (2,0) {$x_0$};
\node[anchor=west] at (30:2) {$x_1$};
\node[anchor=south west] at (60:2) {$x_{2}$};
\node[anchor=north west] at (300:2) {$x_{k-1}$};
\node[anchor=west] at (330:2) {$x_k$};
\node[anchor=north] at (-120:2) {$x_{i+ \ell + 1}$};
\node[anchor=south] at (120:2) {$x_{i-1}$};
\node[anchor=north] at (-80:2.2) {$\cdots$};
\node[anchor=south] at (80:2.2) {$\cdots$};
\end{scope}
\draw[dashed] (0,0) to (7,0);
\draw[fill=black] (0,0) circle (0.1);
\draw[fill=black] (3,0) circle (0.1);
\draw[fill=black] (30:3) circle (0.1);
\draw[fill=black] (60:3) circle (0.1);
\draw[fill=black] (90:3) circle (0.1);
\node[anchor=south] at (90:3) {$\cdots$};
\draw[fill=black] (270:3) circle (0.1);
\node[anchor=north] at (270:3) {$\cdots$};
\draw[fill=black] (300:3) circle (0.1);
\draw[fill=black] (330:3) circle (0.1);
\draw[thick,->] (-90:2) arc[start angle=-90, end angle=90,radius=2];
\node[anchor=north] at (1,0) {$\phi$};
\end{tikzpicture}
\end{center}
\caption{Domains representing the $A_{\infty}$-bimodule operations applied to $T$.}
\label{fig:part_4}
\end{figure}
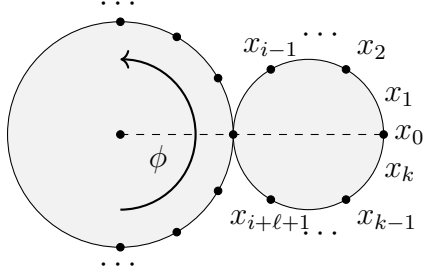

    Recall that the Hochschild differential is given by:
\begin{align*}
    d_{CC^{\ast}}(T)(x_k \otimes \cdots \otimes x_1) = \sum_{i=1}^{k}\sum_{\ell=0}^{k-i} (-1)^{\maltese^{i-1}} T(x_k, \dots, x_{i+\ell+1},\mu^{\ell+1}(x_{i+ \ell}, \dots, x_i), x_{i-1}, \dots, x_1) \\ + \sum_{i=1}^k \sum_{\ell=0}^{k-i}(-1)^{|T|(\maltese^{i-1}+1)+1} \mu_{\phi}^{k-i-\ell,i-1}(x_k, \dots, x_{i+\ell+1}, T(x_{i+\ell}, \dots, x_i), x_{i-1}, \dots, x_1)
\end{align*}
where $\mu_{\phi}^d$ denotes the $A_{\infty}$ bimodule operations for $\phi$ and $\maltese^j = |x_0| + \cdots + |x_j| - j$.

\item Finally, we have the strip-breaking degenerations in Figure \ref{fig:part_5}, which represent the remaining part of the Hochschild coboundary of $T$:
\begin{equation*}
    \sum_{i=1}^k (-1)^{|T|(\maltese^{i-1} +1)+1}T^k(x_k, \dots, \mu^1(x_i), \dots, x_1) + (-1)^{k+1} \mu^{1}_\phi(T^{k}(x_k, \dots, x_1))
\end{equation*}

\begin{figure}[H]
\begin{center}
\begin{tikzpicture}[scale=0.5]
\filldraw[fill=gray!10] (0,0) circle (3);
\begin{scope}[shift={(-4,0)},scale=0.5]
\filldraw[fill=gray!10] (0,0) circle (2);
\draw[fill=black] (0:2) circle (0.2);
\draw[fill=black] (180:2) circle (0.2);
\node[anchor=east] at (180:2) {$x_i$};
\end{scope}
\draw[dashed] (0,0) to (3,0);
\draw[fill=black] (0,0) circle (0.1);
\draw[fill=black] (3,0) circle (0.1);
\node[anchor=west] at (3,0) {$x_0$};
\draw[fill=black] (30:3) circle (0.1);
\node[anchor=west] at (30:3) {$x_1$};
\draw[fill=black] (60:3) circle (0.1);
\node[anchor=south west] at (60:3) {$x_{2}$};
\draw[fill=black] (90:3) circle (0.1);
\node[anchor=south] at (90:3) {$\cdots$};
\draw[fill=black] (270:3) circle (0.1);
\node[anchor=north] at (270:3) {$\cdots$};
\draw[fill=black] (300:3) circle (0.1);
\node[anchor=north west] at (300:3) {$x_{k-1}$};
\draw[fill=black] (330:3) circle (0.1);
\node[anchor=west] at (330:3) {$x_k$};
\draw[thick,->] (-90:2) arc[start angle=-90, end angle=90,radius=2];
\node[anchor=north] at (1,0) {$\phi$};
\end{tikzpicture}
\end{center}
\caption{Strip breaking from $T$.}
\label{fig:part_5}
\end{figure}
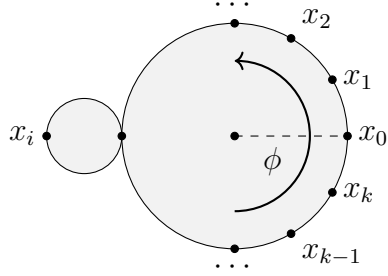
\end{enumerate}
If one takes appropriate sign twisting data and orients the moduli space $\scr{Q}_k$ as in \cite{ganatra_thesis}, one can see that the sum of the number of boundary points of the Gromov compactification of $\tilde{\scr{M}}_k(E,\phi,J,K,x_0,\dots,x_k)$, taken with respect to their canonical orientations, agrees with the signs of the Hochschild differential and $A_{\infty}$ operations. Since this count is zero, this shows that $\scr{CO}_{\phi}(S) = s + d_{CC^{\ast}}(T)$, so they represent the same class in Hochschild cohomology (with $\ZZ$-coefficients).
\end{proof}

\begin{theorem}\label{thm:product_agrees}
The twisted closed-open map $\scr{CO}_{\phi}$ respects the natural product operations on fixed point Floer homology from \cite{ziwenyao}:
\begin{center}
\begin{tikzcd}
    \mathrm{HF}^\ast(\phi^i) \otimes \mathrm{HF}^\ast(\phi^j) \ar{r} \ar{d}{\scr{CO}_{\phi^i} \otimes \scr{CO}_{\phi^j}} & \mathrm{HF}^\ast(\phi^{i+j}) \ar{d}{\scr{CO}_{\phi^{i+j}}} \\
    \mathrm{HH}^{\ast}(\Gamma_{\phi^i}) \otimes \mathrm{HH}^{\ast}(\Gamma_{\phi^j}) \ar{r} & \mathrm{HH}^{\ast}(\Gamma_{\phi^{i+j}})
\end{tikzcd}
\end{center}
\end{theorem}
\begin{proof}
Again, the proof is analogous to \cite[\S 4]{ganatra_thesis}, where we consider a moduli space of parametrized domains. Let $D_k$ be a closed disk $k$ incoming boundary marked points $p_1,\dots,p_k$ and one outgoing marked point $p_0$ fixed at $1$, and with two negative interior punctures, lying on the imaginary axis at a distance of $r \in (0,1)$, see Figure \ref{fig:parametrized_domains}. As well as strip-like and cylindrical ends, $D_k$ also comes with twisting data: equal to $i$ and $j$ for the interior punctures, equal to $i+j$ on the boundary component between $p_k$ and $p_0$, and zero otherwise. For each $r \in (0,1)$, for each $D_k$ inside the moduli space $\scr{Q}_{k}^r$ of such disks, take $\pi:E_r \to D_k$ to be a standard symplectic fiber bundle (with respect to these ends and twisting data). Consider the extended moduli space of $K$-perturbed $J$-holomorphic sections $\tilde{\scr{M}}_{i,j,k}(E_k, J, K, x_0, x_1, \dots, x_k,p_1,p_2)$, consisting of pairs $(u,r)$ where $u: D_{k} \to E_r$ is a $K$-perturbed $J$-holomorphic section of $E_r$ with $D_k \in \scr{Q}_{k}^r$, satisfying the boundary and asymptotic conditions as described in Definition \ref{def:twisted_co}.

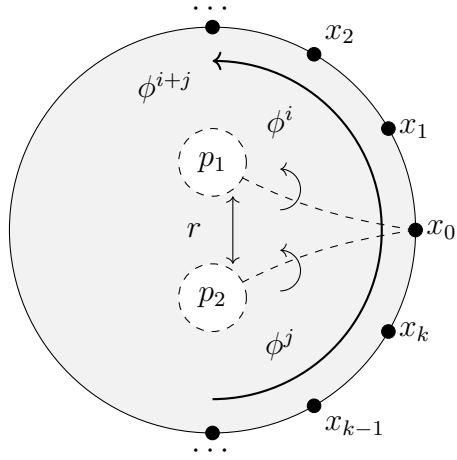
\begin{figure}[H]
\begin{center}
\begin{tikzpicture}[scale=0.9]
\filldraw[fill=gray!10] (0,0) circle (3);
\draw[dashed,bend left=10] (3,0) to (0,1);
\draw[dashed,bend right=10] (3,0) to (0,-1);
\draw[<->] (0.3,0.5) -- (0.3,-0.5);
\draw (0,0) node[anchor=east] {$r$};
\draw[fill=white,dashed] (0,1) circle (0.5);
\draw[fill=white,dashed] (0,-1) circle (0.5);
\draw[fill=black] (3,0) circle (0.1);
\node[anchor=west] at (3,0) {$x_0$};
\draw[fill=black] (30:3) circle (0.1);
\node[anchor=west] at (30:3) {$x_1$};
\draw[fill=black] (60:3) circle (0.1);
\node[anchor=south west] at (60:3) {$x_{2}$};
\draw[fill=black] (90:3) circle (0.1);
\node[anchor=south] at (90:3) {$\cdots$};
\draw[fill=black] (270:3) circle (0.1);
\node[anchor=north] at (270:3) {$\cdots$};
\draw[fill=black] (300:3) circle (0.1);
\node[anchor=north west] at (300:3) {$x_{k-1}$};
\draw[fill=black] (330:3) circle (0.1);
\node[anchor=west] at (330:3) {$x_k$};
\draw[thick,->] (-90:2.5) arc[start angle=-90, end angle=90,radius=2.5];
\node[anchor=north] at (-0.7,2.5) {$\phi^{i+j}$}; 
\node at (0,1) {$p_1$};
\node at (0,-1) {$p_2$};
\draw[->] (1,0.3) arc[start angle=-90, end angle=90,radius=0.3];
\draw (1,-1.3) node[anchor=north] {$\phi^j$};
\draw[->] (1,-0.9) arc[start angle=-90, end angle=90,radius=0.3];
\draw (1,1.2) node[anchor=south] {$\phi^i$};
\end{tikzpicture}
\end{center}
\caption{Domains used to construct the parametrized moduli space $\tilde{\scr{M}}_{i,j,k}$.}
\label{fig:parametrized_domains}
\end{figure}

As in the proof of Theorem \ref{thm:S_equals_S} above, we consider the Gromov compactification of $\tilde{\scr{M}}_{i,j,k}(E_k,J,K)$ as a manifold fibered over $[0,1]$. The components of the Gromov boundary over $r=0, r=1$ are illustrated in Figure \ref{fig:degeneration_1} and Figure \ref{fig:degeneration_2} respectively.
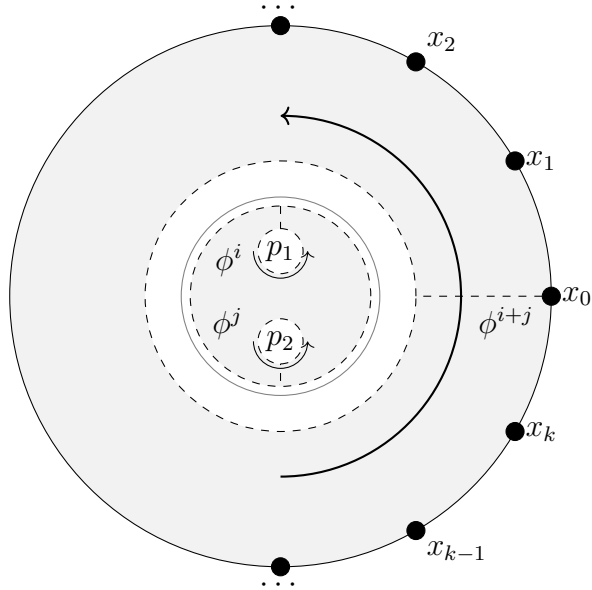
\begin{figure}[H]
\begin{center}
    \begin{tikzpicture}[scale=0.6]
\begin{scope}[scale=2]
\filldraw[fill=gray!10] (0,0) circle (3);
\draw[dashed] (0,0) to (3,0);
\draw[fill=white,dashed] (0,0) circle (1.5);
\draw[fill=black] (3,0) circle (0.1);
\node[anchor=west] at (3,0) {$x_0$};
\draw[fill=black] (30:3) circle (0.1);
\node[anchor=west] at (30:3) {$x_1$};
\draw[fill=black] (60:3) circle (0.1);
\node[anchor=south west] at (60:3) {$x_{2}$};
\draw[fill=black] (90:3) circle (0.1);
\node[anchor=south] at (90:3) {$\cdots$};
\draw[fill=black] (270:3) circle (0.1);
\node[anchor=north] at (270:3) {$\cdots$};
\draw[fill=black] (300:3) circle (0.1);
\node[anchor=north west] at (300:3) {$x_{k-1}$};
\draw[fill=black] (330:3) circle (0.1);
\node[anchor=west] at (330:3) {$x_k$};
\draw[thick,->] (-90:2) arc[start angle=-90, end angle=90,radius=2];
\node[anchor=north] at (2.5,0) {$\phi^{i+j}$}; 
\end{scope}
\filldraw[fill=gray!10,dashed] (0,0) circle (2);
\draw[gray,<-] (0,0) circle (2.2);
\begin{scope}[rotate=-90]
 \draw[dashed] (-2,0) to (-1,0);
  \draw[dashed] (1,0) to (2,0);
\draw[fill=white,dashed] (1,0) circle (0.5);
\draw[fill=white,dashed] (-1,0) circle (0.5);
\draw[->] (1,-0.6) arc[start angle=-90, end angle=90,radius=0.6];
\draw (1.2,-0.6) node[anchor=south east] {$\phi^j$};
\draw[->] (-1,-0.6) arc[start angle=-90, end angle=90,radius=0.6];
\draw (-0.8,-0.6) node[anchor=east] {$\phi^i$};
\node at (-1,0) {$p_1$};
\node at (1,0) {$p_2$};   
\end{scope}
\end{tikzpicture}
\caption{A degeneration of domains in Figure \ref{fig:parametrized_domains} as $r \to 0$, corresponding to the composition of the twisted closed-open map with the product on fixed point Floer cohomology. }
\end{center}
\label{fig:degeneration_1}
\end{figure}

\begin{figure}[H]
\begin{center}
    \begin{tikzpicture}[scale=0.75]
    \begin{scope}[rotate=-90]
\filldraw[fill=gray!10] (0,0) circle (3);
\draw[dashed] (0,0) to (7,0);
\begin{scope}[shift={(5,0)}]
\filldraw[fill=gray!10] (0,0) circle (2);
\draw[dashed, bend right] (-2,0) to (0,2);
\draw[dashed, bend left] (2,0) to (0,2);
\draw[fill=black] (0:2) circle (0.1);
\draw[fill=black] (60:2) circle (0.1);
\draw[fill=black] (120:2) circle (0.1);
\draw[fill=black] (-60:2) circle (0.1);
\draw[fill=black] (-120:2) circle (0.1);
\node[anchor=west] at (0,2) {$x_0$};
\draw[fill=black] (0,2) circle (0.1);
\node[anchor=west] at (60:2) {$x_{k}$};
\node[anchor=north east] at (300:2) {$x_{k-m}$};
\node[anchor=south east] at (-120:2) {$x_{\ell + 1}$};
\node[anchor=west] at (120:2) {$x_{2}$};
\node[anchor=east] at (-90:2.2) {$\vdots$};
\draw[thick,->] (1.5,-0.2) arc[start angle=0, end angle=180,radius=1.5];
\draw (0,0) node[anchor=south] {$\phi^{i+j}$};
\end{scope}
\draw[fill=black] (0,0) circle (0.1);
\draw[fill=black] (3,0) circle (0.1);
\draw[fill=black] (30:3) circle (0.1);
\draw[fill=black] (60:3) circle (0.1);
\draw[fill=black] (90:3) circle (0.1);
\node[anchor=west] at (90:3) {$\vdots$};
\draw[fill=black] (270:3) circle (0.1);
\node[anchor=east] at (270:3) {$\vdots$};
\draw[fill=black] (300:3) circle (0.1);
\draw[fill=black] (330:3) circle (0.1);
\draw[thick,->] (-90:2) arc[start angle=-90, end angle=90,radius=2];
\node[anchor=east] at (1.5,0) {$\phi^i$};
\draw[dashed] (0,0) to (1.5,0);
\draw[fill=black] (0,0) circle (0.1);
\draw[gray,<-] (0,0) circle (1);
\draw[fill=white,dashed] (0,0) circle (1);
\node[anchor=north east] at (0,0) {$p_1$};
\begin{scope}[shift={(10,0)}]
\filldraw[fill=gray!10] (0,0) circle (3);
\draw[dashed] (0,0) to (-3,0);
\draw[fill=black] (-180:3) circle (0.1);
\draw[fill=black] (0,0) circle (0.1);
\draw[fill=black] (3,0) circle (0.1);
\draw[fill=black] (30:3) circle (0.1);
\draw[fill=black] (60:3) circle (0.1);
\draw[fill=black] (90:3) circle (0.1);
\node[anchor=west] at (90:3) {$\vdots$};
\draw[fill=black] (270:3) circle (0.1);
\node[anchor=east] at (270:3) {$\vdots$};
\draw[fill=black] (300:3) circle (0.1);
\draw[fill=black] (330:3) circle (0.1);
\draw[thick,<-] (90:2) arc[start angle=90, end angle=-90,radius=2];
\node at (1.5,0) {$\phi^j$};
\draw[fill=black] (0,0) circle (0.1);
\draw[gray,<-] (0,0) circle (1);
\draw[fill=white,dashed] (0,0) circle (1);
\node[anchor=north east] at (0,0) {$p_2$};
\end{scope}
\end{scope}
\end{tikzpicture}
\caption{A degeneration of domains in Figure \ref{fig:parametrized_domains} as $r \to 1$, corresponding to the product in twisted Hochschild cohomology applied to the twisted closed-open map.}
\end{center}
\label{fig:degeneration_2}
\end{figure}
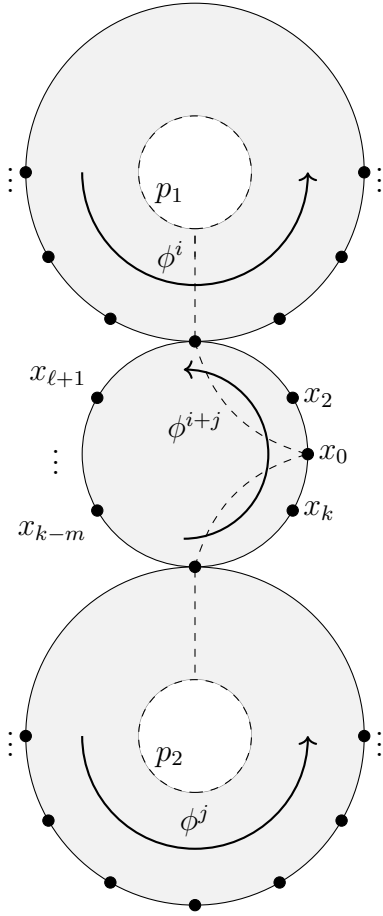
There are also additional components of the Gromov boundary coming from cylinder breaking, corresponding to the the differential of $\mathrm{CF}^{\ast}(\phi^i)$, and disk bubbling and strip breaking, corresponding to the differential of $\mathrm{CC}^{\ast}(\Gamma_{\phi^i})$. When $(K,J)$ are chosen generically and the virtual dimension is zero, the count of points in this moduli space 
provides a chain map:
\begin{equation*}
    h: \mathrm{CF}^{\ast}(\phi^i) \otimes \mathrm{CF}^{\ast}(\phi^j) \to \mathrm{CC}^{\ast}(\Gamma_{\phi^{i+j}})
\end{equation*}
via
\begin{equation*}
    h(p_0,p_1,x_1, \dots, x_k)(x_0) = \#\tilde{\scr{M}}_{i,j,k}(E_k, J, K, x_0, x_1, \dots, x_k,p_1,p_2)
\end{equation*}
that gives a homotopy between the two multiplication operations. If appropriate sign twisting data are used, and $\scr{Q}_k$ is oriented consistently with \cite{ganatra_thesis}, then if this count is taken with its canonically determined signs, the orientations of the boundary strata of the Gromov compactification agree with the orientations for the Hochschild differential and $A_{\infty}$-operations, and this chain homotopy also holds with $\CC$ coefficients.
\end{proof}

\subsubsection*{Lefschetz fibrations}
We note our product formula for fixed point Floer cohomology can be used to compute the Seidel class in more complicated Lefschetz fibrations, which may be of independent interest, \textit{though it is not essential for the rest of the paper}. We give an illustration below. 
\begin{theorem}
Given a single Dehn twist $\phi$, the Seidel element $S_2 \in \mathrm{HF}^0(\phi^2)$ associated to the standard Lefschetz fibration with two critical points constructed in Proposition \ref{prop:two_crit_points} with and $V_{-} = V_{+}$, is equal to $S^2 \in \mathrm{HF}^0(\phi^2)$ where $S \in \mathrm{HF}^0(\phi)$ is the Seidel element for $\phi$. 
\end{theorem}
\begin{proof}
The proof of this result is analogous to Theorem \ref{thm:S_equals_S}, where now we instead stretch the complex structure along two circles, around the two critical values in the base. We construct a family of Lefschetz fibrations $E_r \to \CC$ with strip-like ends living over $\CC$ parametrized by $r \in (0,1)$, and look at sections asymptotic to a given orbit $p$ of $\phi^2$ at infinity: this is given by taking the construction from Proposition \ref{prop:two_crit_points} of the standard Lefschetz fibration with two critical points and instead gluing in at the punctures at $\pm 1$ the Lefschetz fibration $E_r$ pulled back from the standard Lefschetz fibration $E \to \CC$ associated to the vanishing cycle $V$ by the map $z \to z/r$ for $r \neq 0$. This has the effect of stretching the gluing region for the Lefschetz fibration over the complement of $D_{r(1-\epsilon)}$. Denote the extended moduli space of sections of this Lefschetz fibration (with generic perturbation data) by $\tilde{\scr{M}}(E_r,J, K, p)$, consisting of pairs $(u,r)$ where $u: \CC \to E_r$ is a $K$-perturbed $J$-holomorphic section of the Lefschetz fibration $E_r$, horizontally asymptotic to the orbit $p$ at infinity.

Taking the Gromov compactification of $\tilde{\scr{M}}(E_r,J, K, p)$, gives a manifold fibered over $[0,1]$, and the codimension-$1$ boundary is covered by the union of the images of the natural inclusions of the moduli spaces of $J$-holomorphic sections over the domains in Figures \ref{fig:two_crit_pts} and \ref{fig:two_crit_pts_2}.

\begin{figure}[H]
\begin{center}
\begin{tikzpicture}[scale=0.8]
\filldraw[fill=gray!10,dashed] (0,0) circle (3);
\draw[dashed] (3,0) -- (1,0);
\draw[dashed] (-3,0) -- (-1,0);
\draw[fill=white,dashed] (1,0) circle (0.75);
\draw[fill=white,dashed] (-1,0) circle (0.75);
\draw[thick,->] (-90:2.5) arc[start angle=-90, end angle=90,radius=2.5];
\node[anchor=west] at (3,0) {$\phi^{2}$}; 
\draw[->] (1,-0.65) arc[start angle=-90, end angle=90,radius=0.65];
\draw (1.2,-0.6) node[anchor=north] {$\phi$};
\draw[<-] (-1,-0.65) arc[start angle=-90, end angle=-270,radius=0.65];
\draw (-0.8,-0.6) node[anchor=north] {$\phi$};
\begin{scope}[shift={(1,0)},scale=0.3]
\filldraw[fill=gray!10,dashed] (0,0) circle (1.5);
\draw[dashed] (0,0) to (1.5,0);
\draw[fill=black] (0,0) circle (0.1);
\draw[gray,->] (0,0) circle (1.8);
\end{scope}
\begin{scope}[shift={(-1,0)},scale=0.3,rotate=180]
\filldraw[fill=gray!10,dashed] (0,0) circle (1.5);
\draw[dashed] (0,0) to (1.5,0);
\draw[fill=black] (0,0) circle (0.1);
\draw[gray,->] (0,0) circle (1.8);
\end{scope}
\end{tikzpicture}

\caption{Domains representing the product structure on $\mathrm{CF}(\phi)$ applied to the Seidel class.}
\end{center}
\label{fig:two_crit_pts}
\end{figure}
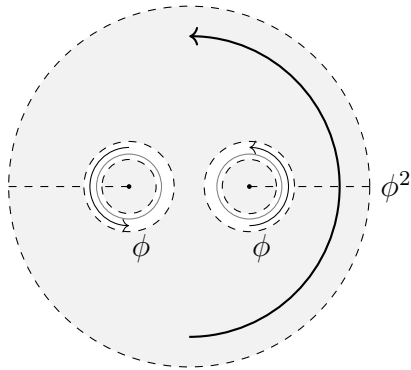

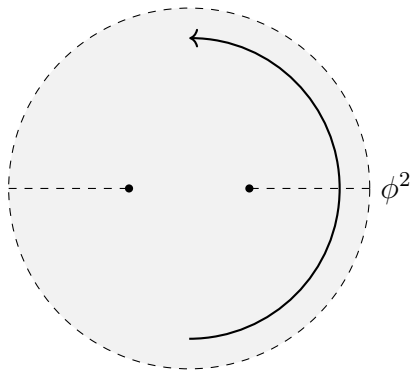
\begin{figure}[H]
\begin{center}
\begin{tikzpicture}[scale=0.8]
\filldraw[fill=gray!10,dashed] (0,0) circle (3);
\draw[dashed] (3,0) -- (1,0);
\draw[dashed] (-3,0) -- (-1,0);
\draw[thick,->] (-90:2.5) arc[start angle=-90, end angle=90,radius=2.5];
\node[anchor=west] at (3,0) {$\phi^{2}$}; 
\draw[thick,fill=black] (-1,0) circle (0.05);
\draw[thick,fill=black] (1,0) circle (0.05);
\end{tikzpicture}

\caption{Domains representing the original count of sections defining the Seidel class for $\phi^2$.}
\end{center}
\label{fig:two_crit_pts_2}
\end{figure}

Essentially by the construction in Proposition \ref{prop:two_crit_points}, the result of stretching the gluing region gives the standard symplectic fiber bundle over the pair of pants used to define the product map in fixed point Floer homology. We may also have an additional term from possible strip-breaking for the Floer differential for the fixed point Floer homology of $\phi^2$. Taking appropriate sign conventions as in \cite{seidel,ganatra_thesis}, we see that the sum of the number of boundary points of the compactification of $\tilde{\scr{M}}(p)$, taken with the correct orientation, is zero. This shows that the difference between $S^2$ and $S_2$ in $\mathrm{CF}(\phi^2)$ is a Floer coboundary.
\end{proof}

\begin{remark}
    We include this theorem as a point of independent interest. This can be used, for instance, when combined with our computation of homogeneous coordinate rings for $\phi^2$ in Theorem \ref{thm:homog}, to compute the direct limit $\dlim_d \mathrm{HF}^*(\Sigma_{g},\phi^{2d})$ when we localize around the monodromy around the Lefschetz fibration constructed in Proposition \ref{prop:two_crit_points}.
\end{remark}

\subsection{Twisted Hochschild cohomology}

In \cite{ganatra_thesis}, Ganatra constructs a split-wrapped Fukaya category $\scr{F}^2(M)$ of $M \times \bar{M}$ whose objects consist of product Lagrangians $L_i \times L_j$ and the diagonal Lagrangian $\Delta$. He moreover constructs an $A_{\infty}$-functor
\begin{equation*}
    \textbf{M}: \scr{F}^2(M) \to \scr{F}(M)-
    \text{bimod}
\end{equation*}
using a version of quilt techniques, which is full on the subcategory split-generated by product Lagrangians. Under a non-degeneracy assumption on $M$ (see \cite[Definition 1.1]{ganatra_thesis}) satisfied always for curves $\Sigma_g$, Ganatra uses methods from \cite{abouzaid_generation} to show that the diagonal Lagrangian $\Delta \subset M \times \bar{M}$ is split-generated by product Lagrangians $L_i \times L_j$ in the split-wrapped category $\scr{F}^2(M)$. This implies that  $\scr{F}(M)$ must be homologically smooth under the non-degeneracy assumption.

One can define an appropriate version of $\scr{F}^2(M)$ in which the Lagrangian correspondence $\Delta_{\phi} \subset M \times \bar{M}$ associated to $\phi$ is also an object, such that
\begin{equation*}
    \mathrm{Hom}^{\ast}_{\scr{F}^2(M)}(\Delta, \Delta_\phi) \cong \mathrm{HF}^{\ast}(\phi)
\end{equation*}
where $\mathrm{HF}(\phi)$ is an (appropriately wrapped) version of fixed point Floer cohomology. Then it is straightforward to see that $\Delta_{\phi}$ is split-generated by the same complex of product Lagrangians $L_i \times \phi(L_j)$. Since $\bf{M}$ is full on the subcategory of $\scr{F}^2(M)$ of product Lagrangians, and faithful because of its algebraic properties \cite{ganatra_thesis}, it follows that there is an isomorphism:
\begin{equation*}
    [\textbf{M}^1]: \mathrm{Hom}^{\ast}_{\scr{F}^2(M)}(\Delta, \Delta_\phi) \to \mathrm{Hom}^{\ast}_{\scr{F}(M)-\mathrm{bimod}}(\mathrm{Id},\Gamma_{\phi})
\end{equation*}

Composing this sequence of equivalences gives a twisted closed-open map $\scr{CO}_{\phi}$:
\begin{equation*}
    \scr{CO}_{\phi}: \mathrm{HF}^{\ast}(\phi) \cong \mathrm{Hom}^{\ast}_{\scr{F}^2(M)}(\Delta, \Delta_\phi) \overset{[\bf{M}^1]}{\longrightarrow} \mathrm{Hom}^{\ast}_{\scr{F}(M)-\mathrm{bimod}}(\mathrm{Id},\Gamma_{\phi}) \cong \mathrm{HH}^{\ast}(\scr{F}(M), \Gamma_\phi)
\end{equation*}
It is not difficult to check that this is chain-homotopic to the closed-open map as defined in Definition \ref{def:twisted_co} (analogous to the `unfolding' argument in \cite[Proposition 9.7]{ganatra_thesis}). 

\begin{proposition}\label{prop:HH_as_limit}
    Suppose $\scr{C}$ is a homologically smooth $A_{\infty}$-category and $s: \mathrm{id} \to F$ is an ambidextrous natural transformation. Then there is a quasi-isomorphism 
    \begin{equation*}
        \mathrm{CC}^{\ast}(\scr{C}[s\inv]) \simeq \dlim_d \mathrm{hom}_{\scr{C}-\mathrm{bimod}}(\mathrm{id}, \scr{F}^d)
    \end{equation*}
    where $\scr{F}$ is the $A_{\infty}$-bimodule associated to $F$ and the direct limit is taken along composition with $s$.
\end{proposition}

Here $\scr{C}[s\inv]$ denotes the localization of $\scr{C}$ at the natural transformation $s$, which is defined to be the $A_{\infty}$-quotient category of $\scr{C}$ by the full subcategory $\scr{A}$ of cones of $s$, in the sense of \cite{quotients,quotients2}. The quotient category $\scr{C}/\scr{A}$ has the same objects as $\scr{C}$ but morphisms are given by the bar complex:
\begin{equation*}
    \mathrm{hom}_{\scr{C}/\scr{A}}(X,Y) = \bigosum_{k \geq 0} \bigosum_{A_1, \dots, A_k \in \scr{A}} \mathrm{hom}_{\scr{C}}(A_k,Y) \otimes \cdots \otimes \mathrm{hom}_{\scr{C}}(A_1,A_2)[1] \otimes \mathrm{hom}_{\scr{C}}(X,A_1)[1]
\end{equation*}
where the $k=0$ term is $\mathrm{hom}_{\scr{C}}(X,Y)$ itself and the $A_{\infty}$ operations given by summation over all ways of collapsing the complex. One moreover has the notion of a quotient module: if $\scr{M}$ is a right $\scr{C}$-module, one may define a right $\scr{C}/\scr{A}$-module \cite{GPS1} denoted $\scr{M}/\scr{A}$ via:
\begin{equation*}
    (\scr{M}/\scr{A})(X) = \bigosum_{k\geq 0}\bigosum_{A_1, \dots, A_k \in \scr{A}} \scr{M}(A_k) \otimes \cdots \otimes \mathrm{hom}_{\scr{C}}(A_1,A_2)[1] \otimes \mathrm{hom}_{\scr{C}}(X,A_1)[1]
\end{equation*}
where again the $k=0$ term is $\scr{M}(X)$ and the $A_{\infty}$-module operations come from summation over all ways of applying collapsing the complex. Similarly, we use $\scr{A}\backslash\scr{N}$ to denote the quotient of a left $\scr{C}$-module $\scr{N}$ to give a left $\scr{C}/\scr{A}$-module; and $\scr{A}\backslash\scr{B}/\scr{A}$ to denote the quotient of a $\scr{C}$-bimodule $\scr{B}$ to give a $\scr{C}/\scr{A}$-bimodule, defined in an analogous fashion.

We recall the following properties of $A_{\infty}$-quotients, which follow directly from the definitions:

\begin{lemma}\label{lem:algebra}
Suppose $\scr{A}$ is a full subcategory of $\scr{C}$, and $X,Z$ are objects of $\scr{C}$: then
\begin{enumerate}
    \item The Yoneda module $\scr{Y}^{r}_{[X]}$ of $[X]$ in $\scr{C}/\scr{A}$ is quasi-isomorphic to the quotient module $\scr{Y}^r_{X}/\scr{A}$;
    \item The Yoneda module $\scr{Y}^{\ell}_{[Z]}$ of $[Z]$ in $\scr{C}/\scr{A}$ is quasi-isomorphic to the quotient module 
    \item The quotient bimodule $\scr{A}\backslash \Delta_{\scr{C}}/\scr{A}$ is quasi-isomorphic to the diagonal bimodule $\Delta_{\scr{C}/\scr{A}}$;
$\scr{A}\backslash\scr{Y}^{\ell}_{Z}$;
    \item The quotient map $\scr{M} \to \scr{M}/\scr{A}$ from $\scr{C}$-modules to $\scr{C}/\scr{A}$-modules is an exact functor.
\end{enumerate}
Hence if $\scr{C}$ is homologically smooth, then so is $\scr{C}/\scr{A}$.
\end{lemma}

\begin{proof} (of Proposition \ref{prop:HH_as_limit}) For a homologically smooth $A_{\infty}$-category $\scr{C}$, every $A_{\infty}$-functor $F: \scr{C} \to \scr{C}$ induces a perfect $\scr{C}$-bimodule $\scr{F}$, so by Lemma \ref{lem:algebra} it suffices to show that if $\scr{M},\scr{N}$ are in the subcategory of $\scr{C}-\mathrm{bimod}$ that is split-generated by tensor products of Yoneda bimodules, one can compute morphisms by
\begin{equation*}
    \mathrm{hom}_{\scr{C}/\scr{A}-\mathrm{bimod}}(\scr{A}\backslash \scr{M}/\scr{A},\scr{A}\backslash \scr{N}/\scr{A}) \simeq \dlim_d \mathrm{hom}_{\scr{C}-\mathrm{bimod}}(\scr{M}, \scr{N}\otimes \scr{F}^d)
\end{equation*} The K\"unneth theorem for bimodules \cite[Proposition 2.13]{ganatra_thesis} says
\begin{equation*}
        \mathrm{hom}_{\scr{C}/\scr{A}-\mathrm{bimod}}(\scr{Y}^{\ell}(X) \otimes \scr{Y}^{r}(Z),\scr{Y}^{\ell}(X\dash) \otimes \scr{Y}^{r}(Z\dash)) \simeq \mathrm{hom}_{\scr{C}/\scr{A}}(X,X\dash) \otimes \mathrm{hom}_{\scr{C}/\scr{A}}(Z,Z\dash)
\end{equation*}
while for $\scr{C}/\scr{A}$ we know from \cite[\S 1]{seidel_subalgebras} that
\begin{equation*}
    \mathrm{Hom}^{\ast}_{\scr{C}/\scr{A}}(Y,Y\dash) \cong \dlim_d \mathrm{Hom}^{\ast}_{\scr{C}}(Y, F^d(Y\dash))
\end{equation*}
Since direct limits commute with tensor products and taking cohomology we have
\begin{equation*}
    \mathrm{hom}_{\scr{C}/\scr{A}-\mathrm{bimod}}(\scr{Y}^{\ell}(X) \otimes \scr{Y}^{r}(Z),\scr{Y}^{\ell}(X\dash)\otimes \scr{Y}^{r}(Z\dash)) \simeq \dlim_{d_1, d_2} \mathrm{hom}_{\scr{C}}(X,F^{d_1}(X\dash)) \otimes \mathrm{hom}_{\scr{C}}(Z,F^{d_2}(Z\dash))
\end{equation*}
and using the fact the diagonal sequence is cofinal, along with the tensor-hom adjunction we get
\begin{equation*}
    \mathrm{hom}_{\scr{C}/\scr{A}-\mathrm{bimod}}(\scr{Y}^{\ell}(X) \otimes \scr{Y}^{r}(Z),\scr{Y}^{\ell}(X\dash)\otimes \scr{Y}^{r}(Z\dash)) \simeq \dlim_{d} \mathrm{hom}_{\scr{C}-\mathrm{bimod}}(\scr{Y}^{\ell}(X) \otimes \scr{Y}^{r}(Z),\scr{Y}^{\ell}(X\dash)\otimes \scr{Y}^{r}(Z\dash) \otimes \scr{F}^{d})
\end{equation*}
which completes the proof.
\end{proof}

\begin{proof}(of Theorem \ref{thm:HF_equals_QH})
By \cite[Lemma 3.16]{GPS1}, the morphisms in the cohomology category of $\scr{F}(M^0)$ are given by
\begin{equation*}
    \mathrm{Hom}_{\scr{F}(M^0)}(X_0,X_1) = \dlim_d \mathrm{Hom}_{\scr{F}(M)}(X_0, \phi^d(X_1))
\end{equation*}
where the connecting maps are given by Seidel's natural transformation $s \in \mathrm{Hom}(X_1, \phi(X_1))$. By Proposition \ref{prop:HH_as_limit}, we know that
\begin{equation*}
    \mathrm{HH}^{\ast}(\scr{F}(M^0)) \cong \dlim_d \mathrm{HH}^{\ast}(\scr{F}(M), \Gamma_{\phi^d}) 
\end{equation*}
and by Theorem \ref{thm:S_equals_S}, we know that the twisted closed-open map takes the Seidel natural transformation to the Seidel element $S$; since $\scr{CO}_{\phi}$ moreover respects multiplicative structures, we know
\begin{equation*}
    \mathrm{HH}^{\ast}(\scr{F}(M^0)) \cong \dlim_d \mathrm{HF}^{\ast}(\phi^d) 
\end{equation*}
as claimed.
\end{proof}

\section{Background on the product on fixed point Floer homology}\label{sec:product}
We begin by describing the product on fixed point Floer cohomology. We note the results we list here come from taking the \emph{cohomology} of the computations in our earlier paper \cite{ziwenyao} instead of homology. In other words, our co-product in \cite{ziwenyao}
\[
\Delta: \mathrm{HF}_*(\phi\circ\psi)\to \mathrm{HF}_*(\phi)\otimes \mathrm{HF}_*(\psi)
\]
is the dual of the product structure
\[
\cdot: \mathrm{HF}^*(\phi)\otimes \mathrm{HF}^*(\psi)\to \mathrm{HF}^*(\phi\circ\psi)
\]
in this paper. We explain this dualization process as a remark \ref{remark:homology vs cohomology}, but we first describe 
the results. 
Recall $\Sigma_g$ denotes a closed Riemann surface of genus g. Let $\phi$ denote a Dehn twist around a circle satisfying conditions of our previous paper \cite{ziwenyao} (in the context of this paper see Remark \ref{remark: topological assumption on Dehn twists}). Then Let $N = [0,1] \times S^1$ denote a Weinstein neighborhood of the essential circle. In this region (which we call the twist region) the symplectic for $\omega = dx\wedge dy$, and the Dehn twist $\phi^d$ can be written as 
\[
(x,y) \rightarrow (x,y-dx)
\]
Outside this Dehn twist region $N$, in $\Sigma_g':=\Sigma_g \setminus N$ the map $\phi^d$ is the Hamiltonian flow of a $C^2$ small Morse function $dH_0$. We further assume that near each boundary component of $\Sigma_g$, there are tubular coordinates $x_i\in(-\epsilon_i,0]$, $y_i\in S^1$ and a small real number $\theta_i$ such that $H_0(x_i,y_i)=\theta_i x_i$.

We now describe the critical points of $\phi^d$, which are the generators of the fixed point Floer cohomology $\mathrm{HF}^*(\phi^d)$. Away from the Dehn twist region the critical points are the same as the critical points of a Morse function. In the Dehn twist region $[0,1]\times S^1$ at $x=i/d, i=1,...,d$ there is a circle's worth of fixed points. This is a Morse-Bott circle. We perturb away the Morse-Bott degeneracy. By abuse of notation we also use $\phi^d$ to denote the perturbed symplectomorphism. Each Morse-Bott circle splits into two fixed points, one elliptic the other hyperbolic. For our notations, the circle of fixed points at $x=i/d$ splits into fixed points $e^d_i$ and $h^d_i$ (e and h for elliptic and hyperbolic respectively). We use the same grading conventions as in \cite{ziwenyao}, under which $|e_i^d|=-1$ and $|h_i^d|=0$. Generators $x$ of $\mathrm{CF}^*$ are graded by $|x|+1$, and we denote by $\mathrm{HF}^k$ the part of $\mathrm{HF}^*$ with grading $k$.

For each $d>0$, the fixed point Floer cohomology $\mathrm{HF}^*(\phi^d)$ is isomorphic to $H^*(\Sigma_g')\bigoplus (\oplus_{i=1}^{d-1}H^*(S^1))$. 

With the above conventions, $[e_i^d]$ and $[h_i^d]$ generate the $i$-th summand of $\oplus_{i=1}^{d-1}H^*(S^1)$.

Note the fixed points near the intersection between the Dehn twist region $N$ and the Morse region $\Sigma_g'$, which correspond to fixed points $e^d_0, h^d_0, e^d_d, h^d_d$ can be considered as either fixed points in the Dehn twist region $N$ or the Morse region $\Sigma_g'$. If we compute the cohomology, we find that $e_0^d + e_d^d$ represents a cohomology class in $\mathrm{HF}^*(\phi^d)$, which we write as $[f^d] \in  H^*(\Sigma_g')$. Similarly, $h_0^d+h_d^d$ represents a cohomology class in $H^*(\Sigma_g') \subset \mathrm{HF}^*(\phi^d)$, which we write as $[g^d]$.

Then the (commutative) product relations from $\mathrm{HF}^*(\phi^n)\otimes \mathrm{HF}^*(\phi^m) \rightarrow \mathrm{HF}^*(\phi^{n+m})$ in the Dehn twist region can be summarized as
\begin{equation}
[e_i^m]\cdot [e_j^n] = [e^{n+m}_{i+j}], \quad 0<i<n,  0<j<m
\end{equation}
\begin{equation}
[h_i^m]\cdot [e_j^n] = [h^{n+m}_{i+j}], \quad 0<i<m,  0<j<n
\end{equation}
\begin{equation}
[f^m] \cdot [e_j^n] = [e^{m+n}_j] + [e^{m+n}_{j+m}], \quad 0<j<n
\end{equation}
\begin{equation}
[g^m] \cdot [e_j^n] = [h^{m+n}_j] + [h^{m+n}_{j+m}], \quad 0<j<n
\end{equation}
\begin{equation}
[f^m] \cdot [f^n] = [e^{m+n}_m] + [e^{m+n}_{n}] + [f^{m+n}]
\end{equation}
\begin{equation}
[g^m] \cdot [f^n] = [h^{m+n}_m] + [h^{m+n}_{n}] + [g^{m+n}]
\end{equation}
\begin{equation}
[h^m_i] \cdot [h^n_j] =0, \quad i=0,...,m, \quad j =0,...,n
\end{equation}

Outside the Dehn twist region $N$ the product is trivial (for the critical points of the Morse functions aside from $e^d_0, e^d_d, h^d_0, h^d_d$).
\begin{remark}[Homology v cohomology] \label{remark:homology vs cohomology}
We briefly explain how to arrive at the previous product relations by dualizing the coproduct structure in \cite{ziwenyao}.

The chain complex for homology is given by $(\mathrm{CF}_*(\phi^d),\partial)$ where the differential counts $J$-holomorphic cylinders between critical points. The cochain complex $(\mathrm{CF}^*(\phi^d),d)$ is defined to be
\[
(\mathrm{CF}^*(\phi^d),d) := \mathrm{Hom}(\mathrm{CF}_*(\phi^d),\mathbb{C})
\]
The differential $d$ is defined by the following. For $\alpha \in \mathrm{CF}^*(\phi^d), a \in \mathrm{CF}_*(\phi^d)$ we define
\[
\langle d \alpha, a \rangle : = \langle \alpha, \partial a \rangle.
\]
Geometrically the differential $d$ on cohomology counts holomorphic cylinders flowing in the opposite direction. By the universal coefficient theorem we have 
$\mathrm{HF}^*(\phi^d)\cong\mathrm{Hom}(\mathrm{HF}_*(\phi^d),\CC) $.  Let $\Delta: \mathrm{CF}_*(\phi^{m+n}) \rightarrow \mathrm{CF}_*(\phi^n)\otimes \mathrm{CF}_*(\phi^m)$, then the product on cohomology is defined by
\[
\langle \alpha \cdot \beta, a\rangle = \langle \Delta a, \alpha \otimes \beta\rangle.
\]
Geometrically, the product structure of $\mathrm{HF}^*$ counts pair of pants with direction opposite to the pair of pants counted by the co-product $\Delta$. Applying this definition to the computations in \cite{ziwenyao} recovers the relations listed above.
\end{remark}

\begin{remark}[Signs] \label{rmk:signs}
Even though the computations in \cite{ziwenyao} were done with $\mathbb{Z}_2$ coefficients, we can choose a coherent orientation so that the above relations hold with $\mathbb{Z}$ (or in our case of interest, $\mathbb{C}$) coefficients. We give a brief outline of this below.

We follow Wendl's \cite{wendl2016lectures} expositions for coherent orientations. Recall to specify a coherent orientation for $\mathrm{HF}^*(\phi^m)$ and the underlying product structure, it suffices to consider the asymptotic operator $A$ associated to each Reeb orbit in $\mathrm{HF}^*(\phi^m)$. Consider the trivial bundle $\mathbb{C}^2\rightarrow \mathbb{C}$, equipped with a Cauchy Riemann type operator $D_A$ that is asymptotic to the operator $A$ as $r\rightarrow \infty$ in the base $\mathbb{C}$ (which we think of as a negative puncture). We choose any orientation in the determinant line bundle of $D_A$, which we write as $\mathfrak{o}_A \in \det(D_A)$, then the collection of such choices specifies a coherent orientation.

Now we look at the Reeb orbits involved in the computation of $\mathrm{HF}^*(\phi^m)$, we realize there are two kinds of Reeb orbits. There are $\gamma_x$ which correspond to Morse critical points $x \in \Sigma_g \setminus N$ with associated asymptotic operators $A_{\gamma_x}$, and Reeb orbits in the Dehn twist region $N$ of the form $e^m_j$ and $h^m_j$. Now we can choose our trivializations so that all $e^m_j$'s (resp. all $h^m_j$'s) have the same asymptotic operator, which we denote by $A_e$ (resp. $A_h$). 

If we restrict our attention to purely the Dehn twist region $N$, then we realize that we are computing the product on the symplectic cohomology of $T^*S^1$. This was already observed in \cite{ziwenyao}: we can view the Dehn twist region as a subset of $T^*S^1$, and the $J$-holomorphic curve equation for our count of sections in the bundle of the form $B_0\times T^*S^1 \rightarrow  B_0$ (here $B_0$ is the thrice punctured sphere) is exactly the Hamiltonian Floer equation computing the pair of pants product on the symplectic cohomology of $T^*S^1$. From its isomorphism with the homology of the loop space  $S^1$, we can choose $\mathfrak{o}_h \in \det (D_{A_h})$ and $\mathfrak{o}_e \in \det (D_{A_e})$ so that all the pairs of pants contained in the Dehn twist region $N$ are counted with sign $+1$. Now for the orbits in the Morse region, choose $\mathfrak{o}_\gamma$ so that the cohomology class $f^m$ multiplies other Reeb orbits in the Morse region as the identity. Then we need to verify that the two coherent orientations agree when we take $\gamma = e^m_0, e^m_m, h^m_0, h^m_m$. We observe from the requirements we imposed on $\mathfrak{o}_e$ on Dehn twist region and Morse region that they agree for $\gamma=e_0^m, e_m^m$. If they do not agree for $\gamma =h^m_0, h^m_m$, we simply reverse $\mathfrak{o}_h$ and observe this has no effect on the sign of the curves we count in the Dehn twist region.

\end{remark}

\subsection{Fixed point Floer homology for Dehn twists on punctured Riemann surfaces} \label{sec:fixed_point_floer_punctured}

In this subsection we outline the difference between Dehn twists on closed Riemann surface $\Sigma_g$ and punctured Riemann surface $\Sigma_{g,k}$. The main difference is that near each puncture $p_i$ we need to consider wrapping in a way that is similar to symplectic cohomology.

For simplicity we assume there is only one puncture $p$. Choose cylindrical neighborhoods $(x,y) \in [0,\infty) \times S^1$ around $p$. Then $\phi$ looks like the time-1 map of the Hamiltonian flow of $H =\frac{1}{2}x^2$ in the neighborhood (then perturbed slightly to $H_0$, in order to break the Morse-Bott degeneracy). The fixed points are located over the circles $x = 0, 1,2,\cdots$. For each non-negative integer $i$ there are two fixed points, $u_i$ an elliptic fixed point of Conley-Zehnder index -1 and $v_i$, an hyperbolic fixed point of Conley-Zehnder index 0. 

We need to further describe the Floer data near the puncture $p$. We make our conventions consistent with the \emph{universal and consistent choice} of Floer data described in \cite{ganatra_thesis}, because we use the open- closed map as defined in \cite{ganatra_thesis} and want to import the technology developed therein. In the following, we briefly recall the Floer data described in \cite{ganatra_thesis} and point out the differences with the setup in \cite{ziwenyao}, where no punctures are dealt with. To define the product $\mathrm{HF}^*(\phi^m)\otimes \mathrm{HF}^*(\phi^n)\rightarrow \mathrm{HF}^*(\phi^{m+n})$, let $B$ be the thrice-punctured sphere with chosen cylindrical coordinates $(s_0, t_0\in [0, \infty)\times S^1$ for the positive puncture, and $(s_i, t_i)\in (-\infty, 0]\times S^1$ for the $i$-th negative puncture ($i=1,2$). To define the product structure, choose a closed one-form $\beta$ on $B$ that is equal to $mdt_1$ and $ndt_2$ for the two negative punctures and $(m+n)dt_0$ for the positive puncture (likewise, if we want to define the coproduct structure, we let $\beta$ to be $(m+n)dt_0$ for the negative puncture and $mdt_1$ and $ndt_2$ for the two positive punctures). We then make the product bundle $B\times [0,\infty)_x\times S^1_y$ into a symplectic fiber bundle, with the fiberwise symplectic form $\omega_\Sigma = dx\wedge dy+d(H_0\beta)$. Remark that near the three punctures of $B$, the parallel transports along the $t_*-$circles are precisely the time-1 maps of the Hamiltonians $mH_0$, $nH_0$ and $(m+n)H_0$ respectively. We also fix a function $a:B\to [\min(m,n), m+n]$ that is equal to $m$ on the first negative puncture, $n$ on the second negative puncture, and $m+n$ on the positive puncture (likewise, if we want to define the coproduct structure, $a$ is assumed to be $m$ and $n$ on the two positive punctures, and $m+n$ on the negative puncture). View the neighborhood of $p$ in $\Sigma_{g,1}$ as the Liouville manifold $[0,\infty)_x\times S^1_y$ with the Liouville form $xdy$, so the Liouville flow $\psi^\rho$ after time $\log \rho$ is given by $(x,y)\mapsto (\rho x, y)$ near the puncture $p$. 

The almost complex structure $J$ on the bundle $X = B\times [0,\infty)\times S^1$ should satisfy the following. Fix a small positive number $\delta$. In a fixed small neighborhood $B\times [0,\delta)\times S^1$ of the slice $\{x=0\}$, $J$ is a fibration-compatible almost complex structure induced by the almost complex structure $J_0$ on $[0,\infty)_x\times S^1_y$ which sends $\partial_x$ to $\partial_y$. In the region $B\times [2\delta,\infty)\times S^1$, $J$ is the fibration-compatible almost complex structure that is induced by the almost complex on the vertical distribution, which is equal to $(\psi^{a(z)})^*J_0$ on the fiber over $z\in B$\footnote{We remark that if we orient the fiber in the same way as \cite{ganatra_thesis}, i.e. the orientation induced by the product $S^1_y\times [0,\infty)_x$, then the almost complex structure described in this paragraph corresponds to the almost complex structure that sends $\partial_x$ to $-\partial_y$ for $x\ge 2\delta$, which is of the $c$-rescaled contact type introduced in Section \ref{sec: Compatibility with wrapping}.}. We note that in order to apply the \emph{local energy inequality} described in \cite{ziwenyao} to the slice $\{x=0\}$, we only need the almost complex structure to be the fibration-compatible one that is induced from $J_0$ \emph{in a neighborhood of} the slice. Consequently, if we make the above requirements for the almost complex structure, the same ``no crossing'' result as described in \cite{ziwenyao} holds. In particular, the no-crossing result tells us the $J$-holomorphic curves with inputs and outputs both below $x=0$ remain entirely within that region, and the analogous statement holds true for $J$-holomorphic curves with input and output above $x=0$. Furthermore, via homology considerations (together with ``no crossing''), we cannot have $J$-holomorphic curves asymptotic to critical points both above $x=0$ and below $x=0$. Hence as far as the product is concerned, the product splits into the product of two distinct regions (with tiny overlap around $x=0$): we get the Morse theoretic product for the region where $x\leq 0$ and away from the Dehn twists; and we get the product for the non-negative sector of symplectic cohomology of $T^*S^1$ for $x\geq 0$.

Here we use the superscript such as  $u_i^n$ and $v_i^n$ to denote the fact we are thinking of the fixed points as living in $\mathrm{HF}(\phi^n)$. Similar to the unpunctured case, the fixed point Floer cohomology $\mathrm{HF}^*(\Sigma_{g,k},\phi^d)$ is isomorphic to $H^*(\Sigma_{g,k}') \oplus (\oplus_{j=1}^{d-1}H^*(S^1)) \oplus (\oplus_{i=1}^\infty \mathbb{C} \langle u_i^d,v_i^d\rangle)$. Here we think of $u_0^d$ and $v_0^d$ as being (the dual of) critical points in $\Sigma_{g,k}'$. In particular, the cohomology class $[f^d]\in H^0(\Sigma_{g,k}')$ is given by $[e_0^d + e^d_d + u_0^d]$, and the cohomology class $[g^d]\in H^1(\Sigma_{g,k}')$ is given by $[h_0^d+h_d^d]$. There is another special cohomology class $[h_0^d-v_0^d]\in H^1(\Sigma_{g,k}')$, which we denote by $[\varphi^d]$. This new cohomology class comes from the fact adding a new puncture changes the first homology group of the surface. The new product relations are given by
\[
[u_i^m]\cdot[u_j^n]=[u_{i+j}^{m+n}],\quad [u_i^m]\cdot[v_j^n]=[v_{i+j}^{m+n}],\quad [v_i^m]\cdot[v_j^n]=0,
\]
\[
[f^m]\cdot [u_i^n]=[u_i^{m+n}],\quad [f^m]\cdot [v_i^n]=[v_i^{m+n}],
\]
\[
[g^m]\cdot [u_i^n]=[v_i^{m+n}],\quad [g^m]\cdot [v_i^n]=0,
\]
\[
[e_i^m]\cdot [u_j^n]=[e_i^m]\cdot[v_j^n]=[h_i^m]\cdot [u_j^n]=[h_i^m]\cdot [v_j^n]=0,
\]
\[
[f^m]\cdot [\varphi^n] = [\varphi^{m+n}]+[h_m^{m+n}],\quad [e_i^m]\cdot [\varphi^n] = [h_i^{m+n}],\quad [h_i^m]\cdot [\varphi^n] = 0,
\]
\[
[u_i^m]\cdot[\varphi^n] = -[v_i^{m+n}],\quad [v_i^m]\cdot [\varphi^n] =0,
\]
and the product relations between $[f^d], [g^d]$ and $[e^m_i], [h_i^m]$ are the same as the unpunctured case. 
These product relations come from the same calculations that were performed in \cite{ziwenyao}. The same no-crossing lemma tells us that when we think of $\Sigma_{g,k}$ as union of the Morse region, Dehn twist region, and the puncture region, holomorphic curves with inputs/outputs in the puncture region cannot enter the Morse region and vice versa (with the exception of $u^d_0$ and $v^d_0$ which we can think of as belonging in either the puncture or the Morse region). Carrying out the same computation as in \cite{ziwenyao} gives us the above results.
\begin{remark}
    In the case where $\Sigma_{g,k}$ has punctures, we need $\phi$ to wrap around the punctures, and in this case there are infinitely many generators in $\mathrm{HF}_*(\Sigma_{g,k},\phi)$. To avoid subtleties with duals of infinite dimensional vector spaces, we simply define the cochain complex $\mathrm{CF}^*(\phi^d)$ to be generated by the fixed points of $\phi^d$, with the differential $d$ and the product structure $\cdot$ defined in the same way as in Remark \ref{remark:homology vs cohomology}.
\end{remark}

\begin{remark}
It is also possible to compute the fixed point Floer cohomology using a direct enumeration of $J$-holomorphic sections by using the standard almost complex structure near each of the punctures -- this is more similar to the approach in \cite{ziwenyao} and similar techniques will give the same co-homology level product relations. However we shift our conventions to agree with that of Ganatra's \cite{ganatra_thesis} to use the machinery developed in that paper.
\end{remark}

\section{Computation of the Seidel class}\label{sec:computation}
In this section we give an explicit computation of the Seidel class in the closed-string setting and in the punctured setting. Throughout this section, we fix a fibration compatible\footnote{A fibration compatible almost complex structure makes the projection map holomorphic, and preserves the horizontal sections.} almost complex structure $J$ on the Lefschetz fibration, that is induced by the complex structure $j_0$ on the Riemann surface which, inside the neighborhood around the circle, sends $\partial_x$ to $\partial_y$. Note by Remark \ref{remark:seidel class for Riemann surface with boundary}, to compute the Seidel class for $\Sigma_{g,k}$, near a puncture $p_i$ where there are cylindrical neighborhoods $S^1 \times [0,\infty) $, we pass to the open Riemann surface with boundary by restricting to $S^1 \times [0,\epsilon)$ for each of these punctures. We call the resulting Riemann surface $\tilde{\Sigma}_{g,k}$. We then form the exact Lefschetz fibration with boundary, with fiber $\tilde{\Sigma}_{g,k}$, and the monodromy near each of these neighborhoods is given by the time-1 map of the Hamiltonian $H(s,t) = s^2/2$.

\begin{theorem} \label{thm:s_calculation}
For either $\mathrm{HF}^*(\Sigma_{g},\phi)$ or $\mathrm{HF}^*(\Sigma_{g,k},\phi)$, the Seidel class $S\in \mathrm{HF}^*(\phi)$ is given by $\pm[f^1]$.
\end{theorem}

Here is the sketch of the proof of Theorem \ref{thm:s_calculation}. We first make this computation in the exact setting, i.e. where the fiber is $\tilde{\Sigma}_{g,k}$. We show by energy estimates that all sections counted by the Seidel class are horizontal. We then construct the horizontal sections by hand and show they are transversely cut out. Finally, using an index argument, we show that the sections counted in the computation of the Seidel class in the non-exact case can be reduced to the exact case.

\subsection{The exact case}
We first compute the index of the curves counted by the Seidel class in the exact case. Let $s:\mathbb{C} \rightarrow E$ denote a $J$-holomorphic section counted by the Seidel element. Let $pt\in \Sigma$ be a generic point away from the Dehn twist region on $\Sigma$, and consider the ``constant'' section $c:\mathbb{C} \rightarrow E$ given by $c(z)=pt$. For any section $u$ of $E$ that is asymptotic to an orbit $\gamma_x$, we have the following definition of the wrapping number (see Definition 4.2 of \cite{hutchings2005periodic}):

\begin{definition}
We define $\eta : = s \cap c \in \mathbb{Z}$ to be the wrapping number of the section $s$.
\end{definition}

The wrapping number $\eta(s)$ of a section $s$ is always non-negative. To see this, notice that in the definition of $\eta(s)$, the intersection number does not depend on the choice of $pt$. In particular, we can choose $pt$ to be a critical point of $H_0$. Under such a choice, the section $c$ is $J$-holomorphic, and hence $\eta(u)\ge0$ by positivity of intersections. Moreover, any $J$-holomorphic section $s$ that's not $c$ with $\eta(s)=0$ is disjoint from $c$. 

\begin{proposition}\label{index}
Let $x \in Fix(\phi)$ denote the fixed point of $\phi$ and let $\gamma_x$ be the corresponding orbit in the mapping torus. Suppose a section $u$ is asymptotic to $\gamma_x$ at $r=\infty$, then the Fredholm index of the section $u$ is given by
\[
\mathrm{ind}(u)=1+CZ^\tau (\gamma_x)+(4-4g)\eta([u]).
\]
 Hence the only fixed points of $\phi$ that can contribute to the Seidel class are $e_0^1$, $e_1^1$ or $u_0^1$.
\end{proposition}
\begin{proof}
    Away from the critical locus, the bundle is a product bundle, and we use the trivialization $\mathrm{Ver}=T((-1,1)_x\times S^1_y)\cong \mathbb{R}^2$ to trivialize the vertical distribution. The Fredholm index formula reads:
    \[
    \begin{split}
    \mathrm{ind}(u)&=-\chi(u)+2\langle c_1^\tau(TE), [u] \rangle+ CZ^\tau (\gamma_x)\\
    &=-1+2(1+\langle c_1^\tau(\mathrm{Ver}), [u]\rangle)+CZ^\tau (\gamma_x)\\
    &=1+2\langle c_1^\tau(\mathrm{Ver}), [u] \rangle+ CZ^\tau (\gamma_x).
    \end{split}
    \]
    The term $\langle c_1^\tau(\mathrm{Ver}), [u] \rangle$ is equal to the wrapping number $\eta([u])$ multiplied by $2-2g$, whose proof (in a slightly different setting) can be found in \cite[Lemma 5.6]{ziwenyao}. So we have
    \[
    \mathrm{ind}(u)=1+CZ^\tau (\gamma_x)+(4-4g)\eta([u]).
    \]
    Since $g\ge 2$, $\eta([u])\ge 0$ and $CZ^\tau(\gamma_x)\in \{-1,0,1\}$, we see that the only possibility for $\mathrm{ind}(u)$ to be zero is $CZ^\tau(\gamma_x)=-1$ and $\eta([u])=0$. By our setting, the only fixed points that have Conley-Zehnder index -1 are $e_0^1$, $e_1^1$ and $u^1_0$, hence the proof.
\end{proof}

\begin{proposition}\label{horizontal lemma}
The only sections of $E$ that are asymptotic to $e_0^1$, $e_1^1$ or $u_0^1$ at $r=\infty$ are horizontal. The horizontal sections exist and are automatically transversely cut out.
\end{proposition}
\begin{proof}
    We use the same notations from Proposition \ref{prop:standard_LF}, and let $G(r, x, y)=(g(\mu)-1)\tilde{R}_{\phi(r))}(\mu)+H_0(x,y)$, and let the $\lambda$ be a primitive of $\omega_\Sigma$ that is equal to $xdy$ in the twist region. We know that there is a one-form $\alpha$ on $E$ such that $\omega=d(\lambda+\alpha)$, and $\alpha=Gd\theta$ away from the critical locus. For every section $u$, we consider the vertical energy
    \[
    E(u)=\frac{1}{2}\int |\partial_s u-\partial_s^\#|^2_{g_J}+|\partial_t u-\partial_t^\#|^2_{g_J} ds\wedge dt,
    \]
    where $(s,t)$ is any local conformal coordinate of the base, $g_J$ is the metric induced by $\omega$ and the almost complex structure $J$, and $\partial_s^\#$, $\partial_t^\#$ are the horizontal lifts of the vector fields. Using the polar coordinate $(r,\theta)$, the above expression can be re-written as
    \[
    E(u)=\int u^*\omega - \int \frac{\partial G}{\partial r} dr\wedge d\theta.
    \]
    Now if $u$ is any section that is asymptotic to the fixed points $e_0^1$, $e_1^1$  or $u_0^1$ at $r=\infty$, then the term $\int u^*\omega$ is equal to $u^*\lambda$ integrated at the boundary of $u$, which is zero. Away from the critical locus, the integrand $\frac{\partial G}{\partial r}$ is non-negative. So we conclude that the second term $\int \frac{\partial G}{\partial r} dr\wedge d\theta \ge 0$. Thus, for any $J$-holomorphic section $u$ that is asymptotic to $e_0^1$, $e_1^1$  or $u_0^1$ at $r=\infty$, we have the vertical energy
    \[
    E(u)\le 0,
    \]
    and this can only happen when $E(u)=0$ and $u$ is a horizontal section.

    There exist horizontal sections asymptotic to $e_0^1$ and $e_1^1$ corresponding to the critical points of $H_0$ near the circles $\{x=\pm \lambda\}$. To see this, locally we can write the fiberwise symplectic two-form $\omega$ as $\omega_\Sigma + d H_0 \wedge d\theta$, so the section given by $\CC\times \{p\}$ where $p$ is the critical point of $H_0$ gives the existence of horizontal sections. A simple energy argument shows that for $e_0^1$, $e_1^1$ and $u_0^1$, such horizontal section with the desired asymptotes is unique.

    That the horizontal sections are transversely cut out follows from the automatic transversality criterion given in \cite[Theorem 1]{wendlautomatic}. To be precise we observe the Fredholm index is zero, $c_N =-1$, and $Z(du) =0$. Hence the automatic transversality criterion is satisfied.
    \end{proof}
    \begin{remark}
        Since all of the sections involved in computing the Seidel class are transverse, there is no need to further perturb the almost complex structure $J$, for example as in definition \ref{def:seidel class}, and hence the above proposition finishes the proof of Theorem \ref{thm:s_calculation} in the exact case.
    \end{remark}

\begin{remark}
Whether the end result is $+[f^1]$ or $-[f^1]$ will necessitate working through the coherent choice of orientations. For us we assume the Seidel class is $[f^1]$ for the rest of the article, and make the observation that this choice does not in fact affect the resulting ring in Theorem \ref{thm:ring_structure}. The signs and orientations for the Seidel class are the subject of future work by Shaoyun Bai and Paul Seidel \cite{shaoyun}.
\end{remark}

\subsection{The non-exact case}

We now deduce the computation of Seidel class for closed Riemann surfaces from the case of punctured Riemann surfaces. 

We consider the standard Lefschetz fibration constructed in Proposition \ref{prop:standard_LF} and count sections of this fibration satisfying conditions described in Definition \ref{def:seidel class}. Recall from Proposition \ref{index}, for a generic almost complex structure $J$, we have 
\[
\mathrm{ind}(s)=1+CZ^\tau (\gamma_x)+(4-4g)\eta([s]),
\]
where $CZ^\tau (\gamma_x)\in\{-1,0,1\}$ and $\eta(s)\ge0$. 

The next proposition follows from the above observation:

\begin{proposition}

Any section $s$ counted by the Seidel element must satisfy $\eta(s)=0$ and $CZ^\tau(\gamma_x)=-1$.
    
\end{proposition}


Hence we see that for generic $J$, all sections counted by the Seidel element must satisfy $\eta=0$. If we pick the constant section $c:\mathbb{C}\to E$ to be $c(z)=pt$ with $pt$ a critical point of $H_0$, then $\eta(u)=0$ implies that $u$ is disjoint from $c$, as we observed before. So we can consider the new bundle $E'=E \setminus c$ with the new generic fiber $\Sigma_{g,1}=\Sigma_g\setminus pt$. On $\Sigma_{g,1}$, we can find a one-form $\lambda_\Sigma$ with $d\lambda_\Sigma=\omega_\Sigma$ and $\lambda_\Sigma=xdy$ in the twist region. Now $E'$ becomes an exact Lefschetz fibration with fiberwise symplectic form
\[
\omega=d(\lambda_\Sigma+\alpha),
\]
where $\alpha$ is the one-form described in Proposition \ref{horizontal lemma}. So the discussions from the previous section imply Proposition \ref{horizontal lemma} in the closed case as well, and this finishes the proof of Theorem \ref{thm:s_calculation}.

\section{A-model computations of symplectic cohomology}\label{sec:A_model_calculations}
\subsection{Single Dehn twist on a closed surface}
We first discuss the case of a single Dehn twist on the closed surface $\Sigma_g$, then we shall point to the extensions to the cases of multiple Dehn twists and punctured surface $\Sigma_{g,k}$. 
In the following, we use the notation $R^d$ to denote the span
\[
\langle [e_0^d], [e_1^d],\cdots,[e_{d-1}^d], [f^d] \rangle,
\]
in other words, $R^d=\mathrm{HF}^0(\phi^d)$ for $d>0$. Notice that when $d=0$, $\mathrm{HF}^0(\mathrm{id})\cong H^*(\Sigma_g)$ is generated by $[f^0]$ together with the fundamental class $K\in H^2(\Sigma_g)$. We grade the elements $[e_i^d]$, $[h_i^d]$, $[f^d]$ and $[g^d]$ with degree $d$, making $\bigoplus_{d=0}^{\infty}R^d$ a graded algebra. We have the following

\begin{theorem}\label{thmsum}
The algebra $\bigoplus_{d=0}^{\infty} R^d$ is isomorphic to 
\[
\CC[X,Y,Z]/(XYZ-Y^3-Z^2)
\]
as graded $\CC$-algebras, where $|X|=1$, $|Y|=2$ and $|Z|=3$. And the $\CC$-algebra $\bigoplus_{d=0}^{\infty} \mathrm{HF}^0(\phi^d)$ is isomorphic to
\[
(\CC[X,Y,Z]/(XYZ-Y^3-Z^2))\oplus \CC\langle K \rangle,
\]
as a vector space, where $|K|=0$. The algebra structure is determined by the subalgebra $\bigoplus_{d=0}^{\infty} R^d\cong \CC[X,Y,Z]/(XYZ-Y^3-Z^2)$, together with the relations $K^2=KX=KY=KZ=0$.

\end{theorem}

\begin{remark}
   The fact that the action of the point class $[f^0]\in \mathrm{HF}^0(\phi^0)$ with any other element is identity, and that $K^2=KX=KY=KZ=0$ follow from the module structure of $\mathrm{HF}^*(\Sigma_g,\phi)$, see the discussions in e.g.  \cite{ruan1995bott,piunikhin1996symplectic,seidel1997floer}. In the following, when deriving the algebra structure, we will often ignore the second factor and only focus on the subalgebra $\bigoplus_{d=0}^{\infty} R^d$ to simplify the notations.
\end{remark}

To prove Theorem \ref{thmsum}, we first prove the algebra $\bigoplus_{d=0}^{\infty} R^d$ is generated by $[f^1]$, $[e_1^2]$ and $[e_1^3]$.

\begin{lemma}\label{lem: generators}
    The algebra $\bigoplus_{d=0}^{\infty} R^d$ is generated by $[f^1]$, $[e_1^2]$ and $[e_1^3]$.
\end{lemma}
\begin{proof}
     Notice that $[f^2]=[f^1][f^1]-[e_1^2]-[e_1^2]$ and $[e_2^3]=[f^1][e_1^2]-[e_1^3]$, so it suffices to show that for each $d\ge 4$, the classes $[e_i^d]$ ($0<i<d-1$) and $[f^d]$ are generated by elements of lower degree. To begin with, we notice that if $1<i<d$ and $d\ge 4$ then $[e_i^d]=[e_1^2][e_{i-1}^{d-2}]$, so we only need to focus on $[f^d]$, $[e_1^d]$ and $[e_{d-1}^d]$.
     
     Next we observe that for each $d\ge 4$, we have
     \[
     [f^d]=[f^2][f^{d-2}]-[e_2^{d}]-[e_{d-2}^d]=[f^2][f^{d-2}]-[e_1^{d-2}][e_1^2]-[e_{d-3}^{d-2}][e_1^2],
     \]
     so $[f^d]$ can be generated by elements of lower degrees. 

     Finally, we observe that for each $d\ge4$,
     \[
     [e_1^d]=[f^1][e_1^{d-1}]-[e_2^d]=[f^1][e_1^{d-1}]-[e_1^2][e_1^{d-2}],
     \]
     and
     \[
     [e_{d-1}^d]=[f^1][e_{d-2}^{d-1}]-[e_{d-2}^d]=[f^1][e_{d-2}^{d-1}]-[e_{1}^2][e_{d-3}^{d-2}].
     \]
     So $[e_1^d]$ and $[e_{d-1}^d]$ (with $d\ge 4$) are also generated by elements of lower degree.
\end{proof}

\begin{lemma}
The generators $[f^1], [e_1^2], [e_1^3]$ satisfy the relation
\begin{equation} \label{eqn:relation}
[f^1] \cdot [e_1^2] \cdot [e_1^3] = [e_1^2]^3 +[e_1^3]^2
\end{equation}
where the exponent is with respect to the product.
\end{lemma}
\begin{proof}
Follows directly from the product relations.
\end{proof}

We next show the relation (\ref{eqn:relation}) generates all of the relations of the generators $[f^1], [e_1^2], [e_1^3]$ in the algebra $\bigoplus_{d=0}^{\infty} R^d$. 
The proof of Theorem \ref{thmsum} then follows immediately.
Let 
\[
P(x,y,z) := xyz-y^3-z^2 \in \mathbb{C}[x,y,z].
\]

\begin{lemma} \label{lemma:relation}
Let $L(x,y,z) \in \mathbb{C}[x,y,z]$ be a nonzero polynomial such that $L([f],[e_1^2], [e_1^3]) =0$, then there is a polynomial $Q(x,y,z)\in \mathbb{C} [x,y,z]$ so that 
\[
L(x,y,z) = Q(x,y,z) \cdot P(x,y,z).
\]
\end{lemma}
\begin{proof}
To set notation, wherever we write a polynomial such as $P(x,y,z)$ we think of it as a polynomial in $\mathbb{C}[x,y,z]$. When we write $P([f],[e_1^2], [e_1^3])$ we think of it as an element in $\bigoplus_{d=0}^\infty R^d$.

We write $L(x,y,z) = g_k(y,z)x^k + g_{k-1}(y,z)x^{k-1} +...+ g_0(y,z)$. The assumption that $L$ is nonzero implies $k>0$, since $\mathbb{C}[ [e_1^2], [e_1^3]] \cong \mathbb{C} [y,z]$, which follows from the relations in Section \ref{sec:product}. 

We next apply the division algorithm in $\mathbb{C}(y,z)[x]$ to write
\[
L(x,y,z) = \left (x- \frac{y^3+z^2}{yz}\right) \left (r_{k-1}(y,z) x^{k-1} +...+ r_0(y,z)\right) + h_0(y,z)
\]
where $r_j(y,z), h_0(y,z) \in \mathbb{C}(y,z)$.

It follows from an inductive argument we can write 
\[
r_j(y,z) = \frac{\phi_j(y,z)}{(yz)^{m_j}}, \quad \phi_j(y,z) \in \mathbb{C}[y,z],\quad m_j \in \mathbb{Z}_{\geq0}
\]
and
\[
h_0(y,z) = \frac{\rho_0(y,z)}{(yz)^M}, \quad \rho_0(y,z)\in \mathbb{C}[y,z],\quad M\geq 1
\]
In particular this implies there is a large enough $N$ so that
\[
(yz)^N L(x,y,z) = (xyz-y^3-x^2)\left(\tilde{r}_{k-1}(y,z) x^{k-1} + \cdots + \tilde{r_0}(y,z)\right ) + \tilde{h_0}(y,z)
\]
where $\tilde{r}_j(y,z), \tilde{h}_0(y,z) \in \mathbb{C}[y,z]$. We next plug $x=[f^1]$, $y=[e_1^2]$ and $z=[e_1^3]$ into the above expression to get $\tilde{h}_0([e_1^2], [e_1^3]) =0$. Using the isomorphism $\mathbb{C}[y,z] \cong \mathbb{C}[[e_1^2], [e_1^3]]$ of subalgebras, we have $\tilde{h}_0(y,z) = 0$.

Using the fact $\mathbb{C}[x,y,z]$ is a UFD and that the polynomial $xyz-y^3-z^2$ is irreducible, we conclude there is another polynomial $Q(x,y,z)$ so that $L(x,y,z) = Q(x,y,z)P(x,y,z)$. This concludes the proof of the lemma.
\end{proof}

It follows also from the above computations that
\begin{theorem}
We have the following description of $\bigoplus_{d=0}^\infty \mathrm{HF}^1(\Sigma_g,\phi^d)$ as a module over $\bigoplus_{d=0}^{\infty} \mathrm{HF}^0(\phi^d)$:
\[
   \bigoplus_{d=0}^\infty \mathrm{HF}^1(\Sigma_g,\phi^d) \cong \CC[X,Y,Z]/(XYZ-Y^3-Z^2) \oplus (\mathbb{C}[X])^{2g-2} \oplus \CC.
\] 

The above expression indicates that $A \subset \mathrm{HF}^0(\Sigma_g,\phi^d)$ acts on the first factor by multiplication,  on the second factor by projection to $\CC[X]$ followed by multiplication, and on the last factor by projection to $\CC$ followed by multiplication.
\end{theorem}
\begin{proof}
    Here is a sketch of the proof. If we denote by $S$ the submodule generated by the classes $[h^i_j]$ and $[g^1]$, then the argument in the proof of Lemma \ref{lem: generators} shows that $S$ is generated by the elements $[g^1]$, $[h^2_1]$, and $[h^3_1]$. The extra classes come from outside of the twist region, i.e. the classes that correspond to $H^1(\Sigma_g)$ (when $d=0$)  and $H^1(\Sigma_g')$ (when $d>0$). The dimensions of the vector spaces spanned by these extra classes are $2g-1$ (when $d=0$) and $2g-2$ (when $d>0$) respectively. Multiplication by the Seidel element $X=[f^1]$ identifies these extra classes from different degrees, except that the Seidel element annihilates one of the extra classes from $\mathrm{HF}^1(\phi^0)\cong H^1(\Sigma_g)$, which corresponds to the fixed point contained in the twist region.
\end{proof}

\begin{remark}
    The main difference from \cite{ziwenyao} is that we need to deal with the product of $\mathrm{HF}^1(\phi^0)$ and $\mathrm{HF}^0(\phi^d)$. The calculations for those products can be derived from the ``extrinsic'' description of the products in Floer homologies, see \cite{ruan1995bott}. More concretely, we can extend the small Hamiltonian $H_0$ to $\Sigma_g$ in such a way that there's an extra pair of critical points with Morse indices $1$ and $2$ inside the twist region $N$. We then count Floer cylinders that intersects the ascending manifolds of the negative gradient flow $-\nabla H_0$. The only nontrivial such count (besides those calculated by the cup product of $H^*(\Sigma_g')$) results in $[g^0]\cdot[e_i^d]=[h_i^d]$ for $d\ge 2$ and $i=1,2,\cdots,d-1$.
\end{remark}

\begin{remark}
Technically speaking there is another multiplication on fixed point Floer cohomology taking pair of elements in $\bigoplus_{d=0}^\infty \mathrm{HF}^1(\Sigma_g,\phi^d)$  to an element in $\bigoplus_{d=0}^\infty \mathrm{HF}^0(\Sigma_g,\phi^d)$. However this product is not very interesting - the only nontrivial piece comes from the product $\mathrm{HF}^1(\Sigma_g, \phi^0) \otimes \mathrm{HF}^1(\Sigma_g,\phi^0) \rightarrow \mathrm{HF}^0(\Sigma_g,\phi^0)$, which is the classical cup product. Henceforth we will not mention this product and focus on the module structure instead.
\end{remark}

It is clear from the proof of Theorem \ref{thmsum} that
\[
R^d\cong (\CC[X,Y,Z]/(XYZ-Y^3-Z^2))_d,
\]
the degree $d$ part of the graded algebra $\CC[X,Y,Z]/(XYZ-Y^3-Z^2)$. If we consider the direct system $\{R^d\}$ with the connecting map $R^d\to R^{d+1}$ given by multiplying Seidel's element $[f^1]\in \mathrm{HF}^0(\phi)$, then the direct limit $\dlim_d R^d$ can be identified with the direct limit
\[
\dlim_d (\CC[X,Y,Z]/(XYZ-Y^3-Z^2))_d
\]
where the connecting map given by multiplication with $X$. With the above observations, we have the following
\begin{theorem}\label{thm:ring_structure}
    $\dlim_d R^d=\dlim_d \mathrm{HF}^0(\phi^d)\cong \CC[Y,Z]/(YZ-Y^3-Z^2)$.
\end{theorem}
\begin{proof}
    When taking the direct limit, we can ignore the second factor in $\mathrm{HF}^0(\phi^d)\cong \CC^2$, as its multiplication with the Seidel elements yields zero. Identify $R^d$ with $(\CC[X,Y,Z]/(XYZ-Y^3-Z^2))_d$. It is clear form the definition of the direct system that the direct limit is isomorphic to
    \[
    (\CC[X,Y,Z,X^{-1}]/(XYZ-Y^2-Z^3))_0,
    \]
    the degree $0$ part of the graded algebra $\CC[X,Y,Z,X^{-1}]/(XYZ-Y^2-Z^3)$. This is isomorphic to $\CC[Y,Z]/(YZ-Y^3-Z^2)$.
\end{proof}

This completes the proof of Theorem \ref{thm:homog}.

Similarly, we can consider the direct limit $\dlim_d\mathrm{HF}^1(\phi^d)$, where the connecting map is also defined by multiplying the Seidel's element $[f^1]$. The exact same calculation for $\dlim_d R^d$, together with the cup product structure on $H^*(\Sigma_g^0)$, shows the following
\begin{theorem} \label{thm:HF^1 for nodal curves}
    $\dlim_d \mathrm{HF}^1(\phi^d)\cong \CC[Y,Z]/(YZ-Y^3-Z^2)\oplus \CC^{2g-2}$ as a $\dlim_d \mathrm{HF}^0(\phi^d)$ module, where $\dlim_d \mathrm{HF}^0(\phi^d)\cong \CC[Y,Z]/(YZ-Y^3-Z^2)$ acts on the first factor by multiplication, and on the second factor by projection to $\CC$ followed by diagonal multiplication.
\end{theorem}


Theorem \ref{thm:main_theorem} then follows from Theorems \ref{thm:ring_structure}, \ref{thm:HF^1 for nodal curves}, and \ref{thm:b_model}.

\subsection{The case of multiple Dehn twists}

Using the product formula we developed in \cite{ziwenyao}, we can also compute the direct limit in the case we are performing multiple Dehn twists simultaneously. 

 The same computation as in Section 6 shows the Seidel element is horizontal and transverse, so each Dehn twist $C_i$ contributes two elliptic Reeb orbits as in Theorem \ref{thm:s_calculation} to the Seidel class. After a choice of coherent orientations (see remark \ref{rmk:signs}) they all contribute with the same sign. With the same algebraic computation of the direct limit, we can show that
\begin{theorem}
Let $C_1,...,C_k$ be circles on $\Sigma_g$ such satisfying the conditions of Remark \ref{remark: topological assumption on Dehn twists}. Let $\phi$ denote the simultaneous Dehn twists around $\{C_1,..,C_k\}$. 
We have 
    \[
    \dlim_d \mathrm{HF}^0(\phi^d) = A \times _{\mathbb{C}} A \times_{\mathbb{C}} A \times... \times_{\mathbb{C}} A
    \]
    as algebras. Here $\times_{\mathbb{C}}$ denotes the fiber product of rings over their common map $A\rightarrow \mathbb{C}$, and there are $k$ copies of $A$ in the fiber product. We also have
    \[
    \dlim_d \mathrm{HF}^1(\phi^d) = (A \times _{\mathbb{C}} A \times_{\mathbb{C}} A \times... \times_{\mathbb{C}} A) \oplus \mathbb{C}^{2g-2} 
    \]
    as $\dlim_d \mathrm{HF}^0(\phi^d)$ modules, where the action of $\times^k _\mathbb{C}A$ on the second factor is the projection to $\CC$ followed by the diagonal multiplication.
\end{theorem}

The above result recovers the mirror statement, Theorem \ref{b side: multiple twists}.

\subsection{Dehn twists on punctured Riemann surfaces}

We next explain what the symplectic cohomology looks like for nodal surfaces with punctures. For simplicity we consider $\Sigma_{g,1}$ with only one puncture, $\phi$ is a Dehn twist around a curve $C$ that is non-separating, and has a puncture at $p$. 

The Seidel element in this setting is now $[e_0^1+e_1^1+u_0^1]$. Then the symplectic cohomology is given by

\begin{theorem}\label{thm: one puncture}
We have
    \[
        \dlim_d \mathrm{HF}^0(\Sigma_{g,1},\phi^d) \cong (\CC[Y,Z]/(YZ-Y^3-Z^2))\times_{\CC}\CC[T],
\]
where $T^i$ is given by the class $[u_i^1]$, and the fiber product is given by evaluations at 0.
We also have 
\[
        \dlim_d \mathrm{HF}^1(\Sigma_{g,1},\phi^d) \cong (\CC[W]\times_{\CC}\CC[T]) \oplus \CC^{2g-2}
\]
as a module over $\dlim_d \mathrm{HF}^0(\Sigma_{g,1},\phi^d)$, where the fiber product is given by evaluations $f(W)\mapsto f(0)-f(1)$ and $g(T)\mapsto g(0)$. The module structure on the first summand is given by $(Y,0)\mapsto (W-W^2,0)$, $(Z,0)\mapsto (W^2-W^3,0)$ followed by multiplications, and the module structure on the second summand is projection to $\CC$ followed by scalar multiplication. 
\end{theorem}

\begin{proof}
    We note that the proof for $\dlim_d \mathrm{HF}^0(\Sigma_{g,1},\phi^d)$ is very similar to the non-punctured case. The key is to realize the generators near the puncture and the generators in the Dehn twist region do not interact with each other save for the fact that both $u_0^1$ and $e_0^1+e^1_1$ contributes to the Seidel element $[f^1]=[e_0^1+e^1_1+u_0^1]$ with respect to which we take the direct limit. This is responsible for the fact we see the fiber product of $A$ with $\mathbb{C}[T]$ instead of the direct product. The main difficulty is working out the module structure of $\dlim_d \mathrm{HF}^1(\Sigma_{g,1},\phi^d)$ over $\dlim_d \mathrm{HF}^0(\Sigma_{g,1},\phi^d)$, which we describe in more detail. 
    
    The extra $2g-2$ generators of $\mathbb{C}^{2g-2}$ correspond to the fixed points away from the twist region and the punctured region; they have trivial products with the elements in $\dlim_d\mathrm{HF}^0(\Sigma_{g,1},\phi^d)$ except for the unit element (that correspond to the classes $[f^d]$), so we can ignore these in the following discussion. The main difference from the non-punctured case is that there is an extra generator $[\varphi^d] = [h_0^d-v_0^d]$. The images of $[\varphi^d]$ in the direct limit then correspond to the element $(W^d,-1)$ in the fiber product $\CC[W]\times_\CC \CC[T]$, the image of $[g^d]=[h_0^d+h_d^d]$ correspond to the elements $(1,0)$, while as the images of the classes $[v_i^d]$ correspond to the elements $(0,T^i)$ in the fiber product. The same proof as in Lemma \ref{lem: generators} shows that the module $\dlim_d \mathrm{HF}^1(\Sigma_{g,1},\phi^d)$ is generated by the images of the classes $[\varphi^d]$, $[g^1]$, $[v_i^d]$, $[h_1^2]$ and $[h_1^3]$. In fact, the last two classes $[h_1^2]$ and $[h_1^3]$ can be generated by $[\varphi^d]$ and $[g^1]$ as well. To see this, notice that 
    \[
    [h_1^2] = [e_0^1+e_1^1+u_0^1]\cdot [h_0^1-v_0^1]-[h_0^2-v_0^2]
    \]
    and 
    \[
    [h_1^3] = [e_0^1+e_1^1+u_0^1]\cdot [h_0^2-v_0^2]-[h_0^3-v_0^3].
    \]
    The images of the classes $[\varphi^d]$ and the image of the class $[g^1]$ are linearly independent over $\CC$. To see this, simply calculate the images of these class in the same degree $D$ by multiplying appropriate powers of $[f^1]$, and we get the following
    \[
    [g^D]=[h_0^D+h_D^D],\quad [\varphi^D] = [h_0^D-v_0^D],\quad [f^1]^{D-k}\cdot [\varphi^k] = [h_0^D+h_{D-k}^D-v_0^D]
    \]
    which are easily seen to be linearly independent over $\CC$ in $\mathrm{HF}^1(\Sigma_{g,1},\phi^D)$.
    
    The actions of $(Y,0)=[e_1^2]$ and $(Z,0)=[e_1^3]$ are $(Y,0)\cdot (W^d,-1) = (W^{d+1},-1)-(W^{d+2},-1)$ and $(Z,0)\cdot (W^d,-1) = (W^{d+2},-1)-(W^{d+3},-1)$, which come from the relations
    \[
    [e_1^2]\cdot [h_0^d-v_0^d] = [e_0^1+e_1^1+u_0^1]\cdot[h_0^{d+1}-v_0^{d+1}]-[h_0^{d+2}-v_0^{d+2}]
    \]
    and
    \[
    [e_1^3]\cdot [h_0^d-v_0^d] = [e_0^1+e_1^1+u_0^1]\cdot[h_0^{d+2}-v_0^{d+2}]-[h_0^{d+3}-v_0^{d+3}].
    \]

\end{proof}

We now present the case for multiple punctures.

\begin{theorem} \label{thm:A side multiple puncture}
    We have 
\[  \dlim_d \mathrm{HF}^0(\Sigma_{g,k},\phi^d) \cong (\CC[Y,Z]/(YZ-Y^3-Z^2))\times_{\CC}\CC[T_1] \times_{\CC} \CC[T_2] \times_{\CC} \cdots \times_{\CC}\CC[T_k],
\]
where the fiber product is taken by evaluations at $0$. For the module structure, we have
\[
\dlim_d \mathrm{HF}^1(\Sigma_{g,1},\phi^d) \cong (\CC[W]\times_{\CC}(\CC[T_1] \times \CC[T_2] \times \cdots \CC[T_k])) \oplus \CC^{2g-2}
\]
where the fiber product is given as follows. Let $f(W) \in \CC[W]$ and let $g_j (T_j) \in \CC[T_j]$, then we require
\[
f(0)-f(1) = g_1(0)+\cdots + g_k(0).
\]
The module action is given by the following. Let $(f(Y,Z),g_1(T_1),\cdots,g_k(T_k)) \in \dlim_d \mathrm{HF}^0(\Sigma_{g,k},\phi^d)$. The factor $Y$ acts by multiplication multiplication via $W-W^2$ on the $\CC[W]$ component. The factor $Z$ acts by multiplication via $W^2-W^3$ on the $\CC[W]$ factor. We extend this multiplicatively to obtain an action of $f(Y,Z)$. The element $g_j(T_j)$ acts by multiplication on $\CC[T_j]$. Finally $(f(Y,Z),g_1(T_1),\cdots,g_k(T_k)) \in \dlim_d \mathrm{HF}^0(\Sigma_{g,k},\phi^d)$ acts by multiplication by $f(0,0)$ in the $\CC^{2g-2}$ factor.
\end{theorem}
\begin{proof}
    The case with multiple punctures is almost the same as the case with only one puncture. For each degree $d$ and $j=1,2,\cdots, k$, we have fixed points $u_{i,j}^d$ and $v_{i,j}^d$ that correspond to the $k$ punctures. The Seidel element is $[f^1]=[e_0^1+e_1^1+u_{0,1}^1+\cdots, +u_{0,k}^1]$, and the elements $T_j^i\in \dlim_d \mathrm{HF}^0(\Sigma_{g,k},\phi^d)$ correspond to the classes $[u_{i,j}^d]$. To describe the module structure, the only difference from the once punctured case is that there are $k$ extra special cohomology classes $[\varphi_j^d]$ in each degree, corresponding to $[h_0^d-v_{j,0}^d]$. These classes are identified with $(W^d, 0,\cdots,0, -1, 0,\cdots,0)$ in our description of the module, where $-1$ is in the $j$-th summand $\CC[T_j]$. The fact that $(1, 0, \cdots, 0)$, $(W, -1, 0, \cdots, 0)$, $(W^2, -1, 0, \cdots, 0), \cdots, (W^D, -1, 0, \cdots, 0)$, $(W^D, 0, -1, 0,\cdots, 0),\cdots, (W^D, 0, \cdots, 0, -1)$ are linearly independent over $\CC$ is reflected by the fact that in $\mathrm{HF}^1(\Sigma_{g,k},\phi^D)$, the classes
    \[
    [g^D] = [h_0^D+h_D^D], [h_0^D+h_{D-1}^D-v_{0,1}^D], \cdots ,[h_0^D+h_1^D-v_{0,1}^D]
    ,\]
    \[
    [h_0^D-v_{0,1}^D],[h_0^D-v_{0,2}^D], \cdots ,[h_0^D-v_{0,k}^D]
    \]
    are linearly independent over $\CC$.
\end{proof}

After we take Spec on $\dlim_d \mathrm{HF}^0$ adding punctures to $\Sigma_g$ corresponds to a nodal degeneration of the mirror. Comparing with the B-side computation in Theorem \ref{thm:b_model_singular}, we recover the mirror statement in Theorem \ref{thm:main_theorem}.

\subsection{Homogeneous coordinate ring for $\phi^2$} \label{sec:phi_squared_homogenous}
In this subsection we compute the homogeneous coordinate ring $\bigoplus_{d\geq 0} \mathrm{HF}^0 (\phi^{2d})$ and its module action on $\bigoplus_{d\geq 0} \mathrm{HF}^1 (\phi^{2d})$.

\begin{theorem}
Let $\phi$ denote the Dehn twist along a non-separating circle $C\subset \Sigma_g$. Then we have
\[
\bigoplus_{d\geq 0} \mathrm{HF}^0 (\phi^{2d}) \cong (\CC[X,Y,Z]/(XYZ-Y^4-Z^2))\oplus \CC
\]
as a graded algebra, where $|X|=|Y|=2$ and $|Z|=4$, and the multiplication of the second $\CC$ factor with any element of positive degree is trivial. And 
\[
\bigoplus_{d\geq 0} \mathrm{HF}^1 (\phi^{2d}) \cong \CC[X,Y,Z]/(XYZ-Y^4-Z^2) \oplus (\CC[X])^{2g-2} \oplus \CC
\]
as a $\bigoplus_{d\geq 0} \mathrm{HF}^0 (\phi^{2d})$ module.
\end{theorem}

\begin{proof}
    The proof is more or less analogous to the proof of Theorem \ref{thmsum}. The elements $X=[e_0^0+e_0^2]$, $Y=[e_1^2]$ and $Z=[e_1^4]$, together with the fundamental class in $H^2(\Sigma_g')$ generate the ring $\bigoplus_{d\geq 0} \mathrm{HF}^0 (\phi^{2d})$, and satisfy the equation $XYZ=Y^4+Z^2$. The same proof as in lemma \ref{lemma:relation} shows that this equation generates all of the relations among $[e_0^0+e_0^2]$, $[e_1^2]$ and $ [e_1^4]$. The proof of module action of $\bigoplus_{d\geq 0} \mathrm{HF}^0 (\phi^{2d})$ also proceeds as before.
\end{proof}
Combined with the B-side computation in Proposition \ref{prop:B side calculation for phi^2} produces the mirror symmetry statement in Theorem \ref{thm:homog}.

\bibliography{biblio}{}
\bibliographystyle{amsalpha}

\end{document}